\documentclass[letterpaper]{amsart}
\usepackage[utf8]{inputenc}
\usepackage{graphicx} 

\usepackage{amsfonts}
\usepackage{mathtools}
\usepackage{amsmath, amsthm, amsopn, amssymb, microtype}
\usepackage{mathrsfs} 
\usepackage[all,cmtip]{xy}
\usepackage{enumerate, tikz, etoolbox, intcalc, caption, subcaption, afterpage}
\usepackage[english]{babel}
\usepackage{enumitem}
\usepackage{csquotes}

\usepackage[colorlinks=true,urlcolor=blue,pdfborder={0 0 0}]{hyperref}
\hypersetup{linkcolor=[rgb]{0.8,0,0}}
\hypersetup{citecolor=[rgb]{0,0.6,0}}
\usepackage[nameinlink]{cleveref}

\newtheorem{theorem}{Theorem}[section]
\newtheorem{definition}[theorem]{Definition}
\newtheorem{lemma}[theorem]{Lemma}
\newtheorem{corollary}[theorem]{Corollary}
\newtheorem{question}[theorem]{Question}
\newtheorem{example}[theorem]{Example}
\newtheorem{proposition}[theorem]{Proposition}
\newtheorem{remark}[theorem]{Remark}
\newtheorem{assumption}[theorem]{Assumption}
\newtheorem{conjecture}[theorem]{Conjecture}

\usepackage[left=30mm, top=25mm, right=30mm, bottom=25mm]{geometry} 

\title[The Yang--Baxter Equation and Characteristic Finite Simple Quotients of $F_2$]{The Yang--Baxter Equation and Characteristic Finite Simple Quotients of the Free Group of Rank $2$}
\author{Liam Hanany}
\thanks{The author was supported by the European Research Council (ERC) under the European Union’s Horizon 2020 (N. 882751)}
\date{\today}

\begin{document}

\begin{abstract}
We show that infinitely many alternating groups arise as quotients of the free group of rank 2, with kernel a characteristic subgroup. We also show that such simple quotients exist of arbitrarily large Lie rank. This resolves two questions posed by \cite{chen2025finite}.
\end{abstract}
\maketitle
\section{Introduction}
Let $F_n$ denote the free group on $n$ generators. In \cite{chen2025finite}, Chen, Lubotzky and Tiep proved a surprising fact about $F_2$. Namely, they showed that $F_2$ has, for all but finitely many prime powers $q$, finite simple quotients of type $\mathrm{PSL}_3(\mathbb{F}_q)$ and $\mathrm{PSU}_3(\mathbb{F}_q)$ with kernel a characteristic subgroup of $F_2$.

An equivalent formulation of this surprising fact is as follows - there is a pair $a,b$ of generators of these finite simple groups, such that any Nielsen-transformation on $a,b$ will yield a new pair of elements that is equivalent to $a,b$ up to an automorphism of the finite simple group. For such groups when $q$ is prime the outer automorphism group is small, of size at most $6$, and so a small index subgroup of $\mathrm{Aut}(F_2)$ will necessarily act on $a,b$ by conjugation. Another consequence is that every primitive element of $F_2$ will be mapped into a single automorphism-orbit of the finite simple group. 

This surprising result goes against the philosophy of the Wiegold conjecture. Specifically, the following is a conjecture that is a special case of the Wiegold conjecture, see \cite[Section 6]{lubotzky2011dynamics}.
\begin{conjecture}[The "Baby Wiegold Conjecture"]
    For every $n \geq 3$ there are no characteristic subgroups $C \leq F_n$ such that $F_n / C$ is a finite simple group.
\end{conjecture}
It has long been known that $F_2$ and $F_n, n\geq 3$ have different behaviours with regards to the Wiegold conjecture, see \cite{garion2009commutator}. The main reason for this difference is that under the $\mathrm{Aut}(F_2)$-action on $F_2$ the conjugacy class of the commutator $[x,y]$ is preserved up to inverse.
We remark that the finiteness of the simple quotient in the Baby Wiegold conjecture is crucial, as it was shown in \cite{coulon2023infinite} that there are infinitely many infinite simple characteristic quotients of $F_n$ for every $n \geq 2$.

The basic idea behind \cite{chen2025finite} is based on an observation of Dyer, Formanek and Grossman \cite{dyer1982linearity} that the braid group $B_4$ and the automorphism group $\mathrm{Aut}(F_2)$ are related. More precisely, it states that $B_4 / Z(B_4) \simeq \mathrm{Aut}^+(F_2)$, where $\mathrm{Aut}^+(F_2)$ is the index-$2$ subgroup of orientation preserving automorphisms, see Section \ref{The Equivariant Quandle} for further details.

The braid groups have a rich representation theory, see \cite{krammer2000braid},\cite{bigelow2001braid},\cite{krammer2002braid},\cite{abad2014introduction}. Given a Zariski-dense representation of a finitely generated group into a simple lie group, the strong approximation theorem of Weisfeiler and Pink \cite{weisfeiler1984strong},\cite{pink2000strong} gives infinitely many finite simple quotients of the same Lie type, obtained by specialisations of this representation to finite fields. The representation of $B_4$ used in \cite{chen2025finite} is the Burau representation, and their finite simple quotients were obtained as specializations of this representation.

Two natural questions were posed in \cite{chen2025finite}. The first is regarding the existence of such finite simple quotients of arbitrarily large Lie rank. The natural idea suggested there is to try and identify large rank Zariski-dense representations of the braid group $B_4$.

The second question offered was regarding the existence of alternating characteristic quotients of the free group $F_2$. In this paper we answer both questions positively.
\begin{theorem}
\label{Large rank quotients}
    For every $n \geq 3$ there are infinitely many prime powers $q$ such that there is a characteristic subgroup $C \leq F_2$ with $F_2 / C \simeq \mathrm{PSL}_{\binom{n}{2}}(\mathbb{F}_q)$.
\end{theorem}
Theorem \ref{Large rank quotients} is proved with the strategy suggested by \cite{chen2025finite}, by considering the representations of braid groups induced by representations of the quantum group $\mathcal{U}_q(\mathfrak{sl}_2)$, see \cite{jackson2011lawrence},\cite{kassel2012quantum},\cite{abad2014introduction}, and showing that they are Zariski dense, and then obtaining an $\mathrm{Aut}(F_2)$-action from the $B_4$-action.

These constructions also resemble the work of Funar and Lochak \cite{funar2018profinite} who gave constructions of characteristic simple quotients of surface groups of genus $g \geq 2$ using representations of mapping class groups arising from TQFT. These representation also have arbitrarily large rank. Theorem \ref{Large rank quotients} obtains such large rank Zariski-dense representations for the four-punctured disk. Moreover, Theorem \ref{Knizhinik Zamolodichikov Zariski density} gives such large rank Zariski-dense representations of $B_n$ for every $n \geq 3$.

The somewhat more surprising result is the following
\begin{theorem}
\label{Alternating quotients}
There are infinitely many natural numbers $n$ and characteristic subgroups $C \leq F_2$ with $F_2 / C \simeq A_n$, the alternating group on $n$ elements.
\end{theorem}
The novel idea here is to use solutions to the set-theoretic Yang--Baxter equation (and more specifically quandles) in order to get non-linear actions of $B_4$ on finite sets, and then identify orbits of these actions having the full alternating action.
The underlying set of the chosen quandle will be a group $\Gamma$, and the quandle operation will be equivariant with respect to the group action, see Section \ref{The Equivariant Quandle} for more details. We will use this quandle and the induced action of $B_4$ for the special case of the groups $\Gamma = \mathrm{PSL}_2(\mathbb{F}_p)$ and study the action in this case. This study resembles previous works \cite{gilman1977finite}, \cite{meiri2018markoff} proving that similar actions of groups on tuples of elements of $\Gamma$ have the full alternating or symmetric group as the induced permutation action, both works dealing specifically with the groups $\Gamma = \mathrm{PSL}_2(\mathbb{F}_p)$. The reason that these results are more accessible for such $\Gamma$ is the well-understood structure of maximal subgroups of $\Gamma$.

We conjecture the following generalization of Theorem \ref{Alternating quotients}.
\begin{conjecture}
    There is an integer $N$ such that for all $n \geq N$ the alternating group $A_n$ is a characteristic quotient of $F_2$.
\end{conjecture}
The smallest alternating characteristic quotient we have been able to obtain with our methods is $A_{25}$. The rate of growth of the family of $n$ given by Theorem \ref{Alternating quotients} can be bounded by a polynomial.

A somewhat more geometric way of describing our result was suggested in \cite[Remark 2.2.1]{chen2025finite}. Consider a once-punctured torus $S_{1,1}$, so $\pi_1(S_{1,1}) \simeq F_2$. An alternating quotient $A_n$ of $F_2$ then yields an $n$-cover of the punctured torus $S_{1,1}$, with the full alternating group $A_n$ as its Galois group. Since these quotients are characteristic, they are independent of the choice of identification $\pi_1(S_{1,1}) \simeq F_2$. In particular, these covers yield covers of punctured-torus bundles.
\begin{corollary}
\label{covering torus bundles}
Let $E \to B$ be a fibre bundle, with $B$ a CW-complex, and with fibres homeomorphic to the once-punctured torus $S_{1,1}$. Then $E$ has infinitely many finite covers, with Galois group the full alternating or symmetric group.
\end{corollary}
Another corollary of our results is that $\mathrm{Aut}(F_2)$ is fully residually finite almost-simple.
\begin{definition}
\begin{enumerate}
\item A group $\Gamma$ is called residually finite simple if for every $1 \neq \gamma \in \Gamma$ there is a finite simple quotient $\Gamma \to S$ such that $\gamma$ is mapped to a non-trivial element.
\item A group $\Gamma$ is called fully residually finite simple if for every finite subset $F \subseteq \Gamma$ there is a finite simple quotient $\Gamma \to S$ such that the restriction to $F$ is injective.
\item A finite group $A$ is called almost simple if there is a non-abelian finite simple group $S$ such that $S \leq A \leq \mathrm{Aut}(S)$
\item A group $\Gamma$ is called fully residually finite almost-simple if for every finite subset $F \subseteq \Gamma$ there is a finite almost-simple quotient $\Gamma \to A$ such that the restriction to $F$ is injective.

\end{enumerate}
\end{definition}
In \cite{gilman1977finite}, Gilman showed that the outer automorphism groups $\mathrm{Out}(F_n), n\geq 3$ are residually finite simple, and in fact residually finite alternating. It follows from his methods that these groups are fully residually finite simple. Note that $\mathrm{Out}(F_2) \simeq \mathrm{GL}_2(\mathbb{Z})$ has non-trivial center, and therefore is not residually finite simple.
\begin{theorem}
\label{Fully Residually Finite Simple}
$\mathrm{Aut}(F_2)$ is fully residually finite almost-simple
\end{theorem}
The idea of the proof is to use the fact that the Lawrence--Krammer representation of $B_4$ is faithful \cite{bigelow2001braid}, and the fact that this representation is known to be a special case of the representations induced from the quantum group $\mathcal{U}_q(\mathfrak{sl}_2)$ \cite{jackson2011lawrence}.
Finally, we show the following conditional result:
\begin{remark}
\label{AutFn not residually finite simple}
Assume the Baby Wiegold Conjecture. Then for every $n\geq 3$, $\mathrm{Aut}(F_n)$ is not residually finite almost-simple.
In fact we can prove more than this. We can show that any surjective homomorphism $\mathrm{Aut}(F_n) \to A$ to a finite almost-simple group $A$ has trivial image on the inner automorphisms $\mathrm{Inn}(F_n)$. Equivalently, every such map factors through $\mathrm{Out}(F_n)$. Using Gilman's result, the intersection of all kernels of maps to finite almost-simple groups is precisely $\mathrm{Inn}(F_n)$. See Section \ref{Section on residual finite almost-simplicity} for more details.
\end{remark}
\subsection{Structure of the paper}
In $\S2$, we remind the reader of the connection between $B_4$ and $\mathrm{Aut}(F_2)$, as well as the definition of a quandle and its induced Braid group actions. We then define the quandle that will be of interest throughout this paper, with base set a group $\Gamma$, which we call the equivariant quandle.
In $\S3$, we consider the action of $B_4$ induced by this quandle, and give some of its basic properties.
In $\S4$ we consider the case of $\Gamma = \mathrm{PSL}_2(\mathbb{F}_p)$, and deduce some further properties in this special case.
In $\S5$ we identify a specific orbit of the action and prove that the induced permutation action is primitive. In $\S6$ we use a criterion of Guralnick--Magaard \cite{guralnick1998minimal} on primitive permutation actions containing an element with a large number of fixed points in order to show that the induced permutation group is in fact the full alternating or symmetric group. We then use this to conclude Theorem \ref{Alternating quotients} and Corollary \ref{covering torus bundles}.
In $\S7$ we recall the definitions of the quantum group $\mathcal{U}_q(\mathfrak{sl}_2)$ and its representation theory, then use these to prove Theorem \ref{Large rank quotients}.
In $\S8$ we prove Theorem \ref{Fully Residually Finite Simple} and explain Remark \ref{AutFn not residually finite simple}.
$\S9$ is dedicated to some remarks on generalizations and ideas for future research.
In the Appendix we describe some further properties of the permutation action used in $\S6$ that were not needed in order to prove the main results of the paper, by giving it the structure of an algebraic variety, and the permutations of algebraic morphisms of this variety.

\subsection{Acknowledgements}
The author is extremely indebted to Alex Lubotzky, both for suggesting the problem, and for numerous highly beneficial discussions throughout the course of this work. The author would also like to thank Pavel Etingof and Mark Shusterman for a few beneficial discussions regarding the Yang--Baxter equation, and further thank the former for suggesting Remark \ref{equivariant quandle is a special case of a conjugation quandle}. Finally, the author would like to thank Francesco Fournier-Facio for helpful discussions on Wiegold's conjecture and other Wiegold-type problems.

The author would like to thank the anonymous referee for carefully reading the article and for valuable comments that significantly improved the clarity and exposition of the paper.

This research was conducted while the author was working at the Weizmann Institute of Science.

\section{The Equivariant Quandle}
\label{The Equivariant Quandle}
It was first observed by Dyer, Formanek and Grossman \cite{dyer1982linearity} that the braid group $B_4$ and the automorphism group $\mathrm{Aut}(F_2)$ are related, see also \cite{chen2025families} for a geometric description of this connection.
Recall that the braid group $B_n$ is given by the presentation with generators $\sigma_1,\dots\sigma_{n-1}$ and the relations
$$\sigma_i\sigma_{i+1}\sigma_i = \sigma_{i + 1}\sigma_i\sigma_{i + 1}$$
$$\sigma_i\sigma_j = \sigma_j\sigma_i, \  |i-j| \geq 2$$
and has infinite cyclic center, generated by $(\sigma_1 \cdots \sigma_{n-1})^n$.
The precise connection between $B_4$ and $\mathrm{Aut}(F_2)$ is given by an isomorphism $B_4 / Z(B_4) \simeq \mathrm{Aut}^+(F_2)$, where $\mathrm{Aut}^+(F_2) \leq \mathrm{Aut}(F_2)$ is the index $2$ subgroup acting with positive determinant on the abelianization $\mathbb{Z}^2$.
This isomorphism is given by the conjugation action of $B_4$ on the free subgroup $\langle \sigma_1\sigma_3^{-1}, \sigma_2 \sigma_1\sigma_3^{-1}\sigma_2^{-1}\rangle \simeq F_2 \triangleleft B_4$, and given by
$$\sigma_1:(x,y) \mapsto (x,yx^{-1})$$
$$\sigma_2:(x,y) \mapsto (y,yx^{-1}y)$$
$$\sigma_3:(x,y) \mapsto (x,x^{-1}y)$$

Our $\mathrm{Aut}(F_2)$ actions will be extensions of natural actions of $B_4$ on a finite set.
These actions will be obtained in a recursive manner via a solution to the set-theoretic Yang--Baxter equation, first introduced by Drinfeld in \cite{drinfeld2006some}.
\begin{definition}
Given a set $V$, a map $\varphi:V\times V\to V \times V$ is said to satisfy the set-theoretic Yang--Baxter equation if the following holds:
Let $\varphi_{12}, \varphi_{23}:V \times V \times V \to V \times V \times V$ be defined by
$$\varphi_{12}(x, y, z) = (\varphi(x, y), z)$$
$$\varphi_{23}(x, y, z) = (x, \varphi(y, z))$$
then
$$\varphi_{12}\circ\varphi_{23}\circ\varphi_{12}=\varphi_{23} \circ \varphi_{12} \circ \varphi_{23}.$$
\end{definition}
Note that this precisely means that the pair of maps $\varphi_{12}, \varphi_{23}$ satisfy the braid relation. In such a case, if $\varphi$ is bijective, we get an action of $B_n$ on $V^n$, the $i^\mathrm{th}$ braid generator acting via $\varphi$ on the $i,i+1$ coordinates of the tuple.
A special case of a solution to the Yang--Baxter equation is given by a quandle:
\begin{definition}
    A quandle structure on a set $V$ is a binary operation, denoted $a\lhd b$, satisfying the following three conditions:
    \begin{enumerate}
        \item $a \lhd a = a$
        \item $(a \lhd (b \lhd c)) = (a \lhd b) \lhd (a \lhd c)$
        \item For every $a \in V$, the map $V \to V$ given by $b \mapsto a \lhd b$ is bijective.
    \end{enumerate}
\end{definition}
Note that given a quandle structure on $V$ we may define a map $$\varphi:V\times V \to V\times V,\ \varphi(a,b) = (a \lhd b, a).$$ This map then satisfies the set-theoretic Yang--Baxter equation. Indeed:
$$\varphi_{12}(\varphi_{23}(\varphi_{12}((a, b, c)))) = \varphi_{12}(\varphi_{23}((a \lhd b, a, c))) = \varphi_{12}(a \lhd b, a \lhd c, a) = ((a \lhd b) \lhd (a \lhd c), a \lhd b, a)$$
$$\varphi_{23}(\varphi_{12}(\varphi_{23}((a, b, c)))) = \varphi_{23}(\varphi_{12}((a,b \lhd c, b))) = \varphi_{12}((a \lhd (b \lhd c), a, b)) = (a \lhd (b \lhd c), a \lhd b, a)$$
and these are equal by the quandle condition.
Moreover, condition (3) of the definition implies that the map $\varphi$ is bijective, therefore yielding an action of $B_n$ on $V^n$.

\begin{example}
The most standard example of a quandle is a conjugation quandle:
Given a group $\Gamma$ and a conjugacy class (or more generally a union of conjugacy classes) $C \subseteq \Gamma$, consider the conjugation action of $C$ on itself, $a \lhd b = aba^{-1}$. This gives $C$ a quandle structure.
\end{example}

We now define a specific quandle structure on every group which we find especially interesting for our purposes.
\begin{definition}
    Given a group $\Gamma$ define the equivariant quandle on $\Gamma$ by $a \lhd b = ab^{-1}a$
\end{definition}
For the remainder of this paper this will be the only quandle structure we will be interested in.
Note that this is indeed a quandle:
$$a \lhd (b \lhd c) = a \lhd (bc^{-1}b) = ab^{-1}cb^{-1}a$$
$$(a \lhd b) \lhd (a \lhd c) = (ab^{-1}a) \lhd (ac^{-1}a) = ab^{-1}a a^{-1}ca^{-1} ab^{-1}a = ab^{-1}cb^{-1}a$$
and also $a \lhd a = a$, and $b \mapsto ab^{-1}a$ is bijective.

The reason that this quandle structure is especially interesting is as it is equivariant under both the left and right action of $\Gamma$ on itself. That is, $(ga) \lhd (gb) = g(a \lhd b)$ and also $(ag) \lhd (bg) = (a \lhd b)g$ for every $a,b,g\in \Gamma$.

\begin{remark}
\label{equivariant quandle is a special case of a conjugation quandle}
    The equivariant quandle for $\Gamma$ is a special case of a conjugation quandle for $\Gamma \wr \mathbb{Z}/2$. This is given by the conjugation-invariant subset $C = \{((a, a^{-1}), \varepsilon) | a \in \Gamma\}$ and $\varepsilon$ the non-trivial element of $\mathbb{Z}/2$. This subset is naturally identified with $\Gamma$, and its conjugation action on itself is given by $$((a,a^{-1}),\varepsilon) ((b,b^{-1}),\varepsilon) ((a,a^{-1},\varepsilon))^{-1} = (a,a^{-1})(b^{-1},b)(a,a^{-1})\varepsilon = ((ab^{-1}a, a^{-1}ba^{-1}), \varepsilon).$$
    and so yields the equivariant quandle.
\end{remark}

\section{$B_4$-action}
The equivariant quandle structure allows us to give a natural action of $B_4$ on $\Gamma^4$ for every group $\Gamma$, in particular for finite ones. These could have been defined in a more elementary fashion via the following formulae
\begin{equation} \label{B4 action equations}
\begin{split}
\sigma_1:(a,b,c,d)\mapsto (ab^{-1}a, a, c, d) \\
\sigma_2:(a,b,c,d)\mapsto (a, bc^{-1}b, b, d) \\
\sigma_3:(a,b,c,d)\mapsto (a, b, cd^{-1}c, c)
\end{split}
\end{equation}
We state some nice properties of this action in the following proposition.
\begin{proposition}
\label{Basic Properties of the B4 Action}
    For every group $\Gamma$, there is an action of $B_4$ on $\Gamma^4$ given by (\ref{B4 action equations}). Denote an element of $\Gamma^4$ by $(a,b,c,d)$ This action satisfies the following:
    \begin{enumerate}
    \item The action is equivariant under both the right and left diagonal $\Gamma$-action on $\Gamma^4$.
    \item $\gamma(a,b,c,d) \coloneq ab^{-1}cd^{-1}$ and $\delta(a,b,c,d) \coloneq a^{-1}bc^{-1}d$ are preserved under the action (we will often abbreviate these as $\gamma,\delta$, omitting the dependence on $a,b,c,d$).
    \begin{enumerate}
        \item $\gamma$ is preserved under the right diagonal $\Gamma$-action and conjugated under the left diagonal $\Gamma$-action.
        \item $\delta$ is preserved under the left diagonal $\Gamma$-action and conjugated under the right diagonal $\Gamma$-action.
    \end{enumerate}
    \item The action commutes with automorphisms of $\Gamma$ acting simultaneously on all coordinates.
    \item The generator of the cyclic group $Z(B_4) \simeq \mathbb{Z}$ acts by $(a,b,c,d)\mapsto (\gamma a \delta^{-1}, \gamma b \delta^{-1}, \gamma c \delta^{-1}, \gamma d \delta^{-1})$.
    \item The involution $\varepsilon: (a,b,c,d) \mapsto (d, c, b, a)$ acts as an outer-automorphism on $B_4$, via $\varepsilon \sigma_i \varepsilon^{-1} = \sigma_{4 - i}^{-1}$, and also sends the preserved elements $\gamma,\delta$ to their inverses.
\end{enumerate}
\end{proposition}
\begin{proof}
    Parts 1,2,3 are easy calculations.

    For part 4, we compute the action of $\sigma_3\sigma_2\sigma_1$:
    $$(a, b, c, d) \overset{\sigma_1}{\mapsto} (ab^{-1}a, a, c, d) \overset{\sigma_2}{\mapsto} (ab^{-1}a, ac^{-1}a, a, d) \overset{\sigma_3}{\mapsto} (ab^{-1}a, ac^{-1}a, ad^{-1}a, a) = a (b^{-1}, c^{-1}, d^{-1}, a^{-1}) a.$$
    Applying $\sigma_3\sigma_2\sigma_1$ again we obtain
    $$(a,b,c,d)\overset{\sigma_3\sigma_2\sigma_1}{\mapsto}a (b^{-1}, c^{-1}, d^{-1}, a^{-1}) a\overset{\sigma_3\sigma_2\sigma_1}{\mapsto}$$
    $$ab^{-1}a a^{-1} (c, d, a, b) a^{-1} ab^{-1}a = ab^{-1} (c, d, a, b) b^{-1}a.$$
    Now, computing $ab^{-1}$ and $b^{-1}a$ for the result we get $ab^{-1}cd^{-1}ba^{-1}$ and $a^{-1}bd^{-1}cb^{-1}a$. So, applying $(\sigma_3\sigma_2\sigma_1)^2$ again we get
    $$ab^{-1}cd^{-1}ba^{-1} \cdot ab^{-1} \cdot (a, b, c, d) \cdot b^{-1}a \cdot a^{-1}bd^{-1}cb^{-1}a = \gamma \cdot (a,b,c,d) \cdot \delta^{-1}$$

    For part 5, $\varepsilon \sigma_1 \varepsilon^{-1} :(a, b, c, d) \mapsto (d, c, b, a) \mapsto (dc^{-1}d, d, b, a) \mapsto (a, b, d, dc^{-1}d)$ which is the same as $\sigma_3^{-1}$.
\end{proof}
Note in particular (from part 2 of the proposition) that given two fixed elements $\gamma,\delta\in\Gamma$, our action can be restricted to the subset of $\Gamma^4$ with these fixed values of $\gamma,\delta$.
We now want to achieve an action of $\mathrm{Aut}^{+}(F_2)$. For this, we will need to have the center of $B_4$ act trivially. We do this as follows - fix $\gamma,\delta$, and then divide by the left diagonal action of $C_\Gamma(\gamma)$ and by the right diagonal action of $C_\Gamma(\delta)$. By the equivariance of the action with respect to the diagonal actions, we know that the $B_4$-action is well defined on this quotient set. Moreover, both $\gamma,\delta$ are still well defined after the passage to this quotient, as they are conjugated by centralizing elements.
Note that in particular, this means that we have identified $(a,b,c,d) \sim \gamma (a, b, c, d) \delta^{-1}$, and so $Z(B_4)$ acts trivially on this quotient set. We summarize this as follows:
\begin{corollary}
For every group $\Gamma$ and pair of elements $\gamma,\delta\in\Gamma$, there is a well-defined action of $B_4$ on the quotient set $$X_{\gamma, \delta}(\Gamma) = \{(a,b,c,d) \in \Gamma^4 | ab^{-1}cd^{-1} = \gamma,\ a^{-1}bc^{-1}d = \delta\}/\sim$$(denote the set before dividing by the equivalence relation by $\tilde{X}_{\gamma, \delta}(\Gamma)$.) The equivalence relation $\sim$ is defined by $(a,b,c,d)\sim (\tilde{a}, \tilde{b}, \tilde{c}, \tilde{d})$ if there exist $\hat{\gamma}\in C_\Gamma(\gamma), \hat{\delta} \in C_\Gamma(\delta)$ such that $(\tilde{a},\tilde{b},\tilde{c},\tilde{d})=\hat{\gamma}(a,b,c,d)\hat{\delta}$ i.e. $\tilde{a} = \hat{\gamma} a \hat{\delta}, \tilde{b} = \hat{\gamma} b \hat{\delta}, \tilde{c} = \hat{\gamma} c \hat{\delta}, \tilde{d} = \hat{\gamma} d \hat{\delta}.$
Under this quotient action, $Z(B_4)$ acts trivially.
\end{corollary}
Whenever the group $\Gamma$ is clear from the context, we will use the notation $X_{\gamma,\delta}$ instead of $X_{\gamma,\delta}(\Gamma)$ and $\tilde{X}_{\gamma,\delta}$ instead of $\tilde{X}_{\gamma, \delta}(\Gamma).$

The involution $\varepsilon$ from the proposition will allow us later on to extend our action from an $\mathrm{Aut}^{+}(F_2)$-action to an $\mathrm{Aut}(F_2)$-action, given certain conditions on $\gamma,\delta \in \Gamma$. Moreover, note that the automorphism $\sigma_1 \mapsto \sigma_3, \sigma_2 \mapsto \sigma_2, \sigma_3 \mapsto \sigma_1$ of $B_4$ is inner, given by conjugation by $\Delta_4 = \sigma_1\sigma_2\sigma_3 \cdot \sigma_1\sigma_2 \cdot \sigma_1$. So, conjugating by $\Delta_4\cdot\varepsilon$ acts by the outer automorphism $\sigma_i \mapsto \sigma_i^{-1}$ of $B_4$, see Lemma \ref{semi-direct product AutF2} below.

We give some extra motivation as to the specific interest in the equivariant quandle in the following remark:
\begin{remark}
    The quandle structure we have described gives rise to a $B_4$-action on $\Gamma^4$ for every group $\Gamma$, in particular for the free group $F_4$. For the free group, we get an action on $F_4^4$ preserving the subgroup generated by the tuple, which may be thought of as a map $B_4 \to \mathrm{Aut}(F_4)$.
    
   For the special case of a conjugation quandle, the resulting map $B_4 \to \mathrm{Aut}(F_4)$ is precisely the Artin representation of the braid group, induced by the geometric action of $B_4$ as the mapping class group of a $4$-punctured disk, acting on its fundamental group.
    
    For the equivariant quandle, we get a map $B_4 \to \mathrm{Aut}(F_4)$ with some extra nice properties. This action is equivariant under both left and right diagonal multiplication. Using this fact, we may divide by the diagonal right action, and get an action on $\Gamma^3$ for every group $\Gamma$, simply assuming the first coordinate is trivial using the diagonal action. This induces a map $B_4 \to \mathrm{Aut}(F_3)$. After dividing by the right diagonal action the element $\gamma$ is still preserved by the $B_4$-action, and also by the right diagonal action. (Note that the element $\delta$ is no longer well-defined, only its conjugacy class is well-defined). Moreover this element $\gamma$ is primitive in $F_3$. Hence, considering the action on $F_2 \simeq F_3/\langle\gamma\rangle$ (which we may think of as assuming $\gamma=1$) we get an action on $\Gamma^2$ for every group $\Gamma$, and as before this induces a map $B_4 \to \mathrm{Aut}(F_2)$. This reproduces the Dyer--Formanek--Grossman isomorphism, sending isomorphically $B_4 / Z(B_4) \simeq \mathrm{Aut}^{+}(F_2)$.
    
    Moreover, a similar procedure yields  more generally maps $B_{2n} \to \mathrm{Aut}(F_{2n-2})$, also constructed in \cite{kassel2013action} by more topological means.
\end{remark}
It also follows from the previous remark that our study is related to the previous work of Meiri and Puder \cite{meiri2018markoff}:
\begin{remark}
    The action of $B_4$ on $X_{1,\delta}$ is precisely the action of $\mathrm{Aut}^+(F_2)$ on those pairs of elements $(x,y) \in \Gamma^2$ such that $[x,y]=\delta$, up to simultaneous conjugation by an element of $C_\Gamma(\delta)$. Equivalently stated, this is the action on the set of pairs of elements of $\Gamma$, up to simultaneous conjugation by an element of $\Gamma$, with commutator in the conjugacy class of $\delta$. This action factors through $\mathrm{Out}^+(F_2)$, as the inner automorphisms act by conjugation. Such actions were studied for a finite simple group $\Gamma$ by Garion and Shalev \cite{garion2009commutator}, see also \cite{lubotzky2011dynamics}. They were further studied by Meiri and Puder \cite{meiri2018markoff} more specifically for $\Gamma = \mathrm{PSL}_2(\mathbb{F}_p)$, who also showed that a related permutation action is the full alternating or symmetric group, when $\delta$ is unipotent.
\end{remark}
We end this section with the following lemma on braid groups which will be useful later
\begin{lemma}
\label{semi-direct product AutF2}
\begin{enumerate}
\item The map $\varphi: B_4 \to B_4$ defined on generators by $\varphi(\sigma_i) = \sigma_i^{-1}$ is an automorphism. Moreover, restricting the automorphism to $B_4/Z(B_4)$, we get an isomorphism $\mathrm{Aut}(F_2) \simeq B_4/Z(B_4) \rtimes_\varphi \mathbb{Z}/2$.
\item For every $n \geq 2$, the map $\psi: B_n \to B_n$ defined on generators by $\psi(\sigma_i) = \sigma_{n-i}$ is an automorphism. Moreover, this automorphism is inner, and given by conjugation by the Garside element $\Delta_n = (\sigma_1 \cdots \sigma_{n-1}) \cdot (\sigma_1 \cdots \sigma_{n-2}) \cdots (\sigma_1\sigma_2) \cdot \sigma_1.$
\end{enumerate}
\end{lemma}
\begin{proof}
\begin{enumerate}
    \item It is easy to see that the map $\varphi$ preserves the braid relations, and is therefore an automorphism.
    Define $x = \sigma_1\sigma_3^{-1}, y = \sigma_2\sigma_1\sigma_3^{-1}\sigma_2^{-1}$. Consider the action of $B_4$ on $F_2 \simeq \langle x, y\rangle.$ It is given by $\varphi(x) = x^{-1}, \varphi(y) = \sigma_2^{-1} \sigma_1^{-1} \sigma_3 \sigma_2$. Therefore,
    $$x^{-1} y x^{-1} = \sigma_1^{-1} \sigma_3 \cdot \sigma_2 \sigma_1 \sigma_3^{-1} \sigma_2^{-1} \cdot \sigma_1^{-1} \sigma_3 = \sigma_3 \sigma_1^{-1} \cdot \sigma_2 \sigma_1 \sigma_3^{-1} \sigma_2^{-1 } \cdot \sigma_3 \sigma_1^{-1} =$$
    $$\sigma_3 \sigma_1^{-1} \cdot \sigma_2 \sigma_1 \sigma_2 \cdot \sigma_2^{-1}\sigma_3^{-1} \sigma_2^{-1 } \cdot \sigma_3 \sigma_1^{-1} = \sigma_3 \sigma_1^{-1} \cdot \sigma_1 \sigma_2 \sigma_1 \cdot \sigma_3^{-1}\sigma_2^{-1} \sigma_3^{-1 } \cdot \sigma_3 \sigma_1^{-1} =$$
    $$\sigma_3 \sigma_2 \cdot \sigma_1 \sigma_3^{-1} \cdot \sigma_2^{-1} \sigma_1^{-1} = \sigma_3 \sigma_2 \cdot \sigma_3^{-1} \sigma_1 \cdot \sigma_2^{-1} \sigma_1^{-1} = \sigma_2^{-1} \cdot \sigma_2 \sigma_3 \sigma_2 \cdot \sigma_3^{-1} \sigma_1 \cdot \sigma_2^{-1} \sigma_1^{-1} \sigma_2^{-1} \cdot \sigma_2 = $$
    $$\sigma_2^{-1} \cdot \sigma_3 \sigma_2 \sigma_3 \cdot \sigma_3^{-1} \sigma_1 \cdot \sigma_1^{-1} \sigma_2^{-1} \sigma_1^{-1} \cdot \sigma_2 = \sigma_2^{-1} \sigma_3 \sigma_1^{-1} \sigma_2 = \varphi(y).$$
So $\varphi$ is an automorphism, $\varphi^2 = 1$, and $\varphi \notin \mathrm{Aut}^+(F_2)$, as $\varphi$ acts on $\mathbb{Z}^2$ as $\begin{pmatrix}-1 & -2 \\ 0 & 1\end{pmatrix}$ which has determinant $-1$. Therefore $\mathrm{Aut}(F_2)$ is generated by its normal index-$2$ subgroup $\mathrm{Aut}^+(F_2)$ and the disjoint subgroup $\langle\varphi\rangle$ of order $2$. Moreover, the conjugation action of $\varphi$ on $\mathrm{Aut}^+(F_2)$ is given by $\varphi \sigma_i \varphi^{-1} = \sigma_i^{-1}$, yielding the required semi-direct product decomposition.

\item This is a well-known fact on braid groups, see for instance \cite[Proposition 9.1.16, Equation 9.1.13]{epstein1992word}. It follows from the fact that the braid $\Delta_n$ corresponds to the mapping class of the $n$-punctured disk which is a $180^\circ$-rotation flipping the order of the $n$ punctures of the disk.
\end{enumerate}
\end{proof}
\section{The case of $\mathrm{PSL}_2(\mathbb{F}_p)$}
We consider in greater detail the case of $\Gamma = \mathrm{PSL}_2(\mathbb{F}_p)$, for $p\geq 3$. Our $\mathrm{Aut}(F_2)$-actions will be the induced actions on specific orbits for particular choices of $\gamma,\delta \in \mathrm{PSL}_2(\mathbb{F}_p)$. 
To represent the fact the we have moved to work with matrices, we now denote points in $X_{\gamma,\delta}$ by $(A,B,C,D)$ instead of $(a,b,c,d)$.

\begin{remark}
In this case, we wish to informally think of the group $\Gamma$ as a $3$-dimensional algebraic variety over $\mathbb{F}_p$ . Fixing the values of $\gamma,\delta$ we get a $6$-dimensional subvariety of $\Gamma^4$, as there are $12$ degrees of freedom in the choice of $4$ elements of $\Gamma$, and two conditions coming from the particular choice of $\gamma,\delta$. Moreover, in $\Gamma$, the centralizer of every non-trivial element is a $1$-dimensional subvariety. Therefore, dividing by the left and right diagonal actions of $C_\Gamma(\gamma),C_\Gamma(\delta)$ we get a $4$-dimensional variety. We will make this informal discussion formal in the Appendix.
\end{remark}
We will make the following two slight modifications on the definitions from the previous section for the specific case of $\mathrm{PSL}_2(\mathbb{F}_p)$. For a matrix $A \in \mathrm{SL}_2(\mathbb{F}_p)$ denote by $\bar{A}$ the corresponding element of $\mathrm{PSL}_2(\mathbb{F}_p)$.
\begin{definition}
Fix two elements $\gamma,\delta \in \mathrm{SL}_2(\mathbb{F}_p)$. Denote
$$\tilde{X}_{\gamma,\delta}^{(2)} = \{(\bar{A},\bar{B},\bar{C},\bar{D})| A,B,C,D \in \mathrm{SL}_2(\mathbb{F}_p), AB^{-1}CD^{-1} = \gamma, A^{-1}BC^{-1}D = \delta\} \subseteq \mathrm{PSL}_2(\mathbb{F}_p)^4.$$
$$X_{\gamma,\delta}^{(2)} = \tilde{X}_{\gamma, \delta}^{(2)} / \sim$$
where the equivalence relation is defined by $(A,B,C,D) \sim (\tilde{A}, \tilde{B}, \tilde{C}, \tilde{D})$ if there are $\hat{\gamma} \in C_{\mathrm{PGL}_2(\mathbb{F}_p)}(\bar{\gamma}),\hat{\delta} \in C_{\mathrm{PGL}_2(\mathbb{F}_p)}(\bar{\delta})$ such that $(A,B,C,D) = \hat{\gamma}(\tilde{A},\tilde{B}, \tilde{C},\tilde{D})\hat{\delta}$.
\end{definition}

Thinking of $\gamma, \delta \in \mathrm{SL}_2(\mathbb{F}_p)$, we note that the $B_4$-action on $\tilde{X}_{\gamma,\delta}(\mathrm{SL}_2(\mathbb{F}_p))$ induces an action of $B_4$ on $\tilde{X}_{\gamma,\delta}^{(2)}$. Note that $\mathrm{PSL}_2(\mathbb{F}_p) \leq \mathrm{PGL_2}(\mathbb{F}_p)$ is exactly the embedding $\Gamma \simeq \mathrm{Inn}(\Gamma) \leq \mathrm{Aut}(\Gamma)$, and is of index $2.$ Therefore by part $3$ of Proposition \ref{Basic Properties of the B4 Action} the $B_4$-action on $\tilde{X}_{\gamma,\delta}^{(2)}$ descends to a quotient action of $B_4$ on $X_{\gamma,\delta}^{(2)}$.

Moreover, the set $X_{\gamma, \delta}^{(2)}$ depends on the lifts $\gamma, \delta \in \mathrm{SL}_2(\mathbb{F}_p)$ of given elements $\bar{\gamma},\bar{\delta} \in \mathrm{PSL}_2(\mathbb{F}_p)$ only up to a common sign. Indeed, changing the signs of $A,B,C,D$ in an arbitrary way will change the signs of $AB^{-1}CD^{-1},A^{-1}BC^{-1}D$ in the same way. In general, $X_{\gamma, \delta}^{(2)} \neq X_{-\gamma,\delta}^{(2)}$, and in fact $\tilde{X}_{\gamma,\delta}^{(2)} \cap \tilde{X}^{(2)}_{-\gamma, \delta} = \varnothing$. However, for ease of notation, we will often write $\gamma,\delta$ to mean $\bar{\gamma},\bar{\delta} \in \mathrm{PSL}_2$, unless this distinction is important.

We now state a basic fact on the action of a generator of the braid group.
\begin{lemma}
\label{Action of single generator}
\begin{enumerate}
\item For every non-trivial, non-involution $\gamma \in \Gamma = \mathrm{PSL}_2(\mathbb{F}_p)$, the centralizer $C_\Gamma(\gamma)$ is cyclic.
\item Consider the action of a generator on pairs of elements $(A,B) \mapsto (AB^{-1}A,A)$.
    \begin{enumerate}
    \item This action preserves both $AB^{-1},A^{-1}B$.
    \item If $AB^{-1}$ generates $C_\Gamma(AB^{-1})$, then the action is transitive on those pairs with fixed values of $AB^{-1}, A^{-1}B$
    \end{enumerate}
\end{enumerate}
\end{lemma}
Note that $AB^{-1},A^{-1}B$ are inverse-conjugate in $\Gamma$, and therefore one of them generates its centralizer if and only if the other one does.

Recall before the proof that elements $\gamma \in \Gamma$ are of one of three types - split, non-split or unipotent. The first two types meaning that $\gamma$ is diagonalizable, with eigenvalues in $\mathbb{F}_p$ or $\mathbb{F}_{p^2}$ respectively, and the third type being a non-trivial Jordan block with eigenvalues $\pm 1$.
\begin{proof}
\begin{enumerate}
    \item If $\gamma$ is diagonalisable, we conjugate $\gamma$ over $\mathbb{F}_{p^2}$ to the form $D_\lambda = \begin{pmatrix}\lambda & 0 \\ 0 & \lambda^{-1}\end{pmatrix}$. If $\gamma \neq 1$ in $\mathrm{PSL}_2(\mathbb{F}_p)$ then $\lambda \neq \pm 1$. Assume that $AD_\lambda = \pm D_\lambda A$, for $A = \begin{pmatrix}a & b \\ c & d\end{pmatrix}$, i.e. $$\begin{pmatrix} \lambda a & \lambda^{-1}b \\ \lambda c & \lambda^{-1}d \end{pmatrix} = \pm\begin{pmatrix}\lambda a & \lambda b \\ \lambda^{-1}c & \lambda^{-1} d\end{pmatrix}.$$ Since $\gamma$ is not an involution, $\lambda \neq \pm \lambda^{-1}$ and it follows that $b=c=0$ and therefore $C_\Gamma(\gamma)$ is a sub-quotient of $\mathbb{F}_{p^2}^\times$ and hence cyclic.
    
    Otherwise, $\gamma$ is unipotent, and we may conjugate it (also over $\mathbb{F}_{p^2})$ to the form $u = \begin{pmatrix}1 & 1 \\ 0 & 1\end{pmatrix}$. Therefore, if $Au = \pm uA$ for $A = \begin{pmatrix}a & b \\ c & d\end{pmatrix}$, then $$\begin{pmatrix} a & a + b \\ c & c + d\end{pmatrix} = \pm\begin{pmatrix}a + c & b + d \\ c & d\end{pmatrix}.$$ If $c \neq 0$, then the sign is $1$ and $a = a + c,$ in contradiction. So $c = 0$, meaning $ad = 1$ so both $a$ and $d$ aren't zero. Therefore the sign is $1$ and $a + b = b + d$, and so $a=d=\pm 1$, and $C_\Gamma(\gamma) \simeq \mathbb{F}_p^+$.

    \item Part (a) is an easy computation. For part (b), note that the action simply multiplies both $A,B$ by $AB^{-1}$ on the left. Fix the values $X,Y$ such that $X = AB^{-1},Y = A^{-1}B$, and note that $AY^{-1}A^{-1} = X$, so $A$ is an element conjugating $X$ to $Y^{-1}$.

    Therefore, any other element $A$ conjugating $X$ to $Y^{-1}$ will be of the form $\hat{X} A$, for $\hat{X} \in C_{\Gamma}(X)$. Assuming that $X$ generates $C_\Gamma(X)$, we get $\hat{X} = X^n$, and therefore applying the transformation $n$ times we move from $(A,B)$ to $(\hat{X}A,\hat{X}B)$, and this is an arbitrary element with the given values $X,Y$.
\end{enumerate}
\end{proof}

The lemma makes the following definition very natural in our context:
\begin{definition}
\label{maximal definition}    
Call an element $\gamma \in \Gamma$ maximal if $\gamma$ generates $C_\Gamma(\gamma)$.
\end{definition}
An important fact that will be used in the sequel is that all unipotents are maximal.

Note also that it follows from the proof of the previous lemma that we may informally think of this centralizer as the $\mathbb{F}_p$-points of a $1$-dimensional algebraic subvariety of $\mathrm{PSL}_2$.

\subsection{Foliations}
We can now define a certain family of foliations of our $4$-dimensional variety that will greatly help us in its study.
These foliations will serve intuitively as "Zariski-closures" of orbits of certain copies of subgroups $\mathbb{Z}^2 \leq B_4$ (This is not technically precise as the orbits are finite). These copies of $\mathbb{Z}^2$ will give $2$-dimensional foliations.

The first foliation is associated to the subgroup $\mathbb{Z}^2 \simeq \langle \sigma_1, \sigma_3\rangle \leq B_4$, and is defined by those quadruples with fixed values of $AB^{-1},A^{-1}B,CD^{-1},C^{-1}D$. This is a $2$-dimensional subvariety, as a given value for both of the elements $AB^{-1}, A^{-1}B$ still leaves one degree of freedom by the previous lemma. Note moreover that by the lemma, if both $AB^{-1},CD^{-1}$ are maximal, every two points on the corresponding leaf of the foliation are on the same $B_4$-orbit.

The second foliation chosen is associated to the subgroup $\mathbb{Z}^2 \simeq \langle \sigma_2^{-1} \sigma_1 \sigma_2, \sigma_2^{-1} \sigma_3 \sigma_2 \rangle \leq B_4$. It can be checked that this subgroup leaves the following elements invariant: $AC^{-1}, A^{-1}C, CB^{-1}CD^{-1}, C^{-1}BC^{-1}D.$
Once again - if $AC^{-1}, CB^{-1}CD^{-1}$ are both maximal, then any two points of this $2$-dimensional subvariety lie on the same $B_4$-orbit.

We can in fact get an infinite sequence of foliations by taking further conjugations of this $\mathbb{Z}^2$ subgroup by powers of $\sigma_2.$ However, only the first foliation will be used in our proof. We mention as a special case the conjugation by $\sigma_2$ which preserves $BD^{-1}, B^{-1}D, AB^{-1}CB^{-1}, A^{-1}BC^{-1}B$.

These foliations give rise to the following definition
\begin{definition}
\label{proper decomposition definition}
    A proper decomposition of a pair of elements $\gamma, \delta \in \mathrm{SL}_2(\mathbb{F}_p)$ is a quadruple of elements $x, y, z, w \in \mathrm{SL}_2(\mathbb{F}_p)$, such that
    \begin{enumerate}
        \item $x,y$ are conjugate in $\mathrm{SL}_2(\mathbb{F}_p)$.
        \item $z,w$ are conjugate in $\mathrm{SL}_2(\mathbb{F}_p)$.
        \item $xz = \gamma$
        \item $wy = \delta^{-1}$
    \end{enumerate}
Moreover, call such a proper decomposition maximal if $x,y,z,w$ are all maximal as elements of $\mathrm{PSL}_2(\mathbb{F}_p)$.
\end{definition}
Given a point $(\bar{A},\bar{B},\bar{C},\bar{D}) \in \tilde{X}^{(2)}_{\gamma, \delta}$, We can build a proper decomposition corresponding to each of the two foliations. For the first:
$$x = AB^{-1}, y = B^{-1}A, z = CD^{-1}, w = D^{-1}C$$
for some choice of lifts $A,B,C,D$ of $\bar{A},\bar{B},\bar{C},\bar{D}$ to $\mathrm{SL}_2(\mathbb{F}_p),$ such that $AB^{-1}CD^{-1} = \gamma, A^{-1}BC^{-1}D = \delta.$
And indeed $xz = AB^{-1}CD^{-1} = \gamma$ and $wy = D^{-1}CB^{-1}A = (A^{-1}BC^{-1}D)^{-1} = \delta^{-1}$.
For the second foliation, we may take
$$x = AC^{-1}, y = C^{-1}A, z = CB^{-1}CD^{-1}, w = D^{-1}CB^{-1}C.$$
Moreover, we already know that the $\mathbb{Z}^2$-subgroup corresponding to each foliation preserves this chosen proper decomposition, and that if the proper decomposition is maximal, then all points on the subvariety are on the same $B_4$-orbit (in fact even the same $\mathbb{Z}^2$-orbit).

Note, however, that a point in $X^{(2)}_{\gamma,\delta}$ is only defined up to an equivalence relation - and therefore we need the following equivalence relation on proper decompositions
\begin{definition}
    Say that two proper decompositions are equivalent if $x,z$ thought of as elements of $\mathrm{PSL}_2(\mathbb{F}_p)$ are equivalent up to a conjugation by an element of $C_{\mathrm{PGL}_2(\mathbb{F}_p)}(\gamma)$, and $y,w$ thought of as elements of $\mathrm{PSL}_2(\mathbb{F}_p)$ are equivalent up to a conjugation by an element of $C_{\mathrm{PGL}_2(\mathbb{F}_p)}(\delta)$ - with the further assumption that the conjugating elements have the same determinant.

    In formulae:
    $(x,y,z,w)\sim(\tilde{x},\tilde{y},\tilde{z},\tilde{w})$ if there exist $\tilde{\gamma}\in C_{\mathrm{PGL}_2(\mathbb{F}_p)}(\gamma),\tilde{\delta}\in C_{\mathrm{PGL}_2(\mathbb{F}_p)}(\delta)$ such that $x=\tilde{\gamma}\tilde{x}\tilde{\gamma}^{-1},z=\tilde{\gamma}\tilde{z}\tilde{\gamma}^{-1}$ and $y=\tilde{\delta}\tilde{y}\tilde{\delta}^{-1}, w=\tilde{\delta}\tilde{w}\tilde{\delta}^{-1}$ and $\det(\tilde{\gamma}) = \det(\tilde{\delta})$. Here the element equalities are as elements of $\mathrm{PSL}_2(\mathbb{F}_p).$
\end{definition}
Here we think of the determinant as a map $\mathrm{PGL}_2(\mathbb{F}_p) \to \{\pm 1\}$, given by the Legendre symbol of the determinant, with kernel $\mathrm{PSL}_2(\mathbb{F}_p).$

Note that two different representations of a given point in $X^{(2)}_{\gamma,\delta}$ yield possibly distinct, but equivalent proper decompositions (under both foliations). Furthermore, given a point and its corresponding proper decomposition, any equivalent proper decomposition is given by a different representation of this same point.
\begin{corollary}
\label{Moving inside first foliation}
    Let $Q \in X^{(2)}_{\gamma,\delta}$ be a point with a maximal first proper decomposition. Then every other point $R \in X^{(2)}_{\gamma,\delta}$ with equivalent first proper decomposition to $Q$ is on the same $\langle\sigma_1,\sigma_3\rangle$-orbit as $Q$.
\end{corollary}
\begin{proof}
    Let $R = (A_R, B_R, C_R, D_R) \in X^{(2)}_{\gamma, \delta}$ be a point with an equivalent first proper decomposition to $Q = (A_Q, B_Q, C_Q, D_Q)$. Choose $\tilde{\gamma} \in C_{\mathrm{PGL}_2(\mathbb{F}_p)}(\gamma),\tilde{\delta} \in C_{\mathrm{PGL}_2(\mathbb{F}_p)}(\delta)$ such that $$A_QB_Q^{-1}=\tilde{\gamma}A_RB_R^{-1}\tilde{\gamma}^{-1},C_QD_Q^{-1}=\tilde{\gamma}C_RD_R^{-1}\tilde{\gamma}^{-1}$$ and $$B_Q^{-1}A_Q=\tilde{\delta}B_R^{-1}A_R\tilde{\delta}^{-1}, D_Q^{-1}C_Q=\tilde{\delta}D_R^{-1}C_R\tilde{\delta}^{-1}$$ and $\det(\tilde{\gamma}) = \det(\tilde{\delta})$. Consider $\tilde{R} = \tilde{\gamma}(A_R, B_R, C_R, D_R)\tilde{\delta}^{-1}$ which is equivalent to $R$ in $X^{(2)}_{\gamma,\delta}$. Note that since $\det(\tilde{\gamma})=\det(\tilde{\delta})$, these four matrices are still in $\mathrm{PSL}_2(\mathbb{F}_p)$.
    Now, $Q,\tilde{R}$ have the same first proper decompositions. In particular, $A_QB_Q^{-1} = A_RB_R^{-1}$ and $B_Q^{-1}A_Q = B_R^{-1}A_R$. By Lemma \ref{Action of single generator}, since $A_QB_Q^{-1}$ is maximal, we can move from $(A_Q,B_Q)$ to $(A_R, B_R)$ by a power of $\sigma_1$. Similarly for $(C_Q, D_Q)$ and $(C_R, D_R)$ by a power of $\sigma_3$.
\end{proof}

\subsection{Matrices up to conjugation}
We will see in the Appendix that understanding triplets in $\Gamma$ up to conjugation will help us understand the algebraic variety in question. Therefore, we are interested in representation varieties (sometimes also called character varieties) of free groups. We will state the necessary well-known results in the following lemma, see for instance \cite{magnus1980rings}.
\begin{lemma}
\label{Magnus character varieties}
\begin{enumerate}
    \item Two matrices $ M,N\in\mathrm{SL}_2(\mathbb{F}_p)$ such that $M, N \neq \pm 1$ are $\mathrm{GL}_2(\mathbb{F}_p)$-conjugate if and only if they have the same trace.
    \item Two matrices $M,N\in\mathrm{PSL}_2(\mathbb{F}_p)$ such that $M,N \neq 1$ are $\mathrm{PGL}_2(\mathbb{F}_p)$-conjugate if and only if they have the same trace up to sign.
    \item The following two equations on traces are true for every $M,N\in\mathrm{SL}_2(\mathbb{F}_p)$, $\mathrm{tr}(M)=\mathrm{tr}(M^{-1})$ and $\mathrm{tr}(MN) + \mathrm{tr}(MN^{-1}) = \mathrm{tr}(M)\mathrm{tr}(N)$.
    \item In $\mathrm{SL}_2(\mathbb{F}_p)$, two tuples of $n$ elements $(M_1,\dots,M_n),(N_1,\dots,N_n)$ yield the same trace function as maps $F_n \to \mathrm{SL}_2(\mathbb{F}_p)$, that is $\mathrm{tr}(w(M_1,\dots,M_n))=\mathrm{tr}(w(N_1,\dots,N_n))$ for every word $w\in F_n$, if and only if every one of the $2^n$ products over subsets of $[n]$ in one tuple have the same trace as the corresponding product of a subset from the other tuple, where the product is taken in increasing order.
    \item For a triple of $\mathrm{SL}_2(\mathbb{F}_p)$ elements $(M_1,M_2,M_3)$ denote
    $$a=\mathrm{tr}(M_1),b=\mathrm{tr}(M_2),c=\mathrm{tr}(M_3),x=\mathrm{tr}(M_2M_3),y=\mathrm{tr}(M_3M_1),z=\mathrm{tr}(M_1M_2),p=\mathrm{tr}(M_1M_2M_3).$$
    Then these satisfy the equation
    $$p^2 - (ax+by+cz - abc)p + (a^2 + b^2 + c^2 + x^2 + y^2 + z^2 + xyz - abz - bcx - cay - 4) = 0.$$
    The other solution to this quadratic equation in $p$ is $q = \mathrm{tr}(M_3M_2M_1)$.
\end{enumerate}
\end{lemma}
Let $\mathrm{Ch}_{\mathrm{SL}_2(\mathbb{F}_p)}(F_n)$ denote the set of characters $\tau$ of two-dimensional representations $\varphi:F_n\to \mathrm{SL}_2(\mathbb{F}_p)$, i.e. $\tau(w) = \mathrm{tr}(\varphi(w))$. Given an $n$-tuple of matrices $A_1,\dots,A_n \in \mathrm{SL}_2(\mathbb{F}_p)$ we get a corresponding character in $\mathrm{Ch}_{\mathrm{SL}_2(\mathbb{F}_p)}(F_n)$. Note that conjugating the matrices simultaneously by an element of $\mathrm{GL}_2(\mathbb{F}_p)$ yields the same character.

We will require the following lemma on generating pairs of matrices. An analogue of this lemma for triples of matrices is Lemma \ref{Triples up to conjugation}.
\begin{lemma}
\label{Pairs up to conjugation}
Let $A,B \in \mathrm{SL}_2(\mathbb{F}_p)$ be two matrices generating $\mathrm{PSL}_2(\mathbb{F}_p)$. Then for every pair of matrices $\tilde{A},\tilde{B}$ yielding the same character in $\mathrm{Ch}_{\mathrm{SL}_2(\mathbb{F}_p)}(F_2)$ there is a matrix $g \in \mathrm{GL}_2(\mathbb{F}_p)$ conjugating $(\tilde{A},\tilde{B})$ simultaneously to $(A,B)$.
\end{lemma}
\begin{proof}
    Assume that we have two matrices $A,B\in\mathrm{SL}_2(\mathbb{F}_p)$ with quotient images generating $\mathrm{PSL}_2(\mathbb{F}_p)$. We start by proving that a conjugating matrix $g$ exists in $\mathrm{GL}_2(\mathbb{F}_{p^2}).$

    Assume first that $\pm A$ is non-unipotent. Up to a conjugation in $\mathrm{GL}_2(\mathbb{F}_{p^2})$, we may diagonalize $A$ to the form $\begin{pmatrix}\lambda & 0 \\ 0 & \lambda^{-1}\end{pmatrix}.$ The trace $\mathrm{tr}(A)$ uniquely determines $\lambda$ up to inverse. Note that $\lambda \neq \pm 1$ otherwise $A = 1$ in $\mathrm{PSL}_2(\mathbb{F}_p)$ and then the two matrices generate a cyclic group. Let $B = \begin{pmatrix}a & b \\ c & d\end{pmatrix}.$ The traces $\mathrm{tr}(B),\mathrm{tr}(AB)$ are $a+d, \lambda a + \lambda^{-1}d$, and so uniquely determine $a,d$. The product $bc$ is uniquely determined by $bc = ad -  \det(B) = ad - 1$. Conjugating by $\begin{pmatrix}\mu & 0 \\ 0 & 1\end{pmatrix} \in \mathrm{GL}_2(\mathbb{F}_p)$, we get $\begin{pmatrix}a & \mu b \\ \mu^{-1}c & d\end{pmatrix}$, so an arbitrary matrix with given values of $a,d,bc$, as long as $bc \neq 0$ i.e. $ad \neq 1.$ So, we are left with the case $ad = 1$ and $bc = 0$. In this case the matrices $A,B$ are contained either in the upper or lower diagonal matrices. In either case - there is a fixed point for the subgroup generated by $A, B$ in $\mathbb{P}^1_{\mathbb{F}_{p^2}}.$ However, prior to the conjugation these matrices generated $\mathrm{PSL}_2(\mathbb{F}_p),$ which has no fixed points in $\mathbb{P}^1_{\mathbb{F}_{p^2}},$ for example because the lower and upper unipotents $\begin{pmatrix} 1 & 1 \\ 0 &  1 \end{pmatrix}, \begin{pmatrix} 1 & 0 \\ 1 & 1 \end{pmatrix}$ have unique and distinct fixed points.

    Otherwise, $A,B$ are both unipotent up to sign. If they are in the same unipotent subgroup, then they generate a cyclic group, in contradiction. Therefore, they are in different unipotent subgroups. As the action of $\mathrm{PSL}_2(\mathbb{F}_p)$ on $\mathbb{P}^1_{\mathbb{F}_p}$ is $2$-transitive (the action of $\mathrm{PGL}_2(\mathbb{F}_p)$ is even $3$-transitive) then these are conjugate to matrices of the form $\varepsilon_A\begin{pmatrix}1 & t \\ 0 & 1\end{pmatrix},\varepsilon_B\begin{pmatrix}1 & 0 \\ s & 1\end{pmatrix}$, for $\varepsilon_A,\varepsilon_B \in \{-1, 1\}$. Note that $\varepsilon_A,\varepsilon_B$ are uniquely determined by $\mathrm{tr}(A),\mathrm{tr}(B)$. Since $\mathrm{tr}(AB) = \varepsilon_A\varepsilon_B(ts + 2)$, it uniquely determines the product $ts$. As $A,B\neq 1$, otherwise they generate a cyclic subgroup in $\mathrm{PSL}_2(\mathbb{F}_p)$, it follows that $t,s \neq 0$. Now, conjugating by a diagonal matrix $\begin{pmatrix}\lambda & 0 \\ 0 & 1\end{pmatrix} \in \mathrm{PGL}_2(\mathbb{F}_p)$ will multiply $t$ by $\lambda$ and divide $s$ by $\lambda$, hence obtaining arbitrary such matrices with the given product $ts$.

    We have proved that a conjugating matrix exists in $\mathrm{GL}_2(\mathbb{F}_{p^2}),$ which we denote by $g_2.$ This matrix $g_2 \in M_2(\mathbb{F}_{p^2})$ solves the system of linear equations $g_2A = \tilde{A} g_2$ and $g_2B = \tilde{B} g_2,$ which is a system of equations with $\mathbb{F}_p$-coefficients. A system of linear equations over $\mathbb{F}_p$ has a non-trivial solution in a field extension if and only if it has a non-trivial solution in $\mathbb{F}_p.$ Therefore, there is a matrix $0 \neq g \in M_2(\mathbb{F}_p)$ satisfying this system of linear equations. It suffices to prove that such a $g$ must be invertible.

    Indeed, assume by contradiction that there is a vector $0 \neq v \in \mathbb{F}_p^2$ such that $gv = 0.$ Let $0 \neq w \in \mathbb{F}_{p^2}.$ Since $A, B$ generate $\mathrm{PSL}_2(\mathbb{F}_p),$ there is some matrix in the subgroup $C \in \langle A, B \rangle$ such that $Cv = \lambda w$ for some $0 \neq \lambda \in \mathbb{F}_p.$ Using the linear equations defining $g$, it follows that there is a corresponding matrix $\tilde{C} \in \langle \tilde{A}, \tilde{B} \rangle$ such that $gC = \tilde{C} g.$ Applying this equality to $v$, we get $\lambda gw = gCv = \tilde{C} gv = 0,$ and therefore $gw = 0.$ As $w$ was chosen arbitrarily - it follows that $g = 0,$ in contradiction.
\end{proof}
\begin{remark}
\label{Character variety F2}
    By Lemma \ref{Magnus character varieties} part (4), in order to check equality of the induced characters corresponding to a pair of matrices $A, B$, it suffices to compute $\mathrm{tr}(A), \mathrm{tr}(B), \mathrm{tr}(AB).$ This will be used in the subsequent section.
\end{remark}
\section{Primitivity of the action}
We start by making the following simplifying assumption that will help us in studying the induced permutation group. However, it is likely that a more general study can be performed, and some of the results follow verbatim without this assumption.
\begin{assumption}
\label{Assumption on non-conjugation}
    Assume that $\gamma,\delta \in \Gamma$ are non-trivial, non-involutions and non-unipotent. Moreover, assume that one of $\gamma,\delta$ is split, and the other is non-split.
    In particular, $\gamma,\delta$ have non-intersecting conjugate-centralizers. That is, for every $g\in \Gamma$, we have that $C_\Gamma(\gamma)\cap gC_\Gamma(\delta)g^{-1} = 1$.
\end{assumption}
\begin{remark}
\label{PGL2-non-conjugation}
    We note that Assumption \ref{Assumption on non-conjugation} further implies that for every $g \in \Gamma$, $C_{\mathrm{PGL}_2(\mathbb{F}_p)}(\gamma) \cap gC_{\mathrm{PGL}_2(\mathbb{F}_p)}(\delta)g^{-1} = 1.$

    Indeed, eigenvectors of a split element have $\mathbb{F}_p$-coefficients, whereas eigenvectors of a non-split element do not. So, it suffices to show that if $\gamma \in \Gamma$ is a non-trivial, non-involution, non-unipotent then elements $\hat{\gamma} \in C_{\mathrm{PGL}_2(\mathbb{F}_p)}(\gamma)$ have the same eigenvectors as $\gamma$.
    
    Indeed, after lifting $\gamma, \hat{\gamma}$ to $A \in \mathrm{SL}_2(\mathbb{F}_p), B \in \mathrm{GL}_2(\mathbb{F}_p)$ respectively - then $AB = \mu BA$ for some $\mu \in \mathbb{F}_p^\times$. Comparing determinants, it follows that $\mu^2 = 1$, i.e. $\mu = \pm 1$, and therefore $A^2B = \mu ABA = \mu^2 BA^2 = BA^2.$ Since $\gamma^2 \neq 1$ is non-unipotent, it follows that $A^2$ has distinct eigenvalues, and therefore $B, A^2$ have the same eigenvectors. Hence $B, A$ have the same eigenvectors, so $\gamma, \hat{\gamma}$ have the same eigenvectors.
\end{remark}
We also make the following additional assumption, which should be true for almost any pair $\gamma,\delta$.
\begin{assumption}
\label{Assumption on point in orbit}
    Assume that there is a point $P = (A,B,C,D) \in X_{\gamma,\delta}^{(2)}$ such that $AB^{-1},BC^{-1},CD^{-1}$ are all unipotent, and moreover that they generate distinct unipotent subgroups.
\end{assumption}
The reason for this is that $X_{\gamma,\delta}^{(2)}$ is $4$-dimensional, and these are three $1$-dimensional conditions on the point. In such a case,
\begin{theorem} \label{Primitive action on orbit}
Under Assumptions \ref{Assumption on non-conjugation} and \ref{Assumption on point in orbit}, and assuming that $p \geq 5,$ then $B_4$ acts primitively on the orbit of $P$.
\end{theorem}
We will need the following lemma for the proof
\begin{lemma}
\label{Lemma on opposite unipotents}
    Let $u, v \in \mathrm{SL}_2(\mathbb{F}_p)$ be unipotents in distinct unipotent subgroups. Then $u, v$ are simultaneously conjugate in $\mathrm{SL}_2(\mathbb{F}_p)$ to elements of the form $\begin{pmatrix}1 & t \\ 0 & 1\end{pmatrix},\begin{pmatrix}1 & 0 \\ s & 1\end{pmatrix}$. Allowing conjugation in $\mathrm{GL}_2(\mathbb{F}_p)$ we may further assume that $t = 1$.
\end{lemma}
\begin{proof}[Proof of Lemma \ref{Lemma on opposite unipotents}]
    A unipotent matrix has a unique preserved point in $\mathbb{P}^1_{\mathbb{F}_p}$. Since the two unipotents are in different subgroups, the preserved points are distinct. The $\mathrm{PSL}_2(\mathbb{F}_p)$-action on $\mathbb{P}^1_{\mathbb{F}_p}$ is $2$-transitive. Therefore we may move the two preserved points to be $0, \infty$ giving the required claim.
    Conjugating by $\begin{pmatrix}\frac{1}{t} & 0 \\ 0 & 1\end{pmatrix} \in \mathrm{PGL}_2(\mathbb{F}_p)$ we may assume that $t = 1$.
\end{proof}
The following proposition is key in the proof of Theorem \ref{Primitive action on orbit}.
\begin{proposition}
\label{Unique unipotent proper decomposition}
    Under Assumption \ref{Assumption on non-conjugation}, there is at most one proper decomposition of $\gamma,\delta$ up to equivalence such that $x,y,z,w$ are all unipotent as elements of $\mathrm{PSL}_2(\mathbb{F}_p)$.
\end{proposition}
We remind the reader that the notation $\bar{A} \in \mathrm{PSL}_2(\mathbb{F}_p)$ denotes the image in $\mathrm{PSL}_2(\mathbb{F}_p)$ of a matrix $A \in \mathrm{SL}_2(\mathbb{F}_p)$.
The following lemma will be used in the proof of Proposition \ref{Unique unipotent proper decomposition}
\begin{lemma}
\label{four_eight_decompositions}
Let $\gamma \in \mathrm{SL}_2(\mathbb{F}_p)$ be such that $\bar{\gamma} \in \mathrm{PSL}_2(\mathbb{F}_p)$ is non-trivial, non-involution, non-unipotent. Then,
\begin{enumerate}
    \item Up to simultaneous $C_{\mathrm{GL}_2(\mathbb{F}_p)}(\gamma)$-conjugation, there are four different pairs $(u, v)$ such that $\gamma = uv$ and $\bar{u},\bar{v}$ are unipotent. 
    
    Moreover, these are uniquely determined by the pair of signs of $\mathrm{tr}(u),\mathrm{tr}(v) \in \{\pm 2\}.$
    \item Up to simultaneous $C_{\mathrm{SL}_2(\mathbb{F}_p)}(\gamma)$-conjugation, there are eight different pairs $(u, v)$ such that $\gamma = uv$ and $\bar{u},\bar{v}$ are unipotent.
\end{enumerate} 
\end{lemma}
\begin{proof}
Let $\gamma = uv$ be a decomposition such that $\bar{u}, \bar{v}$ are unipotent. Then $\mathrm{tr}(u),\mathrm{tr}(v)\in\{\pm 2\}$. Moreover, since $\bar{\gamma}$ is not unipotent, $\bar{u}, \bar{v}$ must be elements of distinct unipotent subgroups. Therefore $\bar{u}, \bar{v}$ generate $\mathrm{PSL}_2(\mathbb{F}_p).$

It follows from Lemma \ref{Pairs up to conjugation} and Remark \ref{Character variety F2} that whenever $\tilde{u},\tilde{v}$ is another decomposition $\gamma = \tilde{u}\tilde{v}$ such that $\mathrm{tr}(u) = \mathrm{tr}(\tilde{u}), \mathrm{tr}(v) = \mathrm{tr}(\tilde{v}),$ there is a matrix $g \in \mathrm{GL}_2(\mathbb{F}_p)$ such that $gug^{-1} = \tilde{u}, gvg^{-1} = \tilde{v}$. In particular $g \gamma g^{-1} = \gamma,$ so $g \in C_{\mathrm{GL}_2(\mathbb{F}_p)}(\gamma).$
It remains to show that for any pair of signs $\varepsilon_u, \varepsilon_v \in \{\pm 1\}$ there are unipotent matrices with $\mathrm{tr}(u) = \varepsilon_u \cdot 2, \mathrm{tr}(v) = \varepsilon_v \cdot 2$ such that $uv = \gamma$. Indeed, choose
$$u = \varepsilon_u\begin{pmatrix} 1 & 1 \\ 0 & 1\end{pmatrix}, v = \varepsilon_v\begin{pmatrix} 1 & 0 \\ \varepsilon_u\varepsilon_v\mathrm{tr}(\gamma) - 2 & 1 \end{pmatrix}.$$
Then $\mathrm{tr}(u) = \varepsilon_u \cdot 2, \mathrm{tr}(v) = \varepsilon_v \cdot 2, \mathrm{tr}(uv) = \mathrm{tr}(\gamma)$. Moreover, as $\mathrm{tr}(\gamma) \neq \pm 2$, $\bar{u}, \bar{v}$ belong to distinct unipotent subgroups, and therefore generate $\mathrm{PSL}_2(\mathbb{F}_p).$ It remains to note that $uv$ is conjugate to $\gamma$ and we may therefore conjugate $u, v$ simultaneously by this conjugating element such that $uv = \gamma$, and $\bar{u}, \bar{v}$ still generate $\mathrm{PSL}_2(\mathbb{F}_p).$ Part (1) now follows.

For part (2), we first show that $C_{\mathrm{PSL}_2(\mathbb{F}_p)}(\bar{\gamma}) \leq C_{\mathrm{PGL}_2(\mathbb{F}_p)}(\bar{\gamma})$ is an index-$2$ subgroup. Indeed, the index is either $1$ or $2$, so we only need to find an element in $C_{\mathrm{PGL}_2(\mathbb{F}_p)}(\bar{\gamma}) \setminus C_{\mathrm{PSL}_2(\mathbb{F}_p)}(\bar{\gamma}).$
If $\bar{\gamma}$ is split, we may diagonalize it in $\mathrm{PSL}_2(\mathbb{F}_p)$ to the form $\begin{pmatrix} \lambda & 0 \\ 0 & \lambda^{-1}\end{pmatrix}.$ Then any matrix of the form $\begin{pmatrix} \mu & 0 \\ 0 & 1\end{pmatrix}$ for a non-square $\mu \in \mathbb{F}_p$ will be of the required form.

If $\bar{\gamma}$ is non-split, we diagonalize it in $\mathrm{PSL}_2(\mathbb{F}_{p^2}).$ Choose a generator $\mu \in \mathbb{F}_{p^2}^\times$. Then the matrix $\begin{pmatrix} \mu & 0 \\ 0 & \mu^p \end{pmatrix}$ has determinant $\mu^{p + 1}$ which is a generator of $\mathbb{F}_p^\times$, and therefore not a square in $\mathbb{F}_p$. Conjugating back, since $\mu, \mu^p$ are Galois conjugates, we get a matrix in $\mathrm{GL}_2(\mathbb{F}_p)$. Indeed, an equivalent description of this construction is as follows - let $\lambda \in \mathbb{F}_{p^2} \setminus \mathbb{F}_p$ be an eigenvalue of $\bar{\gamma}.$ Then there are $a, b \in \mathbb{F}_p$ such that $\mu = a + b \cdot \lambda$. Therefore $\mu^p = a + b \cdot \lambda^p = a + b \cdot \lambda^{-1},$ since the two eigenvalues of $A$ are Galois conjugates. The chosen commuting matrix is then $a \cdot I_2 + b \cdot \gamma \in \mathrm{GL}_2(\mathbb{F}_p)$.

Once this has been established -- it remains to note that the conjugation action of $\mathrm{GL}_2(\mathbb{F}_p)$ on pair of matrices $(u, v)$ factors through $\mathrm{PGL}_2(\mathbb{F}_p)$, and if $\bar{u}, \bar{v}$ generate $\mathrm{PSL}_2(\mathbb{F}_p)$ then the action on the orbit of $(u, v)$ is free -- as no element of $\mathrm{PGL}_2(\mathbb{F}_p)$ centralizes $\mathrm{PSL}_2(\mathbb{F}_p).$ This implies that each $C_{\mathrm{PGL}_2(\mathbb{F}_p)}(\gamma)$-orbit of $(u, v)$ is twice as large as the corresponding $C_{\mathrm{PSL}_2(\mathbb{F}_p)}(\gamma)$-orbits, and the claim follows.
\end{proof}
We remark before the proof of Proposition \ref{Unique unipotent proper decomposition} that a matrix $A \in \mathrm{SL}_2(\mathbb{F}_p)$ is unipotent or trivial up to sign if and only if $\mathrm{tr}(A) = \pm 2$, and is split if and only if $\mathrm{tr}(A)^2 - 4$ is a square, and non-split if and only if $\mathrm{tr}(A)^2 - 4$ is not a square. This is because the eigenvalues $\lambda,\lambda^{-1}$ of $A$ satisfy $\lambda + \lambda^{-1} = \mathrm{tr}(A)$, and so are in $\mathbb{F}_p$ if and only if the quadratic equation $\lambda^2 - \mathrm{tr}(A)\lambda + 1 = 0$ has two solutions in $\mathbb{F}_p$, i.e. the discriminant $\mathrm{tr}(A)^2 - 4$ is a square.

\begin{proof}[proof of Proposition \ref{Unique unipotent proper decomposition}]
We note that there are two $\mathrm{SL}_2(\mathbb{F}_p)$-conjugacy classes of unipotents. These are the conjugacy classes of $\begin{pmatrix}1 & t \\ 0 & 1\end{pmatrix}$ for $t \in \mathbb{F}_p$ a square or not a square. These two conjugacy classes merge as $\mathrm{GL}_2(\mathbb{F}_p)$-conjugacy classes.

Hence, the $\mathrm{SL}_2(\mathbb{F}_p)$-conjugacy class of an element $A \in \mathrm{SL}_2(\mathbb{F}_p)$ with unipotent image in $\mathrm{PSL}_2(\mathbb{F}_p)$ has one of four possible options. These are $\pm \begin{pmatrix} 1 & t \\ 0 & 1 \end{pmatrix}$ - and are dependent on the sign of $\mathrm{tr}(A) \in \{\pm 2\}$, and on the Legendre symbol of $t$.

Let $x, y, z, w$ be a proper decomposition of $\gamma,\delta$.
Equivalence of proper decompositions allows us to change the signs of all of $x,y,z,w$ simultaneously. Also, we may conjugate $x, z$ simultaneously by an element of $C_{\mathrm{PGL}_2(\mathbb{F}_p)}(\gamma)\setminus C_{\mathrm{PSL}_2(\mathbb{F}_p)}(\gamma)$ while simultaneously conjugating $y,w$ by an element of $C_{\mathrm{PGL}_2(\mathbb{F}_p)}(\delta) \setminus C_{\mathrm{PSL}_2(\mathbb{F}_p)}(\delta)$ (here we mean conjugation by an element of $\mathrm{GL}_2(\mathbb{F}_p)$ with quotient image not in $\mathrm{PSL}_2(\mathbb{F}_p)$).

Hence, we may assume without loss of generality by changing signs and by conjugating in $\mathrm{PGL}_2$ as in the previous paragraph, that of the four possible conjugacy classes, $x$ is conjugate to $\begin{pmatrix} 1 & 1 \\ 0 & 1\end{pmatrix}.$ Therefore, as $x,y$ are conjugate in $\mathrm{SL}_2(\mathbb{F}_p)$, the same is true of $y$.

By Lemma \ref{four_eight_decompositions}, given this condition on $x$, there are two remaining options for the pair $x,z$ up to $C_{\mathrm{SL}_2(\mathbb{F}_p)}(\gamma)$-conjugation, and these are uniquely determined by the sign of $\mathrm{tr}(z) = \pm 2.$ Denote these two options by $x_1, z_1$ and $x_2, z_2$ such that $\mathrm{tr}(z_1) = 2, \mathrm{tr}(z_2) = -2$

Similarly, the same is true of $\mathrm{tr}(w) = \pm 2,$ denote the corresponding pairs by $\mathrm{tr}(w_1) = 2, \mathrm{tr}(w_2) = -2.$ The only possibilities for proper decompositions of $\gamma,\delta$ up to equivalence are $x_1, y_1, z_1, w_1$ and $x_2, y_2, z_2, w_2$. However, we still need to check if $z, w$ are $\mathrm{SL}_2(\mathbb{F}_p)$-conjugate. We know that they are $\mathrm{GL}_2(\mathbb{F}_p)$-conjugate as they have the same trace.

It remains to show that exactly one of the pairs $z_1, w_1$ and $z_2, w_2$ are $\mathrm{SL}_2(\mathbb{F}_p)$-conjugate.
Assume at first that $\gamma$ is split and $\delta$ is non-split, and that $p \equiv 3 (\mathrm{mod}\ 4)$.

By Lemma \ref{Lemma on opposite unipotents}, and using the fact that $x_1, x_2$ are $\mathrm{SL}_2(\mathbb{F}_p)$-conjugate to $\begin{pmatrix}1 & 1 \\ 0 & 1\end{pmatrix},$ we may simultaneously $\mathrm{SL}_2(\mathbb{F}_p)$-conjugate the following pairs to the following forms:
$$x_1, z_1 \sim \begin{pmatrix} 1 & 1 \\ 0 & 1\end{pmatrix}, \begin{pmatrix} 1 & 0 \\ t & 1\end{pmatrix}$$
$$x_2, z_2 \sim \begin{pmatrix} 1 & 1 \\ 0 & 1\end{pmatrix}, -\begin{pmatrix} 1 & 0 \\ s & 1\end{pmatrix}$$
Note that $\mathrm{tr}(\gamma) = \mathrm{tr}(x_1z_1) = t + 2$ and $\mathrm{tr}(\gamma) = \mathrm{tr}(x_2z_2) = -(s + 2)$. Therefore, $ts = (\mathrm{tr}(\gamma) - 2)(-\mathrm{tr}(\gamma) - 2) = -(\mathrm{tr}(\gamma)^2 - 4)$. Since $p \equiv 3(\mathrm{mod}\ 4)$ and $\gamma$ is split, it follows that exactly one of $t,s$ is a square in $\mathbb{F}_p.$ So, $z_1, z_2$ are not conjugate in $\mathrm{SL}_2(\mathbb{F}_p).$

Applying the same argument for $y_1, w_1, y_2, w_2$ we see that $w_1, w_2$ are conjugate in $\mathrm{SL}_2(\mathbb{F}_p)$. Therefore, exactly one of the pairs $z_1, w_1$ and $z_2, w_2$ are conjugate in $\mathrm{SL}_2(\mathbb{F}_p).$

If $p \equiv 1(\mathrm{mod}\ 4)$, we reverse the roles of $\gamma,\delta$ in the proof.
\end{proof}
In fact, it follows from the proof that there is exactly one such proper decomposition.

We are now ready to prove Theorem \ref{Primitive action on orbit}.
\begin{proof}[Proof of Theorem \ref{Primitive action on orbit}]
    Assume by contradiction that $X\subseteq \mathrm{Orb}_{B_4}(P)$ is a block containing $P$.
    
    Consider the element $\sigma_1^p\in B_4$. This element preserves our point $P$, as $A_PB_P^{-1}$ is unipotent. Therefore it preserves the block $X$.
    
    Now, assume that there is some point $Q\in X$ such that $A_QB_Q^{-1}$ is not unipotent. This means that $\sigma_1^{p^2-1}$ preserves it, and therefore preserves $X$, so $\sigma_1 \in \langle\sigma_1^p, \sigma_1^{p^2 - 1}\rangle$ also preserves $X$.
    
    Giving a similar argument for $\sigma_2,\sigma_3$ we see that either all points have relevant coordinate unipotent, or the corresponding group element preserves the block $X$.
    
    Now, assume that $\sigma_1$ preserves $X$. In such a case, $\sigma_1^k(P) \in X$ for every $k$. However, $B_{\sigma_1^k(P)}C_{\sigma_1^k(P)}^{-1} = (A_PB_P^{-1})^k \cdot B_PC_P^{-1}$. Seeing as these elements are unipotent and in different unipotent subgroups, the product cannot be unipotent for every $k.$ Therefore $\sigma_2$ also preserves our block. Using a similar argument we see that if any of $\sigma_1,\sigma_2,\sigma_3$ preserves our block, then so do the others, and therefore $X$ is not a proper block.
    
    We now know that none of the $\sigma_i$ preserves $X$, and therefore every point on $X$ satisfies that all three elements $AB^{-1},BC^{-1},CD^{-1}$ are unipotent.
    
    Now, by Proposition \ref{Unique unipotent proper decomposition}, there is a unique proper decomposition with all four elements of the decomposition unipotent. Therefore any such point has the same first proper decomposition as $P$, and moreover this proper decomposition is maximal. Therefore any two such points are on the same $\langle \sigma_1,\sigma_3\rangle$-orbit by Corollary \ref{Moving inside first foliation}. In this case $X$ is naturally identified with a subset of a quotient of $\langle \sigma_1, \sigma_3\rangle/p \simeq(\mathbb{Z}/p)^2$, given by the transitive action of $\langle\sigma_1, \sigma_3\rangle \simeq \mathbb{Z}^2.$

    Since $\langle \sigma_1, \sigma_3 \rangle/p$ is abelian, and the action on the $P$-orbit is transitive, the stabilizers of any two points in the orbit are conjugate and therefore equal. So, the action has a fixed kernel subgroup that is the stabilizer of every point. Since $X$ is a non-trivial block, $|X| \geq 2$. As $X \subseteq (\langle\sigma_1, \sigma_3\rangle/p) \cdot P$, the kernel of the $\langle \sigma_1, \sigma_3 \rangle/p$-action cannot be everything.
    
    Assume by contradiction that the action of $\langle\sigma_1, \sigma_3\rangle/p$ on its $P$-orbit is not free. Then the kernel has the form $\langle\sigma_1^a \cdot \sigma_3^b\rangle,$ where $a, b$ are not simultaneously zero. Since $|X| \geq 2$, there is at least one non-trivial element of $\langle \sigma_1, \sigma_3\rangle/p$ sending $P$ to another element of $X$, and therefore preserving the block $X$. If $a \neq 0$, we may assume that this element has the form $\sigma_3^k, p\nmid k$, and otherwise $b \neq 0$ and the element has the form $\sigma_1^k$. Therefore either $\sigma_1 \in \langle \sigma_1^k, \sigma_1^p\rangle$ or $\sigma_3 \in \langle \sigma_3^k, \sigma_3^p\rangle$ preserves $X$, in contradiction to the fact that none of $\sigma_1, \sigma_2, \sigma_3$ preserve $X$.

    Therefore, the action of $\langle \sigma_1, \sigma_3 \rangle/p \simeq (\mathbb{Z}/p)^2$ on the $P$-orbit must be free, so $X$ is identified with a subset of $(\mathbb{Z}/p)^2.$ As before, $|X| \geq 2,$ implies that there is at least one non-trivial element of $\langle \sigma_1,\sigma_3\rangle/p$ that preserves $X.$ Since $\sigma_1,\sigma_3$ do not preserve $X$, this preserving element is of the form $\sigma_1^i\sigma_3^j$ with $p\nmid i,j.$
    
    Since $\sigma_1^i\sigma_3^j$ preserves $X$, it follows that $X$ contains the set of points of $(\mathbb{Z}/p)^2$ on a specific line corresponding to this preserving element. If any point outside this line were on $X$ this would give an independent $\langle\sigma_1, \sigma_3\rangle/p$-preserving direction, implying that $X$ is all of $(\mathbb{Z}/p)^2,$ in contradiction to the fact that $\sigma_1, \sigma_3$ do not preserve $X.$ So, $X$ is precisely the set of points of $(\mathbb{Z}/p)^2$ on a specific line.

    As a consequence, we get three unipotents $u,v,w$ in distinct unipotent subgroups, such that $u^tvw^t$ are all unipotent for every $t$. These three unipotents are $u = (A_PB_P^{-1})^{i}, v = B_PC_P^{-1}, w=(C_PD_P^{-1})^{-j}.$ We will prove that in such a case $uvw = v$. However, in this case every two points on the line have the same values for $AB^{-1}, BC^{-1}, CD^{-1}$ simultaneously. This means the two points are equal up to right-multiplication by some element $g \in \Gamma$ and therefore as $A^{-1}BC^{-1}D = \delta$, $g \in C_\Gamma(\delta)$ implying that the points on $X_{\gamma,\delta}^{(2)}$ are the same, in contradiction to the freeness of the action of $\langle\sigma_1, \sigma_3\rangle/p.$

    By Lemma \ref{Lemma on opposite unipotents}, up to a $\mathrm{PGL}_2(\mathbb{F}_p)$-conjugation we may assume that $u,w$ are lower and upper unipotents. Then we get $$\pm 2 = \mathrm{tr}\left(\begin{pmatrix}1 & t \\ 0 & 1\end{pmatrix} \cdot \begin{pmatrix}v_{11} & v_{12} \\ v_{21} & v_{22}\end{pmatrix} \cdot \begin{pmatrix}1 & 0 \\ \alpha t & 1\end{pmatrix}\right) = v_{11} + v_{22} + (\alpha v_{12} + v_{21})t + \alpha v_{22}t^2$$ for every $t.$

    This is a quadratic polynomial attaining at most $2$ values when substituting an arbitrary $t \in \mathbb{F}_p.$ Since $p \geq 5,$ this means that at least one of the two values is attained at least $3$ times. This constant polynomial is equal to a quadratic polynomial at least at $3$ points, and therefore the quadratic polynomial is in fact constant. Up to changing the sign of the matrix $v,$ we may assume that the value $2$ is attained for every $t.$

    Therefore $v_{22} = 0$, $v_{11} = 2$ and $\alpha v_{12} + v_{21} = 0.$ Moreover $v_{12}v_{21} = -1,$ as $\mathrm{det}(v) = 1.$ So $\alpha$ is a square. Therefore we may assume up to change of coordinates by a diagonal element $\sqrt{\alpha},\frac{1}{\sqrt{\alpha}}$, and by taking an appropriate power of $u, w$, that $\alpha = 1$. So, $v = \begin{pmatrix}2 & 1 \\ -1 & 0\end{pmatrix}$ and $uvw = v$.
\end{proof}

\begin{remark}
\label{Remark on AutF2 extension}
    Note that the extension from $B_4$ to $\mathrm{Aut}(F_2)$ still acts on $\mathrm{Orb}_{B_4}(P).$ Indeed, the action $(A,B,C,D)\mapsto (D,C,B,A)$ sends our special point $P$ to one where $AB^{-1},CD^{-1}$ are both unipotent. This puts it on the unique unipotent leaf of the foliation. So, the two points are even on the same $\langle\sigma_1,\sigma_3\rangle$-orbit.
\end{remark}
\section{Proof of Theorem \ref{Alternating quotients}}
We have now shown that the action of $B_4$ on $\mathrm{Orb}_{B_4}(P)$ is primitive. The goal of this section is to prove that in many cases this primitive action is in fact the full alternating or symmetric action. We will use the following theorem of Guralnick--Magaard \cite[Theorem 1]{guralnick1998minimal}, which we quote from Meiri--Puder \cite[Theorem 4.13]{meiri2018markoff} (Note that the proof of this theorem uses the classification of finite simple groups).
\begin{theorem}
\label{CFSG alternating}
    Let $G \leq S_n$ be a primitive permutation group. Assume that there exists an $x \in G$ with at least $\frac{n}{2}$ fixed points, and that $x$ is not an involution. Then one of the following holds
    \begin{enumerate}
    \item There are $r \geq 1, m\geq 5, 1 \leq k \leq \frac{m}{4}$, such that $n=\binom{m}{k}^r$, $S_m$ acts on $\binom{[m]}{k}$, the set of subsets of $[m]$ of size $k$ in the natural way, $A_m^r \leq G \leq S_m \wr S_r$ acts on $\binom{[m]}{k}^r$.
    \item There is an $r \geq 1$ such that $n = 6^r$, the group $S_6$ acts on $[6]$ via the non-trivial outer automorphism, and $A_6^r \leq G \leq S_6 \wr S_r$ acts on $[6]^r$.
    \end{enumerate}
\end{theorem}
We also quote the following nice lemma from Meiri--Puder \cite{meiri2018markoff}
\begin{lemma}
\label{primes in subset action}
    Consider the embedding $\iota:S_m\to S_n$, for $n = \binom{m}{k}$ given by the natural action on $\binom{[m]}{k}$, for some $2 \leq k \leq \frac{m}{4}.$ Moreover let $q,s$ be arbitrary primes. Then if for some $\pi \in S_m$ $\iota(\pi)$ has cycles of length divisible by $q$ and $s$, then it also has a cycle of length divisible by $qs$.
\end{lemma}

\begin{theorem}
\label{Alternating action on orbit}
    Under the Assumptions  \ref{Assumption on non-conjugation},\ref{Assumption on point in orbit}, and for sufficiently large primes $p$, the action of $B_4$ on $\mathrm{Orb_{B_4}(P)}$ is the full alternating or symmetric action.
\end{theorem}
The sufficiently large lower bound on $p$ given by this theorem is effectively computable, as can be seen by the proof.

The following definitions will be useful for the proof of the theorem
\begin{definition}
For every point $Q \in \tilde{X}_{\gamma,\delta}^{(2)}$ its $\sigma_1$-matrix is $A_QB_Q^{-1}$, its $\sigma_2$-matrix is $B_QC_Q^{-1}$ and its $\sigma_3$-matrix is $C_QD_Q^{-1}$. We say that $Q$ is $\sigma_i$-unipotent (respectively split, or non-split) if its $\sigma_i$-matrix is unipotent (respectively split, or non-split). We say that $Q$ is a $\sigma_i$-involution if its $\sigma_i$-matrix is an involution in $\Gamma$.
\end{definition}
Note that a point $Q \in X_{\gamma,\delta}^{(2)}$ having a particular $\sigma_i$-type is independent of the choice of lift to $\tilde{X}_{\gamma,\delta}^{(2)}$ as this simply conjugates the corresponding $\sigma_i$-matrix. So, the types are well defined on $X_{\gamma,\delta}^{(2)}.$ We will need the following lemmas for the proof of theorem \ref{Alternating action on orbit}.
\begin{lemma}
\label{arbitrary types}
     Under Assumption \ref{Assumption on point in orbit}, for every $i$ there are points $Q \in \mathrm{Orb}_{B_4}(P)$ of every possible type (among unipotent, split, non-split), and moreover the $\sigma_i$-matrix can have the order of an arbitrary non-trivial element of $\Gamma$.
\end{lemma}
\begin{lemma}
\label{orders of sigmas}
     Under Assumption \ref{Assumption on non-conjugation}, the size of the $\sigma_i$-cycle of a point $Q \in X_{\gamma,\delta}^{(2)}$ is precisely the order in $\Gamma$ of the corresponding $\sigma_i$-matrix.
\end{lemma}
\begin{lemma}
\label{non-commutation of sigma powers}
    Under Assumptions \ref{Assumption on non-conjugation}, \ref{Assumption on point in orbit} and for $k < p$, $\sigma_1^k,\sigma_2^k$ do not commute on $\mathrm{Orb}_{B_4}(P)$, and similarly $\sigma_2^k,\sigma_3^k$ do not commute on $\mathrm{Orb}_{B_4}(P)$.
\end{lemma}
\begin{lemma}
\label{wreath product structure}
    Choose $\tau \in S_m \wr S_r$ of the form $\tau = ((\tau_1,\dots,\tau_r),\tilde{\tau})$ for $\tau_i\in S_m$ and $\tilde{\tau}\in S_r$. Consider $\rho = \tau^{r!} \in S_m^r \leq S_m \wr S_r$, and $\rho = (\rho_1,\dots,\rho_r)$ for some $\rho_i \in S_m$. Then $\rho$ has fixed conjugacy class along the cycles of $\tilde{\tau}$, i.e. $\rho_j,\rho_{\tilde{\tau}(j)}$ are conjugate.
\end{lemma}
\begin{proof}[Proof of Lemma \ref{arbitrary types}]
    Consider the unipotents $u = A_PB_P^{-1},v = B_PC_P^{-1}$. By our assumption they are in different unipotent subgroups. Therefore they are conjugate in $\mathrm{PGL}_2(\mathbb{F}_p)$ to matrices of the form $\begin{pmatrix}1 & 0 \\ \alpha & 1\end{pmatrix}, \begin{pmatrix}1 & 1 \\ 0 &  1\end{pmatrix}$. Now, $uv^k=\begin{pmatrix}1 & 0 \\ \alpha & 1\end{pmatrix}\begin{pmatrix}1 & k \\ 0 &  1\end{pmatrix}=\begin{pmatrix}1 & k \\ \alpha & k\alpha + 1\end{pmatrix}$ has trace $2 + k\alpha$. In particular, as $\alpha \neq 0$, this trace can take any possible value in $\mathbb{F}_p$, and so the matrix $uv^k$ can be in an arbitrary non-unipotent $\Gamma$-conjugacy class (and choosing $k = 0$ or $k = -\frac{4}{\alpha}$ it can be unipotent). Now, noting that $A_{\sigma_2^k(P)}B_{\sigma_2^k(P)}^{-1} = A_PB_P^{-1} (B_PC_P^{-1})^k = uv^k$ we get points with arbitrary $\sigma_1$-matrix order.
    A similar argument works for the $\sigma_2,\sigma_3$-types.
\end{proof}
\begin{proof}[Proof of Lemma \ref{orders of sigmas}]
Clearly the order of the matrix divides the order of the cycle. We wish to prove the converse. let $k$ be the order of the cycle, i.e. $\sigma_i^k(Q) = Q$.
We assume that $i = 1$, the other two cases have a very similar proof.
$\sigma_1^k(Q) = ((A_QB_Q^{-1})^kA_Q,(A_QB_Q^{-1})^kB_Q,C_Q,D_Q)$. Therefore, for some $\hat{\gamma}\in C_{\mathrm{PGL}_2(\mathbb{F}_p)}(\gamma),\hat{\delta}\in C_{\mathrm{PGL}_2(\mathbb{F}_p)}(\delta)$, $\sigma_1^k(Q) = \hat{\gamma}Q\hat{\delta}$. In particular, $\hat{\gamma}C_Q\hat{\delta} = C_Q$ and therefore $\hat{\delta} = C_Q^{-1}\hat{\gamma}^{-1}C_Q$, contradicting our assumption (c.f. Remark \ref{PGL2-non-conjugation}). Therefore $\hat{\gamma}=\hat{\delta} = 1$, and hence $(A_QB_Q^{-1})^k = 1$, as needed.
\end{proof}
\begin{proof}[Proof of Lemma \ref{non-commutation of sigma powers}]
    Assume by contradiction that $\sigma_1^k,\sigma_2^k$ commute. In particular, $\sigma_1^{mk},\sigma_2^{mk}$ commute for every $m$. Consider the action on our special point $P$.
    $$\sigma_2^k(P) = (A_P, (B_PC_P^{-1})^kB_P, (B_PC_P^{-1})^kC_P, D_P).$$
    $$\sigma_1^k(P) = ((A_PB_P^{-1})^kA_P, (A_PB_P^{-1})^kB_P, C_P, D_P).$$
    $$\sigma_1^k(\sigma_2^k(P))=((A_PB_P^{-1}(B_PC_P^{-1})^{-k})^kA_P, (A_PB_P^{-1}(B_PC_P^{-1})^{-k})^k(B_PC_P^{-1})^kB_P, (B_PC_P^{-1})^kC_P, D_P)$$
    If $\sigma_1^k\sigma_2^k=\sigma_2^k\sigma_1^k$ then equating $BC^{-1}$-coordinates we get that
    $$(A_PB_P^{-1})^kB_PC_P^{-1} = (A_PB_P^{-1}(B_PC_P^{-1})^{-k})^k(B_PC_P^{-1})^kB_P ((B_PC_P^{-1})^kC_P)^{-1} =$$
    $$(A_PB_P^{-1}(B_PC_P^{-1})^{-k})^k B_PC_P^{-1}.$$
    We note that by Remark \ref{PGL2-non-conjugation}, the fact that the $D$-coordinate is fixed implies that the last equation is a real equality and not equality up to conjugation, as $\hat{\gamma}D_P\hat{\delta} \neq D_P$ unless $\hat{\gamma} = \hat{\delta} = 1.$
    Therefore $(A_PB_P^{-1})^k = (A_PB_P^{-1}(B_PC_P^{-1})^{-k})^k$. However, the left hand side is a (non-trivial) unipotent, and therefore so is the right hand side. As $k \neq 0\ (\mathrm{mod}\ p)$, it follows that $A_PB_P^{-1}(B_PC_P^{-1})^{-k}$ is unipotent. As this is the product of two unipotents in different unipotent subgroups, this is only possible for one specific value of $k \neq 0 (\mathrm{mod}\ p)$, and therefore cannot be simultaneously true for $k$ and $2k$, in contradiction. (As in the proof of Lemma \ref{arbitrary types}, the trace has the form $2 + k\alpha$ for $\alpha \neq 0$, and so attains the value $\pm 2$ for a unique value of $k \neq 0.$)
\end{proof}
\begin{proof}[Proof of Lemma \ref{wreath product structure}]
Note that if $\tau^k \in S_m^r$ has the required property, then so does $\tau^{mk}$ for any $m$. Therefore, it suffices to prove our property for the case that $\tilde{\tau}$ is a cycle. In this case, $((\tau_1, \dots, \tau_r),(1 \dots r))^r = ((\tau_1 \tau_2 \cdots \tau_r, \tau_2 \cdots \tau_r \tau_1, \dots, \tau_r\tau_1 \cdots \tau_{r - 1}), id) \in S_m^r$.
\end{proof}
We can now turn to the proof of Theorem \ref{Alternating action on orbit}
\begin{proof}[Proof of Theorem \ref{Alternating action on orbit}]
    Consider the element $\sigma_1$ as a permutation of the orbit $\mathrm{Orb}_{B_4}(P)$. Using Lemmas \ref{arbitrary types}, \ref{orders of sigmas} we get that it has cycles of order dividing each one of $\frac{p-1}{2},p,\frac{p+1}{2}$, and not of order $2$. Moreover, every cycle has order dividing one of these three numbers. Now consider the three permutations $\sigma_1^{\frac{p-1}{2} \cdot p},\sigma_1^{\frac{p-1}{2}\cdot\frac{p+1}{2}},\sigma_1^{p\cdot\frac{p+1}{2}}$. One of these three permutations preserves at least $\frac{2}{3}$ of the elements of $\mathrm{Orb}_{B_4}(P)$. Indeed, we simply need to choose the type with the least number of points on the orbit, and raise to the power of the orders of the remaining two types.

    Moreover, note that none of these three permutations are involutions, as $\frac{p-1}{2},p,\frac{p+1}{2}$ are pairwise coprime. Therefore, by Theorem \ref{CFSG alternating}, we get that our action has one of the two possible types.
    Assume first that $A_6^r \leq G \leq S_6 \wr S_r$ acts on $[6]^r$. In this case, as $G$ contains an element of order divisible by $p$, it follows that $r \geq p$. However, $6^p \leq 6^r = |\mathrm{Orb}_{B_4}(P)| \leq |\Gamma|^4 \leq p^{12}$, which is impossible for large $p$.

    Therefore there are values of $m,k,r$ such that the permutation group is of the form $A_m^r \leq G \leq S_m \wr S_r$ acting on $\Delta^r$ for $\Delta = \binom{[m]}{k}$ the set of $k$-subsets of $[m]$, for some $m \geq 5, r \geq 1$ and $1 \leq k \leq \frac{m}{4}$. Once again, there is an element of order $p$, and so either $r \geq p$ or $m \geq p$. If $r \geq p$, then $$5^p \leq 5^r \leq m^r \leq \binom{m}{k}^r = |\mathrm{Orb}_{\mathrm{B_4}}(P)| \leq |\Gamma|^4 \leq p^{12},$$which is once again impossible for large $p$. So, $m \geq p$, implying that $$p^r \leq m^r \leq \binom{m}{k}^r \leq |\mathrm{Orb}_{B_4}(P)| \leq p^{12},$$ so $r \leq 12$.

    For every $1\leq i \leq 3$ consider the permutation $\sigma_i^{r!} \in S_m^r$. Note that for sufficiently large $p,$ this permutation is non-trivial. Moreover, every one of its cycles has order dividing one of $\frac{p-1}{2},p,\frac{p+1}{2},$ and it has cycles of order dividing each of these. So, $\sigma_i^{r!} \in S_m^r$ can only have one non-trivial coordinate. By Lemma \ref{non-commutation of sigma powers}, we get that this coordinate is the same for $\sigma_1^{r!},\sigma_2^{r!},\sigma_3^{r!}$. Assume without loss of generality that this is the first coordinate.

    Denote the projections of $\sigma_1,\sigma_2,\sigma_3$ to $S_r$ by $\tau_1,\tau_2,\tau_3$. By Lemma \ref{wreath product structure}, we know that $\tau_i(1) = 1$, as otherwise there would be at least two non-trivial coordinates in $\sigma_i^{r!}$. It now follows that the permutation group generated by $\sigma_1,\sigma_2,\sigma_3$ is contained in $S_m^r \rtimes S_{r-1}$, where $S_{r - 1} \leq S_r$ is the subgroup preserving the first coordinate. In particular, our permutation group preserves the block structure defined by $\{\{J\} \times \Delta^{r-1} | J \in \Delta\}$, and therefore $r = 1$ by Theorem \ref{Primitive action on orbit}.

Let $q$ be a prime dividing $\frac{\frac{p-1}{2}}{\mathrm{gcd}(\frac{p-1}{2},r!)}$ and let $s$ be a prime dividing $\frac{\frac{p+1}{2}}{\mathrm{gcd}(\frac{p+1}{2}, r!)}$. For large enough $p$ such $q,s$ exist. The permutation $\sigma_i^{r!}$ has cycles of length divisible by each of $q,s,p.$ However, there are no cycles of length divisible by $qs, qp, sp$. Now, by Lemma \ref{primes in subset action}, we know that $k = 1$, as $\sigma_1$ contains cycles of length divisible by $q$ and $s$ but not by $qs$. This implies our required result.
\end{proof}
We now have all of the tools to prove the main theorem. The proof uses particular choices of matrices for convenience, however these are not necessary and much more general constructions are possible.
\begin{proof}[Proof of Theorem \ref{Alternating quotients}]
    Pick a sufficiently large prime $p$ as in Theorem \ref{Alternating action on orbit}. Moreover, assume that $5$ is a square and $13$ is not a square mod $p$. This can be achieved for infinitely many primes $p$ using Dirichlet's Theorem and Quadratic Reciprocity, by assuming (for example) that $$ p \equiv 1 (\mathrm{mod}\ 5), p \equiv -2 (\mathrm{mod}\ 13).$$
    Consider the three unipotents $$u = \begin{pmatrix}1 & 0 \\ 1 & 1\end{pmatrix}, v = \begin{pmatrix} 1 & 1 \\ 0 & 1\end{pmatrix}, w = \begin{pmatrix}-1 & 1 \\ -4 & 3\end{pmatrix},$$generating distinct unipotent subgroups. Note that $$uw = \begin{pmatrix}-1 & 1 \\ -5 & 4\end{pmatrix}, uvwv^{-1} = \begin{pmatrix}-5 & 9 \\ -9 & 16\end{pmatrix}$$ and so $\mathrm{tr}(uw) = 3, \mathrm{tr}(uvwv^{-1}) = 11$. Therefore $uw$ is in a split torus, as $\mathrm{tr}(uw)^2 - 4 = 5$ is a square mod $p.$ Moreover, $uvwv^{-1}$ is in a non-split torus, as $\mathrm{tr}(uvwv^{-1})^2 - 4 = 117 = 3^2 \cdot 13$ is not a square mod $p.$ Define $\gamma = uw$, $\delta = (uvwv^{-1})^{-1}$, and consider the point $P = (1, u^{-1}, v^{-1}u^{-1}, w^{-1}v^{-1}u^{-1}) \in X^{(2)}_{\gamma,\delta}$, as $A_PB_P^{-1}C_PD_P^{-1} = uw=\gamma$ and $D_P^{-1}C_PB_P^{-1}A_P = uvwv^{-1} = \delta^{-1}$, so $A_P^{-1}B_PC_P^{-1}D_P = \delta$. The point $P$ satisfies Assumption \ref{Assumption on point in orbit}. Moreover, $\gamma,\delta$ satisfy Assumption \ref{Assumption on non-conjugation}. By Theorem \ref{Alternating action on orbit} we get that the $B_4$-action on $\mathrm{Orb}_{B_4}(P)$ is alternating or symmetric.

    Using the involution $\varepsilon$ from Proposition \ref{Basic Properties of the B4 Action}, and the fact that every non-unipotent element of $\mathrm{PSL}_2(\mathbb{F}_p)$ is conjugate to its inverse, we may extend the $B_4$-action to an $\mathrm{Aut}(F_2)$-action. By Remark \ref{Remark on AutF2 extension}, this extension preserves $\mathrm{Orb}_{B_4}(P)$, and so the $\mathrm{Aut}(F_2)$ action is also alternating or symmetric. Moreover, note that $|\mathrm{Orb}_{B_4}(P)| \geq p$, as $\sigma_1^k(P) \neq \sigma_1^m(P)$ for all $k,m < p, k \neq m$.

    Consider $F_2 \simeq \mathrm{Inn}(F_2) \lhd \mathrm{Aut}(F_2)$. This subgroup acts on $\mathrm{Orb}_{B_4}(P)$ as a normal subgroup of the alternating or symmetric group. Therefore it is either alternating, symmetric, or trivial. Since $F_2 \leq \mathrm{Aut}(F_2)'$, it cannot be the entire symmetric group. We must prove it is non-trivial. For this, consider the action of $\sigma_1\sigma_3^{-1} \in F_2$ on $P$. This is $\sigma_1\sigma_3^{-1}(P) = (A_PB_P^{-1}A_P, A_P, D_P, D_PC_P^{-1}D_P)$. Assume that $\sigma_1\sigma_3^{-1}(P) = P$, i.e. $(A_PB_P^{-1}A_P, A_P, D_P, D_PC_P^{-1}D_P) = \hat{\gamma} \cdot (A_P, B_P, C_P, D_P) \cdot \hat{\delta}$, for some $\hat{\gamma} \in C_{\mathrm{PGL}_2(\mathbb{F}_p)}(\gamma), \hat{\delta} \in C_{\mathrm{PGL}_2(\mathbb{F}_p)}(\delta)$. In particular, $A_PB_P^{-1} = \hat{\gamma} A_PB_P^{-1} \hat{\gamma}^{-1}$. Since $\hat{\gamma}$ is non-unipotent and $A_PB_P^{-1}$ is unipotent, $\hat{\gamma} = 1$. Arguing similarly for $A_P^{-1}B_P$ we get that $\hat{\delta} = 1$. However, this means that $A_P = B_P$, which is not the case. Hence we get a characteristic alternating quotient of $F_2$.
\end{proof}
We finish this section with a proof of corollary \ref{covering torus bundles}
\begin{proof}[Proof of Corollary \ref{covering torus bundles}]
The fibre bundle $S_{1,1} \to E \to B$ yields a long exact sequence of homotopy groups $$\dots \to \pi_2(B) \to \pi_1(S_{1,1}) \to \pi_1(E) \to \pi_1(B) \to \pi_0(S_{1,1}) \to \dots$$Since $\pi_1(S_{1,1}) \simeq F_2$ has trivial center, the map $\pi_2(B) \to \pi_1(S_{1,1})$ is trivial, see \cite[Section 2.6]{brown2011nonabelian}, \cite[Chapter IV.3]{whitehead2012elements}. Since $S_{1,1}$ is connected, we get an exact sequence $$1 \to \pi_1(S_{1,1}) \to \pi_1(E) \to \pi_1(B) \to 1.$$Using this, the conjugation action gives a map $\pi_1(E) \to \mathrm{Aut}(\pi_1(S_{1,1})) \simeq \mathrm{Aut}(F_2)$. Composing this with our quotients $\mathrm{Aut}(F_2) \to S_n$, we get finite alternating or symmetric quotients of $\pi_1(E)$. Indeed, $A_n$ is contained in the image of these maps as the conjugation action of $\pi_1(S_{1,1})$ on itself implies that the inner automorphisms are contained in the image of the map $\pi_1(E) \to \mathrm{Aut}(F_2)$, and the inner automorphisms map surjectively to $A_n$.
\end{proof}
\section{Large rank quotients}
In this section we will describe an infinite sequence of Zariski-dense representations of $B_4$ in $\mathrm{PSL}_n$, using similar ideas to \cite{chen2025finite}, \cite{funar2018profinite}. Such a Zariski-dense representation of a finitely generated group can then give infinitely many finite simple quotients via the Weisfeiler--Pink strong approximation \cite{weisfeiler1984strong}, \cite{pink2000strong}.
Note that as $\mathrm{PSL}_n$ is center-free, it will follow from the Zariski-density that $Z(B_4)$ will have trivial image. If the inner automorphisms $F_2 \lhd B_4$ are mapped non-trivially, then their image will also be Zariski-dense and therefore we will get infinitely many finite simple quotients of $F_2$ that are $\mathrm{Aut}^+(F_2)$-preserved.
We will then need to find an extension of these from the index-$2$ subgroup $\mathrm{Aut}^+(F_2)$ to all of $\mathrm{Aut}(F_2)$.

The representations we will use are the monodromy representations of the Knizhnik--Zamolodchikov connection, or equivalently - due to a theorem of Kohno--Drinfeld (see \cite{kohno1987monodromy} or \cite{abad2014introduction}, \cite{kassel2012quantum}), representations of braid groups arising from representations of quantum groups. We refer the reader to Kassel's book \cite{kassel2012quantum} for more information and motivation as to the definitions and properties of quantum groups and their representations.

We will be interested specifically in the quantum group $\mathcal{U}_q(\mathfrak{sl}_2)$, and will use the results of \cite{jackson2011lawrence} characterizing these representations.

\subsection{Quantum representations of braid groups}
We will use the notation developed in section $2$ of \cite{jackson2011lawrence}, describing an integral-subalgebra $\mathcal{U}$ of the quantum group $\mathcal{U}_q(\mathfrak{sl}_2)$. This is a Hopf-algebra, which we will define with generators and relations, over the ring $\mathbb{L} = \mathbb{Z}[q,q^{-1},s,s^{-1}]$, with $q,s$ two variables.

Denote the $q$-numbers, $q$-factorials, $q$-binomials by the formulas
$$[n]_q = \frac{q^n - q^{-n}}{q - q^{-1}}, [n]_q! = [n]_q\cdot\ldots\cdot[1]_q, {n \brack k}_q = \frac{[n]_q!}{[k]_q!\cdot[n-k]_q!}.$$
The algebra $\mathcal{U}$ is generated over $\mathbb{L}$ by the elements $K,K^{-1},E,(F^{(n)})_{n\geq0}$ under the relations
$$KK^{-1} = K^{-1}K = 1, KEK^{-1} = q^2E, KF^{(n)}K^{-1} = q^{-2n}F^{(n)},$$
$$F^{(n)}F^{(m)} = {n + m \brack n}_q F^{(n + m)}, [E, F^{(n + 1)}] = F^{(n)}(q^{-n}K-q^nK^{-1})$$
where we think of $F^{(n)}$ as the divided powers $\frac{(q-q^{-1})^n}{[n]_q!}F^n$. The co-product and antipode are defined by
$$\Delta(K) = K \otimes K, \Delta(E) = E \otimes K + 1 \otimes E,$$ $$\Delta(F^{(n)}) = \sum_{j=0}^n q^{-j(n-j)}K^{j-n}F^{(j)}\otimes F^{(n-j)}$$
$$S(K) = K^{-1}, S(E) = -EK^{-1}, S(F^{(n)}) = (-1)^n q^{n(n-1)}K^nF^{(n)}.$$
This Hopf algebra has an infinite dimensional representation $\textbf{V}$ generated over $\mathbb{L}$ freely by $(v_n)_{n\geq0}$, and with the generators acting via
$$Kv_j = sq^{-2j}v_j, Ev_j = v_{j - 1}$$
$$F^{(n)}v_j = \left({n + j \brack j}_q\prod_{k=0}^{n-1} (sq^{-k-j} - s^{-1}q^{k+j})\right) v_{j + n}.$$
where we use the convention that $v_{-1} = 0$.
The action of $\mathcal{U}$ on $\textbf{V}$ yields a natural action of $\mathcal{U}\otimes\mathcal{U}$ on $\textbf{V}\otimes\textbf{V}$, the tensors taken over $\mathbb{L}$. The co-product then naturally gives an action of $\mathcal{U}$ on $\textbf{V} \otimes \textbf{V}$ by $u(v\otimes w) = \Delta(u)(v \otimes w)$, and therefore more generally an action of $\mathcal{U}$ on $\textbf{V}^{\otimes n}$.

There is an element $\reflectbox{R} \in \mathcal{U} \tilde{\otimes} \mathcal{U}$, where the $\tilde{\otimes}$ allows infinite sums, turning $\mathcal{U}$ into a quastriangular Hopf algebra. For any given element of $\textbf{V} \otimes \textbf{V}$ only a finite number of non-zero summands will occur in the infinite sum representation of $\reflectbox{R}$, and therefore multiplication by $\reflectbox{R}$ gives a map $\textbf{V} \otimes \textbf{V} \to \textbf{V} \otimes \textbf{V}$. Composing this map with the transposition of the coordinates of $\textbf{V}$, we get a solution $R \in \mathrm{End}(\textbf{V} \otimes \textbf{V})$ to the Quantum Yang--Baxter equation, i.e. $(R \otimes 1) (1 \otimes R) (R \otimes 1) = (1 \otimes R) (R \otimes 1) (1 \otimes R)$ as elements of $\mathrm{End}(\textbf{V} \otimes \textbf{V} \otimes \textbf{V})$. This $R$ acts on $\textbf{V} \otimes \textbf{V}$ by
$$R(v_i \otimes v_j) = s^{-i-j}\sum_{n=0}^{i} q^{2(i-n)(j+n)+\frac{n(n-1)}{2}}{n + j \brack j}_q\prod_{k=0}^{n-1} (sq^{-k-j} - s^{-1}q^{k + j})v_{j + n} \otimes v_{i-n}.$$
The fact that $R$ satisfies the Quantum Yang--Baxter equation induces a representation of $B_n$ on $\textbf{V}^{\otimes n}$. Moreover $R$ is an automorphism of $\textbf{V}\otimes\textbf{V}$ as a representation of $\mathcal{U}$, i.e. commutes with the $\mathcal{U}$-action, and therefore $\textbf{V}^{\otimes n}$ has commuting $\mathcal{U},B_n$-actions, see also \cite[Theorem $7$]{jackson2011lawrence}.
Let $\textbf{V}_{n,\ell}$ denote the weight subspaces, generated over $\mathbb{L}$ by pure tensors of basis elements of length $n$ and coordinate sum $\ell$. Equivalently, $\textbf{V}_{n,\ell} = \mathrm{ker}(K - s^nq^{-2\ell}) \leq \textbf{V}^{\otimes n}$.
Define $\textbf{W}_{n,\ell} = \mathrm{ker}(E)\cap\textbf{V}_{n,\ell}$, called the highest weight space. The $B_n$-action then restricts to an action on $\textbf{W}_{n,\ell}$. We summarize these in the following theorem, see \cite[Theorems 1, 3]{jackson2011lawrence}.
\begin{theorem}
\label{Jackson irreducibility}
    The highest weight space $\textbf{W}_{n,\ell} \leq \textbf{V}_{n,\ell} \leq \textbf{V}^{\otimes n}$ is a free $\mathbb{L}=\mathbb{Z}[q,q^{-1},s,s^{-1}]$-module of dimension $\binom{n+\ell-2}{\ell}$. There is a linear action of $B_n$ on $\textbf{W}_{n,\ell}$, and $\textbf{W}_{n, \ell}$ is irreducible as a $B_n$-representation over the fraction field $\mathbb{K} = \mathbb{Q}(q,s)$.
\end{theorem}
This representation can be described as a map $B_n \to \mathrm{GL}_{\binom{n + \ell - 2}{\ell}}(\mathbb{L})$, and projects to a map $B_n \to \mathrm{PGL}_{\binom{n+\ell - 2}{\ell}}(\mathbb{L})$. Let $\bar{\mathbb{K}}$ denote the algebraic closure of $\mathbb{K}$, and let $N = \binom{n+\ell-2}{\ell}$. Then there are natural embeddings $\mathrm{PGL}_N(\mathbb{L}) \leq \mathrm{PGL}_N(\mathbb{K}) \leq \mathrm{PGL}_N(\bar{\mathbb{K}}) = \mathrm{PSL}_N(\bar{\mathbb{K}})$, the last equality due to the fact that we may divide by the $N^\mathrm{th}$-root of the determinant of the matrix. We will prove the following fact:
\begin{theorem}
\label{Knizhinik Zamolodichikov Zariski density}
    The action of $B_n$ on $\textbf{W}_{n,\ell}$ has Zariski-dense image in $\mathrm{PSL}_{\binom{n + \ell - 2}{\ell}}(\bar{\mathbb{K}})$.
\end{theorem}
Using this, we obtain in particular that $Z(B_4)$ has trivial image in $\mathrm{PGL}_{\binom{\ell + 2}{2}}(\mathbb{L})$, as $\mathrm{PSL}_{\binom{\ell + 2}{2}}(\bar{\mathbb{K}})$ has trivial center. So, we have an induced projective representation of $B_4/Z(B_4)$.

In \cite{jackson2011lawrence}, the action of $B_{n-1} \simeq \langle\sigma_2, \dots, \sigma_{n-1} \rangle  \leq B_n$ on $\textbf{W}_{n,\ell}$ was studied. We will use this structure for an inductive proof of Zariski-density. The following is \cite[Lemma 16]{jackson2011lawrence}, taking into account also \cite[Theorem 3]{jackson2011lawrence}.
\begin{theorem}
\label{Irreducible Decomposition of Bn-1}
    $\textbf{W}_{n,\ell} \simeq \bigoplus_{k=0}^\ell \textbf{W}_{n-1,k}$ is the decomposition of $\textbf{W}_{n,\ell}$ into irreducible components of the $B_{n-1}$-action.
\end{theorem}
We also mention \cite[Lemma 13]{jackson2011lawrence}
\begin{theorem}
\label{Vn,l irreducible decomposition}
$\textbf{V}_{n, \ell}$ decomposes as a $B_n$-representation into irreducibles as $\textbf{V}_{n, \ell} \simeq \bigoplus_{k=0}^\ell \textbf{W}_{n, \ell-k}$.
\end{theorem}
In the following lemma we compute the eigenvalues of the $\sigma_i$ on $\textbf{W}_{n, \ell}$.
\begin{lemma}
\label{Eigenvalue computation}
    Consider the $B_2 \simeq \mathbb{Z}$-action on the one-dimensional space $\textbf{W}_{2,\ell}$. This action is multiplication by $(-1)^\ell q^{\ell(\ell-1)}s^{-2\ell}$.
\end{lemma}
\begin{proof}
    $E(v_i\otimes v_j) = (E\otimes K + 1\otimes E)(v_i \otimes v_j) = sq^{-2j} \cdot v_{i-1} \otimes v_j + v_i \otimes v_{j-1}$.
    Therefore the kernel of $E: \textbf{V}_{2,\ell} \to \textbf{V}_{2,\ell - 1}$ is the one dimensional space given by $u = \sum_{i=0}^\ell \alpha_i v_{\ell - i} \otimes v_i$ such that $\alpha_i \cdot sq^{-2i} + \alpha_{i+1} = 0$. Assuming that $\alpha_0 = 1$, we get $\alpha_i = (-1)^is^iq^{-i(i-1)}$, so $u = \sum_{i=0}^\ell(-1)^i s^i q^{-i(i-1)} v_{\ell - i} \otimes v_i$.
    Now, applying $\sigma_1$ on $\textbf{V}_{2,\ell}$ is equivalent to applying $R$, whose action is given by
    $$Ru = \sum_i (-1)^i s^i q^{-i(i-1)} R(v_{\ell - i} \otimes v_i) = $$
    $$\sum_i  (-1)^i s^i q^{-i(i-1)} s^{-\ell} \sum_{m=0}^{\ell - i}q^{2(\ell -i -m)(i+m)+\frac{m(m-1)}{2}}{m + i \brack m}_q\prod_{k=0}^{m-1} (sq^{-k-i} - s^{-1}q^{k + i})v_{i + m} \otimes v_{\ell - i-m} = $$
    $$\sum_{t=0}^\ell q^{2(\ell - t)t}s^{-\ell}\left(\sum_{i + m =t}  (-1)^i s^i q^{-i(i-1)} q^{\frac{m(m-1)}{2}}{t \brack m}_q\prod_{k=i}^{t-1} (sq^{-k} - s^{-1}q^{k })\right)v_t \otimes v_{\ell - t} = $$
 In order to prove the claim, it suffices to consider the coordinate of $v_0 \otimes v_\ell$, as we already know that $R$ preserves the one-dimensional kernel of $E$. In this case, the coordinate of $Ru$ is $s^{-\ell}$ and the coordinate of $u$ is $(-1)^\ell s^\ell q^{-\ell(\ell - 1)}$.
\end{proof}
\subsection{Proof of Zariski Density}
In the current subsection we will prove Theorem \ref{Knizhinik Zamolodichikov Zariski density}. The proof will be by induction on $n$, where Theorem \ref{Irreducible Decomposition of Bn-1} gives our representations an inductive structure. Theorem \ref{lie algebra amplification} below, which is \cite[Theorem 7.4]{kuperberg2018tqft} will be the induction step. The base case, $n = 3$, will be proved directly, as Theorem \ref{lie algebra amplification} does not apply to the decomposition into $B_2 \simeq \mathbb{Z}$-irreducible representations. We state this separately in the following Proposition
\begin{proposition}
\label{Zariski density n=3}
    For every $\ell \geq 0$, the action of $B_3$ on $\textbf{W}_{3,\ell}$ has Zariski-dense image in $\mathrm{PSL}_{\ell + 1}(\bar{\mathbb{K}})$.
\end{proposition}
The following lemma \cite[Lemma 19]{jackson2011lawrence} will be a useful computational tool in the proof for $n = 3$.
\begin{lemma}
\label{Mixing subspaces Jackson}
Consider the action of $B_{n}$ on $\textbf{W}_{n, \ell}$ and consider the decomposition $\textbf{W}_{n, \ell} \simeq \bigoplus_{k = 0}^\ell \textbf{W}_{n - 1, k}$ into $B_{n-1} \simeq \langle \sigma_2, \dots, \sigma_n \rangle \leq B_n$-representations.

Then for every $0 < k < \ell$ there is an element $w \in \textbf{W}_{n - 1,\ell - k}$ such that $\sigma_1(w)$ has non-trivial coordinates in all of $\textbf{W}_{n - 1, \ell - r}$ for every $1 \leq r \leq k + 1$.
\end{lemma}
We will also need the following lemma about parabolic subgroups of $\mathrm{GL}_n$. It is a slight strengthening of the fact that whenever $P \leq G$ is a parabolic subgroup of a reductive group $G$, and $U \lhd P$ is its unipotent radical, then $N_G(U) =P$. We state and prove it for the special case of maximal parabolic subgroups of $\mathrm{GL}_n.$
\begin{lemma}
\label{Parabolic subgroup reducibility}
Let $U, V$ be vector spaces over an arbitrary field $k$. Denote $W = \mathrm{End}(U \oplus V)$. Let $\pi: W \to \mathrm{Hom}(V, U)$ be the natural projection, and let $W^+ \leq W$ denote the kernel of $\pi$.
    
Let $A \in \mathrm{GL}(U \oplus V)$. Assume that $A \cdot \mathrm{Hom}(U, V) \cdot A^{-1} \leq W^+$, where $\mathrm{Hom}(U, V) \leq W$ is the natural embedding. Then $A \in W^+$, and therefore $A$ preserves the subspace $V \leq U \oplus V$.
\end{lemma}
\begin{proof}
We prove this claim working with coordinates. Denote $p = \mathrm{dim}(U), q = \mathrm{dim}(V)$. In a basis for $U \oplus V$, the matrices of $\mathrm{Hom}(U, V)$ are those supported on coordinates $(i, j)$ where $p + 1 \leq i \leq p + q$ and $1 \leq j \leq p$.

Fix $i, k \geq p + 1, j, m \leq p.$ Denote by $E_{i, j}$ the matrix with a unique $1$ coordinate at $(i, j)$. Then $E_{i, j} \in \mathrm{Hom}(U, V)$, so by our condition $$(A \cdot E_{i,j} \cdot A^{-1})_{m, k} = A_{m, i} \cdot (A^{-1})_{j, k} = 0.$$
As this is true for every choice of $i, j, k, m,$ either $A$ or $A^{-1}$ is in $W^+$. As $W^+ \cap \mathrm{GL}(U \oplus V) \leq \mathrm{GL}(U \oplus V)$ is a (parabolic) subgroup, if $A^{-1} \in W^+$, then $A \in W^+$.
\end{proof}
Lastly, the following graph-theoretic lemma will be used in the proof of Proposition \ref{Zariski density n=3}.
\begin{lemma}
\label{Graph Theoretic Lemma}
Let $\Gamma$ be a directed loop-less graph on $n$ vertices. Assume that this graph is transitive in the sense that if $\vec{uv}, \vec{vw} \in \vec{E}(\Gamma),$ for three distinct vertices $u,v,w \in V(\Gamma)$, then $\vec{uw} \in \vec{E}(\Gamma).$ Moreover, assume
\begin{enumerate}
    \item For every vertex $v \in V(\Gamma)$, the out-degree of $v$ is at least $\max\{1, n - 3\}.$
    \item For every vertex $v \in V(\Gamma)$, the in-degree of $v$ is at least $\max\{1, n - 4\}.$
    \item There is a vertex $u \in V(\Gamma)$ with out-degree at least $n - 2.$
\end{enumerate}
Then either $\Gamma$ is the complete directed graph on $n$ vertices, or $n$ is either $4$ or $5$ and there is a partition $V(\Gamma) = A \cup B$, where the induced subgraph on both $A$ and $B$ is the complete directed graph, and every vertex of $A$ has an outgoing edge to every vertex of $B$. Moreover $|A|, |B| \geq 2$.
\end{lemma}
\begin{proof}
Assume first that $n \geq 6$. Fix a pair of vertices $u, v \in V(\Gamma)$. We wish to prove that $\vec{uv} \in \vec{E}(\Gamma).$ Assume by contradiction that this is not the case.
Then both the out-neighborhood of $u$, and the in-neighborhood of $v$ are subsets of $V(\Gamma) - \{u, v\}$ of sizes $n-3, n-4$ respectively. Since $(n - 3) + (n - 4) > n - 2$, these subsets must intersect, and then the transitivity property implies our claim.

We are left with the case $n \leq 5$. Note that if the graph contains a cycle of length $n - 1$ then every two vertices on the cycle are connected by edges in both directions. The remaining vertex must have an in-edge and an out-edge to vertices of this cycle, and therefore the graph is the complete directed graph.

Moreover, note that the graph must contain a cycle, as we may pick an arbitrary vertex and walk from it forward until the walk crosses itself. This completes the case of $n \leq 3$.

The remaining cases are of $n \in \{4, 5\}$.
If $n = 4$, then either the graph is the complete directed graph, or it contains a cycle of length $2$. The remaining two vertices have both incoming and outgoing edges. Either they form a second cycle, or both are connected in both directions to the cycle, in which case the graph is the complete directed graph.
So, we may assume that $\Gamma$ contains two disjoint cycles on $2$ vertices. In this case, as there is a vertex with out-degree at least $2$, this means that there is at least one edge between the two cycles. By transitivity, the required claim follows.

We are left with the case $n = 5$. In this case every vertex has at least $2$ outgoing edges. Assume firstly that the graph contains a cycle of length $3$. In this case, as before, the remaining $2$ vertices are either connected to the cycle in both directions, in which case we have a cycle of length at least $4$, or are on a cycle of length $2$. As the graph also has a vertex of out-degree at least $3$, the two cycles are connected in some direction and we are once again done by transitivity.

So, we may assume that our graph contains no cycle of length greater than $2$. Consider a maximal (directed) path in the graph. The last vertex of the path can only have out-going edges back to vertices of the path. However, as it has at least $2$ such edges, this gives a cycle of length at least $3$.
\end{proof}
We are now ready to prove Proposition \ref{Zariski density n=3}.
\begin{proof}[Proof of Proposition \ref{Zariski density n=3}]
The claim is trivial for $\ell = 0.$ Hence, we may assume $\ell \geq 1.$

Consider the Zariski-closure of $B_3$, $G \leq \mathrm{PSL}_{\ell + 1}(\bar{\mathbb{K}})$. Theorem \ref{Irreducible Decomposition of Bn-1} gives a decomposition of this representation into $B_2$-irreducible representations, or equivalently into the eigenspaces of $\sigma_2$, $\textbf{W}_{3,\ell} = \bigoplus_{k=0}^\ell \textbf{W}_{2, k}$. By Lemma \ref{Eigenvalue computation}, we obtain that these eigenvalues are all distinct, and moreover the quotients of pairs of these eigenvalues are also all distinct.
    
    Consider the lie algebra $\mathfrak{sl}(\textbf{W}_{3,\ell}) = D \bigoplus_{i \neq j} \textbf{W}_{2,i} \otimes \textbf{W}_{2,j}^\ast$, where $D \simeq \bar{\mathbb{K}}^{n-1}$ are the diagonal matrices of trace $0$. Let $\mathfrak{g}$ be the lie algebra of $G$. We wish to prove that $\mathfrak{g} = \mathfrak{sl}(\textbf{W}_{3,\ell})$. Firstly, seeing as $\sigma_2$ has pairs of eigenvalues with quotient that is not a root of unity, the subgroup generated by $\sigma_2$ is indiscrete in the Zariski topology, and therefore $G$ contains a positive dimensional torus. Therefore $\dim(\mathfrak{g}) > 0$. Moreover, the adjoint action of $B_3$ on $\mathfrak{sl}(\textbf{W}_{3,\ell})$ factors through the action of $G$, and so preserves the sub-algebra $\mathfrak{g}$. Hence, $\mathfrak{g}$ is the direct sum of a subspace of $D$ and a subset of the subspaces $\textbf{W}_{2,i} \otimes \textbf{W}_{2,j}^\ast$.

    Assume firstly that $\mathfrak{g}$ contains some subspace $\textbf{W}_{2,i_0} \otimes \textbf{W}_{2,j_0}^\ast$ for some $i_0 \neq j_0$. In such a case we define a finite directed graph $\Gamma$ on the set of spaces $\textbf{W}_{2, i}$, with a directed edge from $\textbf{W}_{2, j}$ to $\textbf{W}_{2, i}$ if $\textbf{W}_{2, i} \otimes \textbf{W}_{2, j}^\ast \leq \mathfrak{g}$. By our assumption the edge-set of the graph $\Gamma$ is non-empty. Moreover, as $\mathfrak{g}$ is a lie algebra, this graph is transitive. Indeed this follows as the lie brackets $[E_{i,j}, E_{j,k}] = E_{i,k}$, for distinct $i, j, k.$ We will prove the remaining assumptions of Lemma \ref{Graph Theoretic Lemma}.
    
    Choose a pure tensor in $\textbf{W}_{2, i_0} \otimes \textbf{W}_{2, j_0}$. If we consider the adjoint action of $b \in B_3$ on such a matrix, we get a pure tensor in $\textbf{W}_{3, \ell} \otimes \textbf{W}_{3, \ell}^\ast.$ Therefore the set of pairs $(i, j)$ with non-trivial coordinate after the adjoint action is a product of two subsets. Since the trace of the matrix is $0$, the non-trivial coordinates cannot be only of the form $\textbf{W}_{2, k} \otimes \textbf{W}_{2, k}^\ast$ for some $k.$ In fact there cannot be exactly one non-zero diagonal coordinate.

    Fix $0 \leq i \leq \ell$. Using Theorem \ref{Jackson irreducibility} we may find a braid $b \in B_3$ such that the image of an arbitrary vector in $\textbf{W}_{2, i_0}$ has a non-zero $\textbf{W}_{2, i}$-coordinate. Applying this braid element to $\textbf{W}_{2, i_0} \otimes \textbf{W}_{2, j_0}^\ast$ we obtain that $\mathfrak{g}$ contains $\textbf{W}_{2, i} \otimes \textbf{W}_{2, j_1(i)}^\ast$ for some $j_1(i) \neq i.$ Indeed, if the only non-zero coordinate of the $\textbf{W}_{3,\ell}^\ast$-vector is $\textbf{W}_{2,i}^\ast$, then there must be at least one diagonal coordinate other than $\textbf{W}_{2, i} \otimes \textbf{W}_{2, i}^\ast$, say $\textbf{W}_{2, j_1(i)} \otimes \textbf{W}_{2, j_1(i)}^\ast$. As the element is a pure tensor the $\textbf{W}_{2, i} \otimes \textbf{W}_{2, j_1(i)}^\ast$-coordinate is also non-zero.
    
    Applying a similar argument, we see that for every $0 \leq j \leq \ell$, there is an $0 \leq i_1(j) \leq \ell$ such that $\textbf{W}_{2, i_1(j)} \otimes \textbf{W}_{2, j}^\ast \leq \mathfrak{g},$ and $i_1(j) \neq j.$ In particular, we see that the graph $\Gamma$ has both an incoming and outgoing edge out of every vertex.

    If $\ell = 1,$ then this means that the graph $\Gamma$ is the complete directed graph (and in particular satisfies the conditions of Lemma \ref{Graph Theoretic Lemma}). So, we assume that $\ell \geq 2.$

    Therefore $0 < \ell - 1 < \ell$, so using Lemma \ref{Mixing subspaces Jackson}, applying $\sigma_1$ to $\textbf{W}_{2, 1} \otimes \textbf{W}_{2, j_1(1)}^\ast$ we obtain that $\mathfrak{g}$ contains every subspace of the form $\textbf{W}_{2, i} \otimes \textbf{W}_{2, j_2}^\ast$ for some fixed $j_2$ and every $j_2 \neq i < \ell$. This means that every vertex of the graph $\Gamma$, with the exception of at most two, has an incoming edge from $\textbf{W}_{2, j_2}$. Similarly, applying $\sigma_1^{-1}$ to $\textbf{W}_{2, i_1(\ell - 1)} \otimes \textbf{W}_{2, \ell - 1}^\ast$ and remembering that the action on the dual will then be given by $\sigma_1^T$, we obtain that $\mathfrak{g}$ contains every subspace of the form $\textbf{W}_{2, i_2} \otimes \textbf{W}_{2, j}^\ast$ for some fixed $i_2$ and every $0 < j < \ell, j \neq i_2.$ That is, every vertex with the exception of at most three, has an outgoing edge to $\textbf{W}_{2, i_2}$.

    It follows that there is a subspace $V \leq \textbf{W}_{3, \ell}$ of codimension at most $2$ such that $V \otimes \textbf{W}_{2, j_2}^\ast \leq \mathfrak{g}$. Using Theorem \ref{Jackson irreducibility} once more, for every fixed $j$ we get that we can find a braid $b \in B_3$ such that the image of an arbitrary vector of $\textbf{W}_{2, j_2}^\ast$ has non-trivial $\textbf{W}_{2, j}^\ast$-coordinate. The subspace $b(V)$ is of codimension at most $2$, and therefore has vectors with non-zero $\textbf{W}_{2, i}$-coordinate for all but at most $2$ choices of $i$. Therefore, up to at most $3$ exceptions, $\textbf{W}_{2, i} \otimes \textbf{W}_{2, j}^\ast \leq \mathfrak{g}$ (possibly removing $i = j$). Reformulated in terms of the graph $\Gamma$, every vertex has outgoing edges to all vertices with at most $3$ exceptions.

    Similarly, we obtain that every vertex has incoming edges from all vertices with at most $4$ exceptions. This implies all of the conditions of Lemma \ref{Graph Theoretic Lemma}.

    If $\ell \neq 3,4$ this means that the graph $\Gamma$ is the complete directed graph. Using the fact that $[E_{i,j}, E_{j, i}] = E_{i, i} - E_{j, j}$, we see that $\mathfrak{g} = \mathfrak{sl}(\textbf{W}_{3, \ell}).$ This implies that $B_3$ is Zariski dense.

    For the remaining cases of $\ell \in \{3, 4\}$, no extra edge can be added to $\Gamma$ without implying that $\Gamma$ is the full directed graph. So, $\mathfrak{g}$ is a subalgebra of a maximal parabolic Lie algebra, containing its unipotent radical. Using Lemma \ref{Parabolic subgroup reducibility}, we see that every matrix in $B_3$ keeps a subspace of $\textbf{W}_{3, \ell}$ invariant, in contradiction to Theorem \ref{Jackson irreducibility}.

    We are left with the case that $\mathfrak{g}$ is contained in the diagonal matrices. In this case, the connected component of the identity $G^\circ \leq G$ is abelian, and has finite index in $G$. Hence, there is an integer $N$ such that $\sigma_1^N, \sigma_2^N$ commute projectively. In this case there is a constant $\lambda$ such that $\sigma_1^N \sigma_2^N = \lambda \sigma_2^N \sigma_1^N.$ Comparing determinants, it follows that $\lambda$ is a root of unity. Since $\sigma_1^{kN}\sigma_2^{kN}=\lambda^{k^2}\sigma_2^{kN}\sigma_1^{kN}$, it follows that after further increasing $N$, we may assume that the matrices commute.

    However, both $\sigma_1^N,\sigma_2^N$ have distinct eigenvalues, since no two of the eigenvalues of $\sigma_1$ (or $\sigma_2$) are equal up to a root of unity (by Lemma \ref{Eigenvalue computation}, Theorem \ref{Irreducible Decomposition of Bn-1}). This means that the eigenvectors of $\sigma_1^N, \sigma_2^N$ are the same. Since these are also the eigenvectors of $\sigma_1, \sigma_2$ - this contradicts the irreducibility of $\textbf{W}_{3,\ell}$.
\end{proof}
We now prove Theorem \ref{Knizhinik Zamolodichikov Zariski density} by induction on $n.$ This will be done using similar ideas to \cite{kuperberg2018tqft}. We will use \cite[Theorem 7.4]{kuperberg2018tqft}:
\begin{theorem}
\label{lie algebra amplification}
    Let $V$ be a finite-dimensional irreducible representation of a connected algebraic group $G$. Let $H \leq G$ be a Zariski-closed, connected subgroup. Let $V|_H = \bigoplus_{k=1}^m W_k$ be the irreducible decomposition of the restriction to $H$. Assume that
    \begin{enumerate}
        \item $(W_k)_{k=1}^m$ are pairwise non-isomorphic and non-dual as projective representations of $H$.
        \item For every $k$, $H$ surjects onto $\mathrm{PSL}(W_k)$.
        \item At most one of the $W_i$ has dimension $1$ and at most one has dimension $2$.
    \end{enumerate}
Then $G$ surjects onto $\mathrm{PSL}(V)$.
\end{theorem}
Note that conditions $1, 3$ are immediately satisfied if the $W_k$ all have different dimensions, as will be our case.

\begin{proof}[Proof of Theorem \ref{Knizhinik Zamolodichikov Zariski density}]
The case $n = 2$ is trivial as the representations are $1$-dimensional. The case $n = 3$ is Proposition \ref{Zariski density n=3}. We assume $n \geq 4.$

Let $G \leq \mathrm{PSL}(\textbf{W}_{n,\ell} \otimes \overline{\mathbb{K}})$ be the Zariski-closure of the subgroup $B_n.$ Let $H_1, H_2 \leq G$ be the Zariski-closures of the subgroups $\langle \sigma_1, \dots, \sigma_{n-2} \rangle, \langle \sigma_2, \dots, \sigma_{n-1} \rangle.$ Let $H_1^\circ \leq H_1, H_2^\circ \leq H_2, G^\circ \leq G$ denote the connected components of the identity in these groups. Then $H_1^\circ, H_2^\circ \leq G^\circ,$ since every connected subgroup of $G$ is contained in $G^\circ.$

By Theorem \ref{Jackson irreducibility}, $\textbf{W}_{n, \ell}$ is $B_n$-irreducible, and is therefore $G$-irreducible. We will first prove that $\textbf{W}_{n, \ell}$ is $G^\circ$-irreducible.

By Theorem \ref{Irreducible Decomposition of Bn-1}, $\textbf{W}_{n, \ell} \simeq \bigoplus_{k=0}^\ell \textbf{W}_{n-1, k}$ is the decomposition into $\langle \sigma_2, \dots, \sigma_{n-1}\rangle \simeq B_{n-1}$-irreducible representations. By the induction hypothesis, the image of the map $B_{n-1} \to \mathrm{PSL}(\textbf{W}_{n-1, k} \otimes \overline{\mathbb{K}})$ is Zariski-dense, and therefore the induced map $H_2 \to \mathrm{PSL}(\textbf{W}_{n-1, k} \otimes \overline{\mathbb{K}})$ is surjective. Therefore, the map $H_2^\circ \to \mathrm{PSL}(\textbf{W}_{n-1, k} \otimes \overline{\mathbb{K}})$ is also surjective. In particular, $\textbf{W}_{n, \ell} \simeq \bigoplus_{k=0}^\ell \textbf{W}_{n-1, k}$ is also the decomposition of $\textbf{W}_{n, \ell}$ into $H_2^\circ$-irreducible representations.

The dimensions of $\textbf{W}_{n-1, k}, 0 \leq k \leq \ell$ are all distinct, and therefore no two of these irreducible representations of $H_2^\circ$ are equivalent. So every $H_2^\circ$-invariant subspace of $\textbf{W}_{n, \ell}$ is a direct sum of a subset of these spaces. In particular, such a subspace is also $H_2$-invariant.

By Lemma \ref{semi-direct product AutF2}, the two subgroups $B_{n-1} \leq B_n$ are conjugate, and therefore also every $H_1^\circ$-invariant subspace of $\textbf{W}_{n, \ell}$ is $H_1$-invariant.

Assume that $U \leq \textbf{W}_{n, \ell}$ is a $G^\circ$-invariant subspace. Therefore, it is both $H_1^\circ, H_2^\circ$-invariant, and hence it is both $H_1, H_2$-invariant. Since the two subgroups $\langle \sigma_1, \dots, \sigma_{n-2} \rangle, \langle \sigma_2, \dots, \sigma_{n-1} \rangle$ generate $B_n$, it follows that $U$ is $B_n$-invariant. By Theorem \ref{Jackson irreducibility}, it now follows that $\textbf{W}_{n, \ell}$ is $G^\circ$-irreducible.

We are now in the position to apply Theorem \ref{lie algebra amplification}, for the subgroup $H_2^\circ \leq G^\circ.$ We have proved that $\textbf{W}_{n, \ell}$ is $G^\circ$-irreducible. Further, $\textbf{W}_{n, \ell} \simeq \bigoplus_{k=0}^\ell \textbf{W}_{n-1, k}$ is the decomposition into $H_2^\circ$-irreducible components. These all have distinct dimensions, implying conditions 1, 3 of the theorem. Condition 2 follows by the induction hypothesis. Therefore, by Theorem \ref{lie algebra amplification}, $G = \mathrm{PSL}(\textbf{W}_{n,\ell} \otimes \overline{\mathbb{K}}).$
\end{proof}
\subsection{Extension to $\mathrm{Aut}(F_2)$}
Knowing this Zariski density, we almost have our required finite simple quotients. We need to be able to extend our maps from $\mathrm{Aut}^+(F_2)$ to $\mathrm{Aut}(F_2)$. We do this using the following theorem:
\begin{theorem}
    The map $B_4 / Z(B_4) \to \mathrm{PGL}_{\binom{\ell + 2}{2}}(\mathbb{L})$ extends to a map $\mathrm{Aut}(F_2) \to \mathrm{Aut}(\mathrm{PSL}_{\binom{\ell + 2}{2}}(\mathbb{L}))$.
\end{theorem}
For the proof of this theorem, we will need a certain skew-linear non-degenerate form on the representations of a quantum group, see \cite{wenzl1998C}. For the convenience of the reader we include the details with our notations here.

Consider the order $2$ automorphism of $\mathbb{K}$ given by $q \mapsto q^{-1}, s\mapsto s^{-1}$, which preserves $\mathbb{L}$. Denote this automorphism by $f \mapsto \bar{f}$. Under this automorphism, $\mathcal{U}_q(\mathfrak{sl}_2)$ has a structure of a Hopf-$\ast$ algebra, given by the anti-linear anti-algebra homomorphism $E^\ast = F, F^\ast = E, K^\ast = K^{-1}$, i.e. it is extended by the properties $\forall a,b \in \mathcal{U}_q(\mathfrak{sl}_2),\ (ab)^\ast=b^\ast a^\ast$ and $\forall f \in \mathbb{K}, a \in \mathcal{U}_q(\mathfrak{sl}_2)\ (fa)^\ast = \bar{f}a^\ast$
We now define a hermitian form on $\textbf{V}$ as in \cite[Lemma 2.2]{wenzl1998C}
\begin{lemma}
    There is a unique skew-linear form $(\ ,\ )$ on $\textbf{V}$ that is $\mathcal{U}_q(\mathfrak{sl}_2)$-invariant, i.e. $(av_0,bv_0)=(b^\ast av_0, v_0) = (v_0, a^\ast b v_0)$ and normalized, i.e. $(v_0, v_0) = 1$.
\end{lemma}
\begin{proof}
    Under such a product for every $n > 0$, $(v_0, F^{(n)}v_0) = ((F^{(n)})^\ast v_0,v_0)=\overline{\frac{(q-q^{-1})^n}{[n]_q!}}(E^nv_0, v_0) = 0$, and so $(v_0, v_n) = 0$. A similar argument will show that $v_n$ is an orthogonal basis. Moreover,
    $$(F^{(n)}v_0, v_n) = \frac{[n]_q!}{(q-q^{-1})^n}(v_0, E^nv_n) = \frac{[n]_q!}{(q-q^{-1})^n} (v_0, v_0) = \frac{[n]_q!}{(q-q^{-1})^n},$$and so $(v_n, v_n)$ is uniquely defined, and given by
    $$(v_n, v_n) = ({n \brack 0}_q \prod_{k=0}^{n-1} (sq^{-k} - s^{-1}q^{k}))^{-1} (F^{(n)}v_0, v_n) = \prod_{k=0}^{n-1} (sq^{-k} - s^{-1}q^{k})^{-1} \frac{(q-q^{-1})^n}{[n]_q!}(F^nv_0, v_n) = $$
    $$\prod_{k=0}^{n-1} (sq^{-k} - s^{-1}q^{k})^{-1} \frac{(q-q^{-1})^n}{[n]_q!} (v_0, E^n v_n) = \prod_{k=0}^{n-1} (sq^{-k} - s^{-1}q^{k})^{-1} \frac{(q-q^{-1})^n}{[n]_q!}$$
    To prove existence, it suffices by linearity to prove that the form defined by these equations satisfies for $a \in \{E, F, K\}$ and $n, m\geq 0$, $(v_n, av_m) = (a^\ast v_n, v_m)$.
    If $a=K$ and $n\neq m$ both sides are zero. If $n = m$, then $$(v_n, Kv_n) = (v_n, sq^{-2n}v_n) = \overline{(sq^{-2n})}(v_n, v_n) = s^{-1}q^{2n}(v_n, v_n) = (K^{-1}v_n, v_n).$$If $a = E$ and $n + 1 \neq m$ then both sides are zero. If $n + 1 = m$, then
    $$(Fv_n, v_{n+1}) = (q-q^{-1})^{-1}(F^{(1)}v_n, v_{n+1}) = $$
    $$(q-q^{-1})^{-1}{n + 1 \brack n}_q (sq^{-n} - s^{-1}q^n) (v_{n+1}, v_{n+1}) = (v_n, v_n) = (v_n, Ev_{n+1}).$$and the argument for $a = F$ is the same.
\end{proof}
Note that the computation in the proof also implies that this form is non-degenerate.
We can now extend this form to $\textbf{V}^{\otimes n}$ (denoted by $( , )_p$ in \cite{wenzl1998C}) by $(v_{k_1} \otimes \dots \otimes v_{k_n}, v_{m_1} \otimes \dots \otimes v_{m_n})_p = \prod_{i=1}^n (v_{k_i}, v_{m_i})$. Then \cite[Lemma 1.4.1 part (e)]{wenzl1998C} gives in our notation, when $T_2$ denotes the flip of the two coordinates on $\textbf{V}^{\otimes 2}$:
\begin{lemma}
    For $u,v \in \textbf{V}^{\otimes 2}$, $(u, Rv)_p = (T_2R^{-1}T_2u, v)_p$, or equivalently $(T_2Ru, Rv)_p = (T_2u, v)_p.$
\end{lemma}
Consider the map $T_4: \textbf{V}^{\otimes 4} \to \textbf{V}^{\otimes 4}$ given by $v_a \otimes v_b \otimes v_c \otimes v_d \mapsto v_d \otimes v_c \otimes v_b \otimes v_a$. It follows from the previous lemma
\begin{corollary}
For $u, v \in \textbf{V}^{\otimes 4}$
$$(T_4\sigma_1 u, \sigma_3 v)_p = (T_4u, v)_p, (T_4\sigma_2u, \sigma_2v)_p = (T_4u, v)_p, (T_4\sigma_3u, \sigma_1v)_p = (T_4u, v)_p.$$
\end{corollary}
\begin{proof}
    We compute
    $$(T_4\sigma_1 (v_{a_1} \otimes v_{a_2} \otimes v_{a_3} \otimes v_{a_4}), \sigma_3 (v_{b_1} \otimes v_{b_2} \otimes v_{b_3} \otimes v_{b_4}))_p =$$
    $$(T_4(R(v_{a_1} \otimes v_{a_2}) \otimes v_{a_3} \otimes v_{a_4}), v_{b_1} \otimes v_{b_2} \otimes R(v_{b_3} \otimes v_{b_4}))_p =$$
    $$(v_{a_4} \otimes v_{a_3} \otimes T_2(R(v_{a_1} \otimes v_{a_2})), v_{b_1} \otimes v_{b_2} \otimes R(v_{b_3} \otimes v_{b_4}))_p =$$
    $$(v_{a_4},v_{b_1})(v_{a_3},v_{b_2})(T_2(R(v_{a_1} \otimes v_{a_2})), R(v_{b_3} \otimes v_{b_4}))_p = $$
    $$(v_{a_4},v_{b_1})(v_{a_3},v_{b_2})(T_2(v_{a_1} \otimes v_{a_2}), v_{b_3} \otimes v_{b_4})_p =$$$$(v_{a_4},v_{b_1})(v_{a_3},v_{b_2})(v_{a_2}, v_{b_3}) (v_{a_1}, v_{b_4}) =$$
    $$(T_4(v_{a_1} \otimes v_{a_2} \otimes v_{a_3} \otimes v_{a_4}), v_{b_1} \otimes v_{b_2} \otimes v_{b_3} \otimes v_{b_4})_p$$
\end{proof}
Considering the skew-linear form on $\textbf{V}^{\otimes 4}$ given by $(u, v) = (T_4u, v)_p$, we have shown that relative to this inner product, $\sigma_1^\ast \sigma_3 = \sigma_2^\ast \sigma_2 = \sigma_3^\ast \sigma_1 = 1$. Moreover, note that under this hermitian form, the spaces $\textbf{V}_{4, \ell}$ are mutually orthogonal, hence the adjoint properties are also true on $\textbf{V}_{4, \ell}$. Since by Theorem \ref{Vn,l irreducible decomposition} $\textbf{V}_{4, \ell}$ is a direct sum of non-isomorphic irreducible representations, it follows that after taking the dual we get a new decomposition into irreducible preserved $B_4$-subspaces. Therefore the decompositions are the same, so the hermitian form also restricts to $\textbf{W}_{4, \ell}.$
\begin{corollary}
\label{preserved hermitian form}
Under the skew-linear product on $\textbf{V}^{\otimes 4}$ given by $(u, v) = (T_4u, v)_p$ the subspace $\textbf{W}_{4,\ell}$ is self-dual. Moreover, $\sigma_1^\ast \sigma_3 = \sigma_2^\ast \sigma_2 = \sigma_3^\ast \sigma_1 = 1.$
\end{corollary}
Let $D_2: \textbf{V}^{\otimes 2} \to \textbf{V}^{\otimes 2}$ denote the linear map $v_m \otimes v_k \mapsto q^{m(m+1) + k(k+1)} v_m \otimes v_k$, and let $T_2: \textbf{V}^{\otimes 2} \to \textbf{V}^{\otimes 2}$ be the linear map $v_m \otimes v_k \mapsto v_k \otimes v_m$. Note that $D_2,T_2$ commute.
\begin{lemma}
$D_2T_2R^{-1}T_2^{-1}D_2^{-1} = \overline{R}$, or equivalently stated $D_2 \reflectbox{R}^{-1}D_2^{-1} = \overline{\reflectbox{R}}$.
\end{lemma}
\begin{proof}
    We wish to prove that $\overline{\reflectbox{R}}D_2 \reflectbox{R}D_2^{-1} = 1$. It suffices to prove this on elements of the form $v_i \otimes v_j$. Applying $\reflectbox{R}D_2^{-1}$, we get
    $$q^{- i(i+1) - j(j+1)} s^{-i-j}\sum_n q^{2(i-n)(j+n)} q^{\frac{n(n-1)}{2}}{ j + n \brack j}_q \prod_{k=0}^{n-1} (sq^{-k-j} - s^{-1}q^{k+j})  v_{i-n} \otimes v_{j+n}$$
    Applying $D_2$ again we get
    $$q^{- i(i+1) - j(j+1)} s^{-i-j}\sum_n q^{2(i-n)(j+n)} q^{\frac{n(n-1)}{2}} q^{(i-n)(i-n+1) + (j+n)(j+n+1)} \cdot $$
    $$\left({ j + n \brack j}_q \prod_{k=0}^{n-1} (sq^{-k-j} - s^{-1}q^{k+j})  v_{i-n} \otimes v_{j+n}\right) = $$
    $$s^{-i-j} \sum_n q^{2ij}q^{\frac{n(n-1)}{2}} {j + n \brack j}_q \prod_{k=0}^{n-1} (sq^{-k-j} - s^{-1}q^{k+j}) v_{i-n} \otimes v_{j + n}.$$
    Applying $\overline{\reflectbox{R}}$, we get
    $$\sum_n q^{2ij} q^{\frac{n(n-1)}{2}} {j + n \brack j}_q \cdot\left(\prod_{k=0}^{n-1} (sq^{-k-j} - s^{-1}q^{k+j})\right)\cdot \sum_m q^{-2(i-n-m)(j+n+m)} q^{-\frac{m(m-1)}{2}} \cdot$$
    $$\left({j + n + m \brack j + n}_{q^{-1}} \prod_{r=0}^{m - 1} (s^{-1}q^{j + n + r} - sq^{-j -n -r}) v_{i-n-m} \otimes v_{j+n+m}\right) = $$
    $$\sum_{n,m} q^{2ij} q^{\frac{n(n-1)}{2}}  q^{-2(i-n-m)(j+n+m)} q^{-\frac{m(m-1)}{2}} \cdot $$
    $$\left(\frac{[j + n + m]_q!}{[j]_q![n]_q![m]_q!} (-1)^m \left(\prod_{k=0}^{n + m -1} (sq^{-k-j} - s^{-1}q^{k+j})\right) v_{i-n-m} \otimes v_{j+n+m}\right) = $$
    Note that for $n=m=0$ we get simply $v_i \otimes v_j$.
    For every $t \geq 1$, we wish to prove that
    $$\sum_{n + m = t} q^{2ij} q^{\frac{n(n-1)}{2}} q^{-2(i-n-m)(j+n+m)} q^{-\frac{m(m-1)}{2}} \frac{[j+n+m]_q!}{[j]_q![n]_q![m]_q!} (-1)^m = 0.$$
    Equivalently
    $$\sum_{n + m = t} (-1)^m q^{2ij} q^{\frac{n(n-1)}{2}} q^{-2(i-t)(j+t)} q^{-\frac{m(m-1)}{2}} {t \brack m}_q = 0$$
    and again equivalently
    $$\sum_{n + m = t} (-1)^m q^{\frac{(t- m)(t - m - 1)}{2}} q^{-\frac{m(m -1)}{2}} {t \brack m}_q = 0$$
    or,
    $$\sum_{m} (-1)^m q^{m(1-t)} {t \brack m}_q = 0.$$
    It is in fact true in greater generality that
    $$\sum_m {t \brack m}_q x^m = \prod_{k=0}^{t-1} (x + q^{1 - t + 2k})$$
    and in particular for $x = -q^{1 - t}$ we get the required equality.
    We prove the last identity by induction on $t$. It is clear for both $t = 0, 1$. We now need to prove that
    $$\sum_m {t + 2 \brack m}_q x^m = (\sum_m {t \brack m}_q x^m)(x + q^{-t - 1})(x + q^{t + 1})$$
or equivalently ${t + 2 \brack m}_q = {t \brack m}_q + {t \brack m - 1}_q(q^{-t-1} + q^{t + 1}) + {t \brack m - 2}_q$
This follows from the $q$-binomial identity
${t + 1 \brack m}_q = q^m {t \brack m}_q + q^{m - t - 1}{t \brack m - 1}_q = q^{-m}{t \brack m}_q + q^{t + 1 - m}{t \brack m - 1}_q$, and so
$${t + 2 \brack m}_q = q^m{t + 1 \brack m}_q + q^{m - t - 2}{t + 1 \brack m - 1}_q = {t \brack m}_q + q^{t + 1}{t \brack m - 1} + q^{-t - 1}{t \brack m - 1}_q + {t \brack m - 2}_q.$$
\end{proof}
Assuming this lemma, consider the maps $D_n: \textbf{V}^{\otimes n} \to \textbf{V}^{\otimes n}$, given by $v_{a_1} \otimes \dots \otimes v_{a_n} \mapsto q^{\sum_{i=1}^n a_i(a_i+1)} v_{a_1} \otimes \dots \otimes v_{a_n}$, and let $T_n : \textbf{V}^{\otimes n} \to \textbf{V}^{\otimes n}$ be the map sending $v_{a_1} \otimes \dots \otimes v_{a_n} \mapsto v_{a_n} \otimes \dots \otimes v_{a_1}$. Note that $D_n,T_n$ commute. Then
\begin{proposition}
\label{Inverse conjugate to bar}
    $D_nT_n\sigma_i^{-1}T_n^{-1}D_n^{-1} = \overline{\sigma_{n - i}}$
\end{proposition}
\begin{proof}
    $$D_nT_n\sigma_1^{-1}T_n^{-1}D_n^{-1}(v_{a_1} \otimes \dots \otimes v_{a_n}) = q^{-\sum a_i(a_i + 1)} D_nT_n\sigma_1^{-1}(v_{a_n} \otimes \dots \otimes v_{a_1}) =$$
    $$q^{-\sum a_i(a_i + 1)} D_nT_n(R^{-1}(v_{a_n} \otimes v_{a_{n - 1}}) \otimes v_{a_{n-2}} \otimes \dots \otimes v_{a_1}) =$$
    $$ q^{-\sum a_i(a_i + 1)} D_n(v_{a_1} \otimes \dots \otimes v_{a_{n-2}} \otimes T_2R^{-1}(v_{a_n} \otimes v_{a_{n-1}}))=$$
    $$q^{-\sum_{i \leq n - 2} a_i(a_i + 1)}D_n(v_{a_1} \otimes \dots \otimes v_{a_{n-2}} \otimes T_2R^{-1}T_2^{-1}D_2^{-1}(v_{a_{n-1}} \otimes v_{a_n})) =$$
    $$q^{-\sum_{i \leq n - 2} a_i(a_i + 1)}D_n(v_{a_1} \otimes \dots \otimes v_{a_{n-2}} \otimes D_2^{-1}\overline{R}(v_{a_{n-1}} \otimes v_{a_n})) =$$
    Note that $q^{-\sum_{i \leq n - 2} a_i(a_i + 1)}D_n(v_{a_1} \otimes \dots \otimes v_{a_{n - 2}} \otimes  D_2^{-1}(v_m \otimes v_k) ) = v_{a_1} \otimes \dots \otimes v_{a_{n - 2}} \otimes v_m \otimes v_k $, and therefore the last term is simply $v_{a_1} \otimes \dots \otimes v_{a_{n-2}} \otimes \overline{R}(v_{a_{n-1}} \otimes v_{a_n}) = \overline{\sigma_{n-1}}(v_{a_1} \otimes \dots \otimes v_{a_n})$.
\end{proof}
Note that the hermitian form in Corollary \ref{preserved hermitian form} has a corresponding bilinear form, and we may take transpose with respect to this bilinear form. Now using Corollary \ref{preserved hermitian form} and Proposition \ref{Inverse conjugate to bar} we obtain:
\begin{corollary}
\label{Extension of representation to AutF2}
There is a matrix $J \in \mathrm{PSL}(\textbf{V}_{4,\ell})$ such that $J\sigma_i^TJ^{-1} = \sigma_i$.
Therefore the automorphism $\varphi$ of $\mathrm{PSL}(\textbf{V}_{4,\ell})$ given by $A \mapsto J(A^T)^{-1}J^{-1}$ satisfies $\varphi \sigma_i \varphi^{-1} = \sigma_i^{-1}$. 
Moreover, $\varphi$ preserves $\mathrm{PSL}(\textbf{W}_{4, \ell})$, and $\varphi|_{\mathrm{PSL}(\textbf{W}_{4, \ell})}$ is an involution.
\end{corollary}
\begin{proof}
By Corollary \ref{preserved hermitian form}, if we let $H$ denote the matrix of the inner product, then $\sigma_i H \sigma_{4-i}^\ast = H$, therefore $\overline{\sigma_{4-i}}^T = \sigma_{4-i}^\ast = H^{-1}\sigma_i^{-1}H$. Moreover, by Proposition \ref{Inverse conjugate to bar}, there is a matrix $L$ such that $\overline{\sigma_{4-i}} = L\sigma_i^{-1}L^{-1}$. Therefore $(L\sigma_i^{-1}L^{-1})^T = H^{-1}\sigma_i^{-1} H$, or equivalently $(L^{-1})^T\sigma_i^TL^T = H^{-1}\sigma_i H$. So, if $J = H(L^{-1})^T$, then $J \sigma_i^T J^{-1} = \sigma_i.$

Denote $\varphi(A) = J(A^T)^{-1}J^{-1}.$ Then $\varphi(\sigma_i) = J (\sigma_i^T)^{-1} J^{-1} = \sigma_i^{-1}.$

Now, it follows that $\sigma_1, \sigma_2, \sigma_3$ are symmetric with respect to the bilinear form given by $\langle x,y\rangle  = (x, J^{-1}y)$, since $$\langle\sigma_ix, y\rangle =(\sigma_ix, J^{-1}y) = (x, \sigma_i^TJ^{-1}y) = (x, J^{-1}\sigma_i y) = \langle x, \sigma_iy \rangle.$$Using the fact that the irreducible components of the $B_4$ representation on $\textbf{V}_{4, \ell}$ have different dimensions, it follows that each of the preserved subspaces is self-dual under this bilinear form. Indeed, it suffices to prove that every two of the subspaces are orthogonal with respect to the bilinear form. Ordering the subspaces by increasing order of their dimensions, the orthogonal subspace to the smallest dimensional subspace is also preserved, and is therefore a direct sum of a subset of the spaces. The only subset whose direct sum has the correct dimension is the complementary set of subspaces, and therefore the minimal dimensional subspace is orthogonal to the remaining spaces. Continuing this argument by order of the spaces, we get that every two of the spaces are orthogonal, and each of the spaces are self-dual.

Finally, $\varphi^2$ is an automorphism of $\mathrm{PSL}(\textbf{W}_{4,\ell})$ commuting with the $B_4$-action, and therefore commuting with every matrix in $\mathrm{PSL}(\textbf{W}_{4,\ell})$ by Zariski density. Therefore, $\varphi^2$ is the identity.
\end{proof}
\begin{remark}
    The map $\sigma_i \mapsto \sigma_i^{-1}$ yields an outer automorphism of $B_n$ for every $n$. The proof in this section in fact gives an extension of the $B_n$-representations to this index-$2$ super-group for every $n$. However for our required application the case $n = 4$ is sufficient.
\end{remark}
\subsection{Proof of Theorem \ref{Large rank quotients}}
We may now prove Theorem \ref{Large rank quotients} using the strong approximation theorem of Weisfeiler \cite{weisfeiler1984strong} and Pink \cite{pink2000strong}, see also \cite[Window 9, Corollary 3]{lubotzky2012subgroup}, \cite[Theorem 3.2]{chen2025finite}.
\begin{theorem}[Strong Approximation Theorem]
    Let $K$ be a number field and let $\textbf{G}$ be a simple and simply connected linear algebraic group over $K$. Let $\Lambda$ be a finitely generated Zariski-dense subgroup of $\textbf{G}(K)$. Assume that the field generated by the traces $\mathbb{Q}((\mathrm{tr}(\mathrm{Ad}(\gamma)))_{\gamma\in\Lambda}) = K,$ where $\mathrm{Ad}$ is the adjoint representation of $G.$ Then there is a finite set of places $S$ of $K$ such that for every $\mathfrak{p} \notin S$, $\pi_\mathfrak{p}(\Lambda) = \textbf{G}(\mathbb{F}_\mathfrak{p})$, where $\pi_\mathfrak{p}$ is the reduction mod $\mathfrak{p}$ map.
\end{theorem}
We remark that if $K = \mathbb{Q}$ the condition on the traces is vacuous.

We will need the following result of Larsen and Lubotzky, see \cite[Lemma 4.4]{larsen2011normal}.
\begin{lemma}
\label{Open subset of good specializations}
Let $A$ be an integral domain, finitely generated over $\mathbb{Q}$, with fraction field $K$. Let $\textbf{G}$ be a connected and absolutely simple linear algebraic group over $K$. Let $\Gamma \leq \mathrm{GL}_n(A)$ and assume that its Zariski-closure over $K$ is $\textbf{G}.$

Denote by $\mathcal{G}$ the Zariski closure of $\Gamma$ over $A$, which has generic fiber $\textbf{G}$. Then there is an open subset $U \subseteq \mathrm{spec}(A)$ such that for every closed point $x \in U$ the image of $\Gamma$ in $\mathcal{G}(k_x)$ is either finite or Zariski-dense, where $k_x$ is the residue field of $\mathrm{spec}(A)$ at $x.$
\end{lemma}
We will also need the following lemma
\begin{lemma}
\label{Eigenvalues of sigma1sigma3inv}
Consider the action of $\sigma_1\sigma_3^{-1}$ on $\textbf{W}_{4, \ell}$, for $\ell \geq 1.$ Then all eigenvalues of this matrix have the form $\pm q^as^b$ for some $(a, b) \in \mathbb{Z}^2$, and furthermore these eigenvalues are not all equal up to sign.
\end{lemma}
\begin{proof}
By Lemma \ref{Eigenvalue computation} and Theorem \ref{Irreducible Decomposition of Bn-1}, it follows that $\sigma_3$ is diagonalizable, and every eigenvalue of $\sigma_3$ is of the form $\pm q^as^b$ for some $a, b \in \mathbb{Z}$. As $\sigma_1, \sigma_3$ are conjugate in $B_4$, the same is true of $\sigma_1$. Moreover, $\sigma_1,\sigma_3 \in B_4$ commute, and therefore so do their actions on $\textbf{W}_{4, \ell}$. Hence, $\sigma_3$ preserves the eigenspaces of $\sigma_1$. Choosing a basis of eigenvectors of $\sigma_3$ in its action on each eigenspace of $\sigma_1$, it follows that $\sigma_1\sigma_3^{-1}$ is diagonalizable, and that all eigenvalues of $\sigma_1\sigma_3^{-1}$ are quotients of an eigenvalue of $\sigma_1$ and an eigenvalue of $\sigma_3$ and hence of the required form.

Assume by contradiction that all eigenvalues of $\sigma_1\sigma_3^{-1}$ are equal up to sign to the same constant $\lambda.$ These eigenvalues are quotients of eigenvalues of $\sigma_1, \sigma_3$ and therefore by Lemma \ref{Eigenvalue computation} have the form
$$\frac{q^{\ell_1(\ell_1-1)}\cdot(-s^{-2})^{\ell_1}}{q^{\ell_3(\ell_3-1)}\cdot(-s^{-2})^{\ell_3}},$$
for some values of $0 \leq \ell_1, \ell_3 \leq \ell.$ If two such eigenvalues are equal up to sign, then by equating the powers of $s,$ we see that the two values of $\ell_1 - \ell_3$ are equal. Therefore the signs are also equal, and hence the eigenvalues are in fact equal.

Since $\sigma_1 \sigma_3^{-1}$ is diagonalizable, $\sigma_1 = \lambda \sigma_3$ under the action on $\textbf{W}_{4, \ell},$ for the unique eigenvalue $\lambda$ of $\sigma_1\sigma_3^{-1}.$ As this is a braid group representation, it follows that
$$\sigma_1 \sigma_2 \sigma_1 = \lambda^2 \sigma_3 \sigma_2 \sigma_3 = \lambda^2 \sigma_2 \sigma_3 \sigma_2 = \lambda \sigma_2 \sigma_1 \sigma_2 = \lambda \sigma_1 \sigma_2 \sigma_1,$$
and therefore $\lambda = 1.$

Therefore, our representation of $B_4$ is in fact a representation of $B_3$ factoring through the map $B_4 \to B_3$ sending $\sigma_1 \mapsto \sigma_1, \sigma_2\mapsto \sigma_2, \sigma_3 \mapsto \sigma_1$. In particular, $\textbf{W}_{4, \ell}$ is an irreducible representation of $B_3$. However, by Theorem \ref{Irreducible Decomposition of Bn-1}, we know that $\textbf{W}_{4,\ell} \simeq \bigoplus_{k=0}^\ell \textbf{W}_{3, k}$ is reducible as a $B_3$-representation when $\ell \geq 1,$ in contradiction.
\end{proof}
Using these results, we prove Theorem \ref{Large rank quotients}.
\begin{proof}[Proof of Theorem \ref{Large rank quotients}]
Let $\ell \geq 1$ and $N = \binom{\ell + 2}{2}$. By Theorem \ref{Knizhinik Zamolodichikov Zariski density}, $B_4 \to \mathrm{PSL}_N(\bar{\mathbb{K}})$ has Zariski-dense image. By Lemma \ref{Eigenvalues of sigma1sigma3inv}, the image of $F_2$ is non-trivial, and therefore as the subgroup $F_2 \lhd B_4$ is normal, it has Zariski-dense image. Let $\pm q^as^b \neq \pm 1$ be a quotient of two eigenvalues of $\sigma_1\sigma_3^{-1}$ that are not equal up to sign.

Since $F_2$ is also contained in the derived subgroup of $B_4$, then the map $B_4 \to \mathrm{PGL}_N(\mathbb{K})$ restricts to a map $F_2 \to \mathrm{PSL}_N(\mathbb{K})$, with Zariski dense image. Moreover, the coefficients of every matrix in the image is contained in the finitely generated integral domain $\mathbb{Q}[q^{\pm 1}, s^{\pm 1}]$ over $\mathbb{Q}$.

Choose an open subset $U \subseteq \mathrm{spec}(\mathbb{Q}[q^{\pm 1}, s^{\pm 1}])$ as in Lemma \ref{Open subset of good specializations}, such that at every closed point $x \in U$, the image of $F_2$ in $\mathrm{PGL}_N(k_x)$ is either finite or Zariski-dense. Moreover, by Corollary \ref{Extension of representation to AutF2}, we get an automorphism $\varphi \in \mathrm{Aut}(\mathrm{PSL}_N(\mathbb{K}))$ such that $\varphi\sigma_i\varphi^{-1} = \sigma_i^{-1}$, $\varphi^2 = 1$. Let $V \subseteq \mathrm{spec}(\mathbb{Q}[q^{\pm 1}, s^{\pm 1}])$ be an open subset where this automorphism is defined, obtained by inverting all denominators appearing in $\varphi, \varphi^{-1}$.

Pick a point $(q_0, s_0) \in \mathbb{Q}^2 \cap U \cap V,$ and further assume the open algebraic condition $q_0^as_0^b \neq \pm 1$. The image of $F_2$ at the residue field at $(q_0, s_0)$ is not finite - as $q_0^as_0^b \neq \pm 1$ is rational, and therefore not a root of unity. Hence, the image of $F_2$ is Zariski-dense.

After this specialization, we get a map $B_4 \to \mathrm{GL}_N(\mathbb{Q})$, such that the image of $F_2$ is contained in $\mathrm{SL}_N(\mathbb{Q})$, and $F_2$ has Zariski dense image in $\mathrm{PGL}_N(\mathbb{Q})$. Denote the image of $F_2$ by $\Gamma \leq \mathrm{SL}_N(\mathbb{Q}),$ and let $H$ denote its Zariski-closure in $\mathrm{SL}_N.$ The image of $H$ under the projection $\pi:\mathrm{SL}_N \to \mathrm{PGL}_N$ is all of $\mathrm{PGL}_N$. Since $\pi$ has finite fibers, $\dim(H)=\mathrm{dim}(\mathrm{PGL}_N)=\mathrm{dim}(\mathrm{SL}_N)$. However, as $\mathrm{SL}_N$ is simple, it has no proper closed subgroups of the same dimension, implying that $H = \mathrm{SL}_N$.

By the Strong Approximation Theorem, it follows that there are infinitely many specializations of the representation of $F_2$ to prime fields that map onto $\mathrm{SL}_N(\mathbb{F}_p)$ and in particular map onto $\mathrm{PSL}_N(\mathbb{F}_p).$

We have now shown that there are infinitely many specializations to finite fields such that $F_2 \to \mathrm{PSL}_N(\mathbb{F}_p)$ is surjective. This map extends to a map $B_4 \to \mathrm{PGL}_N(\mathbb{F}_p)$, and moreover as $(q_0, s_0) \in V$ the automorphism $\varphi$ restricts to an automorphism $\varphi \in \mathrm{Aut}(\mathrm{PSL}_N(\mathbb{Q}))$ such that $\varphi\sigma_i\varphi^{-1} = \sigma_i^{-1}$, $\varphi^2 = 1$. Further choosing the prime $p$ such that it does not appear in any of the denominators of $\varphi, \varphi^{-1}$, we may assume that $\varphi$ restricts to an automorphism in $\mathrm{Aut}(\mathrm{PSL}_N(\mathbb{F}_p)).$ By Lemma \ref{semi-direct product AutF2} this gives a map $\mathrm{Aut}(F_2) \to \mathrm{Aut}(\mathrm{PSL}_N(\mathbb{F}_p)).$ It now follows that the kernel of the map $F_2 \to \mathrm{PSL}_N(\mathbb{F}_p)$ is characteristic.
\end{proof}
\begin{remark}
\label{specializations to number fields}
Specializing to $\mathbb{Q}$ rather than a general number field $K$ allowed us to avoid the technical difficulty of proving that the field generated by the traces is $K.$

However - given a number field $K$, specializing to generic values $q_0, s_0 \in K$, the field generated by the traces will be all of $K.$ Indeed, choose some element of $F_2$ (for example $\sigma_1 \sigma_3^{-1}$) whose trace under the adjoint representation is a non-constant Laurent polynomial in $q, s$. The condition that the evaluation of this polynomial at $(q_0, s_0)$ generates the field $K$ is a genericity condition on $q_0, s_0$, as almost every element of $K$ generates $K$ as a field over $\mathbb{Q}$.

Choosing $(q, s)$ to satisfy this genericity condition, as well as $(q_0, s_0) \in U$ and that $q_0^as_0^b$ is not a root of unity, we will obtain a Zariski-dense representation $F_2 \to \mathrm{PGL}_N(K)$ satisfying the assumptions of the strong approximation theorem. Using the same argument as in the proof of Theorem \ref{Large rank quotients}, we would get characteristic quotients of the form $F_2 \to \mathrm{PSL}_N(\mathbb{F}_r)$ for prime powers $r$, by considering specializations to the residue fields of sufficiently large prime ideals of $K.$
\end{remark}

\section{Residual Finite Almost-Simplicity}
\label{Section on residual finite almost-simplicity}
Our goal in this section is to explain Remark \ref{AutFn not residually finite simple} and prove Theorem \ref{Fully Residually Finite Simple}.

\begin{proof}[Explanation of Remark \ref{AutFn not residually finite simple}]
Let $\varphi: \mathrm{Aut}(F_n) \to A$ be a map to a finite almost-simple group. Assume that $S \leq A \leq \mathrm{Aut}(S)$ for a finite simple $S$. In particular, $S \lhd A$.

Schreier's conjecture proved by the classification of finite simple groups states that the outer automorphism group $\mathrm{Out}(S)$ is solvable. As $\mathrm{Aut}^+(F_n)$ is perfect for $n \geq 3$, it follows that the inner automorphisms $\mathrm{Inn}(F_n) \leq \mathrm{Aut}^+(F_n)$ are mapped trivially to the solvable group $A / S.$ Hence, the image of $\mathrm{Inn}(F_n)$ is contained in $S$. Since $F_n \simeq \mathrm{Inn}(F_n) \lhd \mathrm{Aut}(F_n)$, the image of $\mathrm{Inn}(F_n)$ is normalized by $A$, and in particular by $S$. As $S$ is simple, the image is either trivial or all of $S$.

In the latter case, the kernel of the map $F_n \to S$ is a characteristic subgroup of $F_n$, in contradiction to the Baby Wiegold Conjecture.
\end{proof}
We now turn to the proof of Theorem \ref{Fully Residually Finite Simple}. We will use the following two facts about the representation of $B_n$ on $\textbf{W}_{n, 2}$. The first being \cite[Theorem 2]{jackson2011lawrence} that the representation $\textbf{W}_{n, 2}$ is isomorphic to the Lawrence--Krammer--Bigelow representation of $B_n$, and the second being the fact that these representations are faithful \cite{bigelow2001braid}.
\begin{proof}[Proof of Theorem \ref{Fully Residually Finite Simple}]
Consider the $6$-dimensional representation of $B_4$ on $\textbf{W}_{4, 2}$. We know that this representation is Zariski-dense by Theorem \ref{Knizhinik Zamolodichikov Zariski density}. As in the proof of Theorem \ref{Large rank quotients}, we obtain an open subset $U \subseteq \mathrm{spec}(\mathbb{Q}[q^{\pm 1}, s^{\pm 1}])$ such that closed points $x \in \mathbb{Q}^2 \cap U$ have Zariski-dense image in $\mathrm{PGL}_6(\mathbb{Q})$. Moreover, specializing at sufficiently large primes then yields characteristic quotients $F_2 \to \mathrm{PSL}_6(\mathbb{F}_p).$

By Bigelow's theorem, the representation is faithful. Note that this implies that the projective representation to $\mathrm{PGL}(\textbf{W}_{4,2})$ is a faithful representation of $\mathrm{Aut}^+(F_2)$. Indeed, if $b \in B_4$ has image in $Z(\mathrm{GL}(\textbf{W}_{4,2}))$, then it commutes with every element of $B_4$, and so is in $Z(B_4)$.

Choose a finite subset $F \subseteq \mathrm{Aut}^+(F_2)$. Then the images in $\mathrm{PGL}(\textbf{W}_{4,2})$ are distinct. Using the fact that distinct polynomials have different values when evaluated at sufficiently large points, we can find a rational specialization such that $F$ is mapped injectively. Reducing modulo a sufficiently large prime $p$, $F$ would still be mapped injectively.

Now, choose a finite subset $F \subseteq \mathrm{Aut}(F_2)$. Under the extension given by Corollary \ref{Extension of representation to AutF2}, elements of the non-trivial coset of $\mathrm{Aut}(F_2)$ act as outer automorphisms of $\mathrm{PGL}(\textbf{W}_{4,2})$, and the same is true after the specialization to $\mathbb{F}_p$. Therefore, no two elements of different cosets of $\mathrm{Aut}^+(F_2)$ will have the same image in $\mathrm{Aut}(\mathrm{PSL}_6(\mathbb{F}_p))$.

Fix an element $\alpha$ in the non-trivial coset of $\mathrm{Aut}^+(F_2)$. Let $F = A \cup B \alpha$ for $A, B \subseteq \mathrm{Aut}^+(F_2)$ be the decomposition of $F$ into the two cosets. Consider $\tilde{F} = A \cup B \subseteq \mathrm{Aut}^+(F_2)$, and choose the specialization to be injective on $\tilde{F}$. Then the restriction to $F$ of the map $\mathrm{Aut}(F_2) \to \mathrm{Aut}(\mathrm{PSL}_6(\mathbb{F}_p))$ is injective.
\end{proof}
\section{Some remarks for future research}
\subsection{Possible prime powers}
Our Zariski-dense representations have coefficients in the ring $\mathbb{Z}[q^{\pm 1}, s^{\pm 1}]$, so using similar ideas we could in fact prove a stronger result about the set of prime powers $r$ such that $\mathrm{PSL}_N(\mathbb{F}_r)$ is a characteristic quotient. 
Using similar ideas to Lemma \ref{Open subset of good specializations} we should be able to prove that for almost every prime $p$ the specialization to $\mathbb{F}_p[q^{\pm 1}, s^{\pm 1}]$ is Zariski-dense. Then, using a version of the strong approximation theorem in positive characteristic, we could get almost every power of $p$ as a quotient. Pushing these ideas further it seems likely that we can obtain all but finitely many prime powers as quotients.

Using the preserved hermitian product in Corollary \ref{preserved hermitian form}, it should also be possible to get quotients of the form $\mathrm{PSU}_{\binom{n}{2}}(\mathbb{F}_r)$, in the same way that the $\mathrm{PSU}_3(\mathbb{F}_r)$-quotients are obtained in \cite{chen2025finite}.

\subsection{Induced Cayley graphs}
Using ideas of Bourgain and Gamburd, a generalization of the strong approximation theorem was proved in \cite{golsefidy2012expansion}, sometimes called the superstrong approximation theorem.

Given a set of generators $S$ of a Zariski-dense subgroup $\Gamma$ of a simply connected, simple algebraic group $\textbf{G}$ over $\mathbb{Q}$, the strong-approximation states that for sufficiently large primes $p$, the images of $S$ generate $\textbf{G}(\mathbb{F}_p).$ Equivalently stated, the Cayley graph of $\textbf{G}(\mathbb{F}_p)$ with respect to the image of $S$ is connected. The superstrong approximation theorem states that this family of Cayley graphs is not only connected for sufficiently large primes $p$, they moreover form a family of expanders. This can be thought of as the statement that they are connected in a uniformly quantitative sense.

In our case, this means that given a fixed Zariski-dense specialization to $\mathbb{Q}$, the family of Cayley graphs of the characteristic quotients $F_2 \to \mathrm{PSL}_{\binom{n}{2}}(\mathbb{F}_p)$ given by specialization to primes $p$, forms a family of expanders.

Two natural questions therefore arise
\begin{question}
\begin{enumerate}
    \item Consider the full family of quotients arising from specializations of the Knizhnik--Zamolodchikov representations as in Theorem \ref{Large rank quotients} (see also Remark \ref{specializations to number fields}), i.e. allowing both $n, r$ to range over their possible values in a characteristic quotient $F_2 \to \mathrm{PSL}_{\binom{n}{2}}(\mathbb{F}_r).$ Do the corresponding Cayley graphs form a family of expanders?
    \item Is the family of Cayley graphs corresponding to the alternating characteristic quotients from Theorem \ref{Alternating quotients} a family of expanders?
\end{enumerate}
\end{question}

Another natural question arises regarding these Cayley graphs. Since these arise from a family of quotients of the free group, there are no natural relations on our generators. Therefore, it is possible for this family of graphs to have girth tending to infinity. In fact, Theorem \ref{Fully Residually Finite Simple} shows that the quotients of the form $\mathrm{PSL}_6(\mathbb{F}_r)$ can indeed have girth tending to infinity.
It is then natural to ask, how fast do these girths grow to infinity?
\begin{question}
What is the order of magnitude of the girth of the constructed Cayley graphs?
\end{question}
The question of the maximal girth of Cayley graph of a finite simple group is an interesting open problem, especially in the case of the alternating groups, see for instance \cite{gamburd2009girth}, \cite{eberhard2017trivial}. These characteristic quotients are a natural source of examples, as any short cycle in such a graph would then yield many more short cycles after applying automorphisms of the free group.

A first interesting result to prove would be that these alternating quotients have girth that tends to infinity, also implying that $\mathrm{Aut}(F_2)$ is residually finite alternating or symmetric. We believe that this should be provable assuming the following conjecture
\begin{conjecture}
\label{Transitive action}
    Under Assumption \ref{Assumption on non-conjugation}, the action of $B_4$ on $\mathrm{X}^{(2)}_{\gamma,\delta}$ is transitive.
\end{conjecture}
In fact, a much weaker result should be sufficient - merely that most (a subset of density tending to 1) of the points of $X_{\gamma,\delta}^{(2)}$ are on the same orbit. Using the fact that the map $B_4 \to \mathrm{Aut}(F_4)$ induced by the equivariant quandle is injective, it should then be possible to prove that for generic values of $\gamma,\delta$ we can distinguish finite subsets of elements of $B_4$ by their actions on $X_{\gamma,\delta}^{(2)}$. The algebraic description for the variety given in the Appendix could be helpful in understanding these claims.

A possible approach towards Conjecture \ref{Transitive action} is to try and use similar ideas to \cite{bourgain2016markoff}. The set $X_{\gamma,\delta}^{(2)}$ can be given the structure of a $4$-dimensional variety, and has $2$-dimensional foliations arising from the proper decompositions. These can be used in a similar fashion to the conic sections used in \cite{bourgain2016markoff}. In fact, the weaker claim that most of the points are on a single orbit is also a step in the proof in \cite{bourgain2016markoff}, which they call "The Endgame". The relevant subset of points $(A,B,C,D)$ we could consider (and try to prove that they are all on the same orbit) are those such that either the first or second proper decompositions are maximal in the sense of Definition \ref{proper decomposition definition}.

\subsection{Residual Finite Simplicity}
We would expect the following strengthened version of Theorem \ref{Fully Residually Finite Simple} to hold:
\begin{conjecture}
$\mathrm{Aut}(F_2)$ is fully residually finite simple
\end{conjecture}
The Lie-type quotients constructed both in this paper and in \cite{chen2025finite} are insufficient in proving this result, as the extension from $\mathrm{Aut}^+(F_2)$ to $\mathrm{Aut}(F_2)$ will necessarily act via an outer automorphism of the finite simple quotient.
However, we believe that a more careful analysis of the alternating quotients given by Theorem \ref{Alternating quotients} might prove this conjecture. For this it would be necessary to compute the signs of all permutations involved, and hence it would be necessary to identify precisely (the size and structure of) the orbits of the actions. Proving Conjecture \ref{Transitive action} would be a step in this direction.

\appendix
\section{Description of the algebraic variety}
In this Appendix we will give an algebraic structure for the finite sets $X^{(2)}_{\gamma,\delta}$, and show that under a small assumption on $\gamma,\delta$ the group $B_4$ acts by algebraic automorphisms of this algebraic variety. This could be used in further research for example in order to prove Conjecture \ref{Transitive action}.

We make the following extra assumption in addition to Assumption \ref{Assumption on non-conjugation} in order to simplify the discussion, but it could be relaxed.
\begin{assumption}
\label{Assumption on large orders}
    Assume that either $\mathrm{ord}(\gamma)> 60$ or $\mathrm{ord}(\delta) > 60$.
\end{assumption}
\begin{remark}
    Using the description of maximal subgroups of $\mathrm{PSL}_2(\mathbb{F}_p)$, c.f. \cite{hall1936eulerian}, every maximal subgroup of sufficiently large order (greater than $60$) is either cyclic, dihedral or contained in a maximal parabolic of $\mathrm{PSL}_2(\mathbb{F}_p)$ in which case it does not contain any non-split elements. Therefore under Assumption \ref{Assumption on non-conjugation} if the orders of either $\gamma$ or $\delta$ are sufficiently large, then every pair of conjugates of $\gamma,\delta$ generates all of $\mathrm{PSL}_2(\mathbb{F}_p)$.
\end{remark}
\begin{corollary}
\label{Conjugates Generate}
    Under Assumptions \ref{Assumption on non-conjugation}, \ref{Assumption on large orders}, any pair of conjugates of $\gamma,\delta$ generate $\mathrm{PSL}_2(\mathbb{F}_p)$.
\end{corollary}

We also need the following Lemma with a similar proof to Lemma \ref{Pairs up to conjugation}:
\begin{lemma}
\label{Triples up to conjugation}
    Let $A,B,C \in \mathrm{SL}_2(\mathbb{F}_p)$ be three matrices with images generating $\mathrm{PSL}_2(\mathbb{F}_p)$. Then every triple of matrices $\tilde{A},\tilde{B},\tilde{C} \in \mathrm{SL}_2(\mathbb{F}_p)$ yielding the same character in $\mathrm{Ch}_{\mathrm{SL}_2(\mathbb{F}_p)}(F_3)$ is conjugate under $\mathrm{GL}_2(\mathbb{F}_p)$ to $A,B,C$.
\end{lemma}
\begin{proof}
    We will closely follow the strategy of proof of Lemma \ref{Pairs up to conjugation}. We note that the two paragraphs at the end of the proof of Lemma \ref{Pairs up to conjugation} imply that it suffices to prove that such a conjugating matrix exists in $\mathrm{GL}_2(\mathbb{F}_{p^2}).$

    Firstly, assume as before that $\pm A$ is non-unipotent, and conjugate $A$ to the form $\begin{pmatrix}\lambda &  0 \\ 0 & \lambda^{-1}\end{pmatrix}.$ Now, knowing both $\mathrm{tr}(B),\mathrm{tr}(AB)$ we get as before the diagonal entries of $B$, and similarly for $C$ and for $BC$. The diagonal entries of $BC$ are $B_{11}C_{11} + B_{12}C_{21}$ and $B_{22}C_{22} + B_{21}C_{12}$ and so both $B_{12}C_{21}$ and $B_{21}C_{12}$ are fixed. Moreover, $B_{12}B_{21}$ and $C_{12}C_{21}$ are fixed as $\det(B)=\det(C)=1$. Conjugating by a matrix $\begin{pmatrix}\mu & 0 \\ 0 & 1\end{pmatrix}$, and assuming that none of $B_{12},B_{21},C_{12},C_{21}$ are zero, gives us an arbitrary choice of the off-diagonal entries with all the fixed product values. Moreover, if at most one of the off-diagonal values is zero, the argument still works (there are two product conditions and one degree of freedom $\mu$ for three variables). So, at least two are zero. If they are both on the same matrix, then two of our matrices are diagonal and we may finish via a completely analogous argument to the previous lemma. If both of the matrices are upper triangular or both are lower triangular, once again they generate a solvable subgroup in contradiction. So, one of the matrices is upper triangular and the other is lower triangular. Now, the fixed product value and the $\mu$-degree of freedom allows us to fix all the values.

    We are now left with the case that $\pm A,\pm B,\pm C$ are all unipotent. In such a case, either they are all in the same unipotent subgroup, and therefore generate a cyclic subgroup of $\mathrm{PSL}_2$, or some pair is in distinct unipotent subgroups. Assume that these are $A,B$, and conjugate them to the form $\varepsilon_A\begin{pmatrix}1 & t \\ 0 & 1\end{pmatrix},\varepsilon_B\begin{pmatrix}1 & 0 \\ s & 1\end{pmatrix}$, and $t,s\neq 0$ as they are in distinct unipotent subgroups. As before, $ts$ is fixed by $\mathrm{tr}(AB)$, and so up to conjugation we may fix the values of $t,s$. Now, for any matrix $C$,
    $$\mathrm{tr}(A^2C) = C_{11} + 2tC_{21} + C_{22} = \mathrm{tr}(C) + 2tC_{21}$$ and $$\mathrm{tr}(B^2C) = \mathrm{tr}(C) + 2sC_{12}.$$ Therefore the off-diagonal entries are uniquely defined, and $$\mathrm{tr}(A^2B^2C) = \mathrm{tr}\left(\begin{pmatrix}1 + 4ts &  2t \\ 2s & 1\end{pmatrix}C\right) = (1+4ts)C_{11} + 2tC_{21} + 2sC_{12} + C_{22} = \mathrm{tr}(C) + 2tC_{21} + 2sC_{12} + 4tsC_{11}$$ so $C_{11}$ is fixed. Similarly $\mathrm{tr}(B^2A^2C)$ fixes the value of $C_{22}$, proving our claim.
\end{proof}

Given a homomorphism $F_n \to \mathrm{PSL}_2(\mathbb{F}_p)$, let $\mathrm{Ch}_{\mathrm{PSL}_2(\mathbb{F}_p)}(F_n)$ be the set of $2^n$ possible characters of lifts to homomorphisms $F_n \to \mathrm{SL}_2(\mathbb{F}_p)$. So, two $\mathrm{PSL}_2(\mathbb{F}_p)$-characters are equal if and only if there are some two $\mathrm{SL}_2(\mathbb{F}_p)$-lifts with the same character. Therefore the previous Lemmas \ref{Pairs up to conjugation}, \ref{Triples up to conjugation} imply that for a generating set of $2$ or $3$ matrices the $\mathrm{PSL}_2(\mathbb{F}_p)$-character uniquely defines the matrices up to simultaneous $\mathrm{PGL}_2(\mathbb{F}_p)$-conjugation.
Now, using these Lemmas, we obtain:
\begin{proposition}
\label{Characterization of the algebraic variety}
Let $\gamma,\delta \in \mathrm{SL}_2(\mathbb{F}_p)$ satisfy Assumptions \ref{Assumption on non-conjugation}, \ref{Assumption on large orders}.
Let $$\varphi:\tilde{X}^{(2)}_{\gamma,\delta} \to \mathrm{Ch}_{\mathrm{PSL_2}(\mathbb{F}_p)}(F_3),\ (A,B,C,D) \mapsto (B^{-1}A, A^{-1}C, D^{-1}C).$$ Then,
\begin{enumerate}
    \item $\varphi$ is well defined as a map from $X_{\gamma,\delta}^{(2)}$, and as such is injective.
    \item The image of $\varphi$ is precisely those triples $(M_1,M_2,M_3) \in \mathrm{PSL}_2(\mathbb{F}_p)^3$ that have lifts to $\mathrm{SL}_2(\mathbb{F}_p)$ satisfying $\mathrm{tr}(M_1M_3)=\mathrm{tr}(\delta)$ and $\mathrm{tr}(M_1M_2M_3M_2^{-1})=\mathrm{tr}(\gamma)$.
\end{enumerate}
\end{proposition}
\begin{proof}
Throughout the proof we will use $\gamma, \delta$ to denote the elements of $\mathrm{SL}_2(\mathbb{F}_p)$ and $\bar{\gamma}, \bar{\delta}$ to denote the corresponding elements of $\mathrm{PSL}_2(\mathbb{F}_p).$
    \begin{enumerate}
    \item We start by proving well-definiteness. If we multiply all matrices on the left by an element of $C_{\mathrm{PGL}_2(\mathbb{F}_p)}(\bar{\gamma}),$ then the map is preserved (The elements may no longer be in $\mathrm{PSL}_2(\mathbb{F}_p),$ but the map is still well-defined, and its image is preserved). Moreover, if we multiply all elements on the right by an element of $C_{\mathrm{PGL}_2(\mathbb{F}_p)}(\bar{\delta})$, then this simultaneously conjugates all three matrices in the image, resulting in the same point on the character variety.

    We now note that whenever $(A, B, C, D) \in \tilde{X}_{\gamma, \delta}^{(2)}$, then $(B^{-1}A, A^{-1}C, D^{-1}C)$ generates $\mathrm{PSL}_2(\mathbb{F}_p).$ Indeed,
    $$D^{-1}C \cdot B^{-1}A = \bar{\delta}^{-1}$$
    $$B^{-1}A \cdot A^{-1}C \cdot D^{-1}C \cdot (A^{-1}C)^{-1} = B^{-1}CD^{-1}A = A^{-1} \bar{\gamma} A$$
    and therefore by Corollary \ref{Conjugates Generate}, the triple generates $\mathrm{PSL}_2(\mathbb{F}_p).$

    Now, we prove injectivity. Let $(A_1,B_1,C_1,D_1),(A_2,B_2,C_2,D_2) \in X_{\gamma, \delta}^{(2)},$ and assume they have the same image in the $\mathrm{PSL}_2(\mathbb{F}_p)$-character variety. After choosing compatible lifts of these matrices to $\mathrm{SL}_2(\mathbb{F}_p),$ it follows from Lemma \ref{Triples up to conjugation} that there is a $\mathrm{GL}_2(\mathbb{F}_p)$-element conjugating one tuple to the other. Therefore, there is an $h \in \mathrm{PGL}_2(\mathbb{F}_p)$ such that $h (B_1^{-1}A_1, A_1^{-1}C_1, D_1^{-1}C_1) h^{-1} = (B_2^{-1}A_2, A_2^{-1}C_2, D_2^{-1}C_2)$

    Since $D_2^{-1}C_2\cdot B_2^{-1}A_2=\bar{\delta}^{-1}=D_1^{-1}C_1\cdot B_1^{-1}A_1$, it follows that $h \in C_{\mathrm{PGL}_2(\mathbb{F}_p)}(\bar{\delta})$. As was shown in the proof of Lemma \ref{four_eight_decompositions}, whenever $\gamma \in \Gamma$ is non-unipotent, the determinant map $C_{\mathrm{PGL}_2(\mathbb{F}_p)}(\bar{\gamma}) \overset{\det}{\to} \{\pm 1\}$ is surjective. Therefore, there is an element $g \in C_{\mathrm{PGL}_2(\mathbb{F}_p)}(\bar{\gamma})$ such that $\det(g) = \det(h).$ Replacing the tuple $(A_1, B_1, C_1, D_1)$ by $g(A_1, B_1, C_1, D_1)h^{-1}$ - we obtain an equivalent representation of the same point. Hence, we may assume without loss of generality that the two triples of matrices in the image are equal.

    Denote by $k \in \Gamma$ the element such that $A_2=kA_1$. It follows that $(A_2,B_2,C_2,D_2)=k(A_1,B_1,C_1,D_1)$. On the other hand, $A_1B_1^{-1}C_1D_1^{-1}=\gamma=A_2B_2^{-1}C_2D_2^{-1}$, and so $k \in C_\Gamma(\gamma)$. Therefore the points are equal in $X_{\gamma,\delta}^{(2)}.$

    \item We start by proving that the image is contained in this set. Let $(\bar{A}, \bar{B}, \bar{C}, \bar{D}) \in \tilde{X}_{\gamma, \delta}^{(2)}.$ Pick lifts of $\bar{A}, \bar{B}, \bar{C}, \bar{D}$ to $\mathrm{SL}_2(\mathbb{F}_p)$ such that $AB^{-1}CD^{-1} = \gamma, A^{-1}BC^{-1}D = \delta.$ Then, 
    $$\mathrm{tr}(M_1M_3)=\mathrm{tr}(B^{-1}AD^{-1}C)=\mathrm{tr}(D^{-1}CB^{-1}A)=\mathrm{tr}(\delta^{-1})=\mathrm{tr}(\delta),$$
    $$\mathrm{tr}(M_1M_2M_3M_2^{-1}) = \mathrm{tr}(B^{-1}AA^{-1}CD^{-1}CC^{-1}A)=\mathrm{tr}(B^{-1}CD^{-1}A)=\mathrm{tr}(AB^{-1}CD^{-1})=\mathrm{tr}(\gamma).$$

    Now, choose a triple of matrices $M_1,M_2,M_3 \in \mathrm{SL}_2(\mathbb{F}_p)$ such that $\mathrm{tr}(M_1M_3)=\mathrm{tr}(\delta)$ and $\mathrm{tr}(M_1M_2M_3M_2^{-1})=\mathrm{tr}(\gamma)$.
    Finding a quadruple $(A,B,C,D) \in \mathrm{SL}_2(\mathbb{F}_p)^4$ with the correct equalities is easy, simply choose $A = 1$, and then this gives $B = M_1^{-1},C=M_2,D=M_2M_3^{-1}$. It follows that $\mathrm{tr}(AB^{-1}CD^{-1})=\mathrm{tr}(\gamma)$, as well as $\mathrm{tr}(A^{-1}BC^{-1}D)=\mathrm{tr}(\delta)$.
    
    Since $\gamma,\delta$ are non-unipotent, it follows that there are $g,h\in \mathrm{GL}_2(\mathbb{F}_p)$ such that
    $$g AB^{-1}CD^{-1} g^{-1} = \gamma,\ hA^{-1}BC^{-1}Dh^{-1} = \delta.$$
    Furthermore, using once again that $\gamma, \delta$ are non-unipotent, it follows that both maps
    $$C_{\mathrm{PGL}_2(\mathbb{F}_p)}(\bar{\gamma}), C_{\mathrm{PGL}_2(\mathbb{F}_p)}(\bar{\delta}) \overset{\det}{\longrightarrow} \{\pm 1\}$$
    are surjective. Since $\gamma, \delta$ are not involutions, it follows that $\mathrm{tr}(\gamma), \mathrm{tr}(\delta) \neq 0,$ and therefore $\gamma, \lambda\gamma$ and $\delta, \lambda\delta$ are not conjugate for any $\lambda \neq 1.$ Hence, the natural maps $$C_{\mathrm{GL}_2(\mathbb{F}_p)}(\gamma) \to C_{\mathrm{PGL}_2(\mathbb{F}_p)}(\bar{\gamma}),\ C_{\mathrm{GL}_2(\mathbb{F}_p)}(\delta) \to C_{\mathrm{PGL}_2(\mathbb{F}_p)}(\bar{\delta})$$ are surjective.

    Since the conjugation actions of $\mathrm{GL}_2(\mathbb{F}_p), \mathrm{SL}_2(\mathbb{F}_p)$ factor through $\mathrm{PGL}_2(\mathbb{F}_p), \mathrm{PSL}_2(\mathbb{F}_p)$ respectively, it now follows that we may assume without loss of generality that $g, h \in \mathrm{SL}_2(\mathbb{F}_p).$ Indeed, if for example $\bar{g}$ is not in $\mathrm{PSL}_2(\mathbb{F}_p),$ then we may multiply $g$ on the right by an element of $C_{\mathrm{GL}_2(\mathbb{F}_p)}(\gamma)$ whose image in $\mathrm{PGL}_2(\mathbb{F}_p)$ is not in $\mathrm{PSL}_2(\mathbb{F}_p),$ thereby making this an element whose image is in $\mathrm{PSL}_2(\mathbb{F}_p).$ Since the conjugation action factors through $\mathrm{PSL}_2(\mathbb{F}_p)$, we may choose a lift of this element to $\mathrm{SL}_2(\mathbb{F}_p)$ with the corresponding conjugation relation.

    Now, replacing $(A, B, C, D)$  with $g(A, B, C, D)h^{-1}$, we get a quadruple in $X_{\gamma,\delta}^{(2)}$, with our given image $(M_1,M_2,M_3)$ in the $\mathrm{PSL}_2(\mathbb{F}_p)$-character variety.
    \end{enumerate}
\end{proof}
Note that we used the map $\varphi:(A,B,C,D) \mapsto (B^{-1}A, A^{-1}C, D^{-1}C)$ in the proposition as a choice of coordinates, however we could just as well have used the more natural map $(A,B,C,D) \mapsto (A^{-1}B, B^{-1}C, C^{-1}D)$.

After having proved this proposition, we see that our question is entirely one of a $4$-dimensional subvariety of the character variety of $F_3$. Moreover the proposition allows us to give explicit algebraic equations for the algebraic variety in question. We include the equations for the interested reader:
    $$p^2 - (ax+by+cz - abc)p + (a^2 + b^2 + c^2 + x^2 + y^2 + z^2 + xyz - abz - bcx - cay - 4) = 0.$$
    $$y = \mathrm{tr}(\delta)$$
    $$ac + bp - xz = \mathrm{tr}(\gamma) + \mathrm{tr}(\delta)$$
    Where
    $$a = \mathrm{tr}(B^{-1}A), b=\mathrm{tr}(A^{-1}C), c = \mathrm{tr}(D^{-1}C)$$
    $$x = \mathrm{tr}(A^{-1}CD^{-1}C), y = \mathrm{tr}(B^{-1}AD^{-1}C), z = \mathrm{tr}(B^{-1}C), p = \mathrm{tr}(B^{-1}CD^{-1}C)$$
    are the $7$-coordinate tuple of $(B^{-1}A, A^{-1}C, D^{-1}C)$ as in Lemma \ref{Magnus character varieties}.

Since the character variety in question is the $\mathrm{PSL}_2(\mathbb{F}_p)$-character variety, the set in question is the quotient of the set of $\mathbb{F}_p$-points of this algebraic variety by the three involutions corresponding to changing the signs of $M_1, M_2, M_3.$ However, in order to preserve the signs of $\mathrm{tr}(\gamma), \mathrm{tr}(\delta),$ the signs of $M_1, M_3$ may only be changed simultaneously. Hence, the set in question is the quotient of the set of $\mathbb{F}_p$-points of this algebraic variety by the group $(\mathbb{Z}/2)^2$ generated by the two involutions
$$(a, b, c, x, y, z, p) \mapsto (a, -b, c, -x, y, -z, -p)$$
$$(a, b, c, x, y, z, p) \mapsto (-a, b, -c, -x, y, -z, p)$$

A given first proper decomposition fixes the values of $a,c$, and a given second proper decomposition fixes the values of $b,p$.
    
Explicit formulae for the algebraic actions of $\sigma_1,\sigma_2,\sigma_3$ can be written using this description. We  include them here for the interested reader:
    $$\sigma_1: (a,b,c,x,y,z,p)\mapsto(a,ab-z,c,ax-p,y,b,x)$$
    $$\sigma_2: (a,b,c,x,y,z,p)\mapsto(az-b,a,cz-p,acz-ap-bc+x,y,z,c)$$    $$\sigma_3: (a,b,c,x,y,z,p)\mapsto(a,x,c,cx-b,y,p,cp-z)$$
\bibliographystyle{alpha}
\bibliography{bibli}

\vspace{0.5cm}

\noindent{\textsc{DPMMS, Centre for Mathematical Sciences, Wilberforce Road, Cambridge, CB3 0WB, UK}}

\noindent{\textit{Email address:} \texttt{liamhanany@gmail.com}}

\end{document}